\documentclass[11pt,reqno]{amsart}
\usepackage{amssymb,mathrsfs,graphicx,amsmath}
\usepackage{ifthen}
\usepackage{caption}
\usepackage{sidecap}
\usepackage{rotating}
\usepackage{mathtools}
\usepackage[foot]{amsaddr}
\usepackage{colortbl}
\definecolor{black}{rgb}{0.0, 0.0, 0.0}
\definecolor{red}{rgb}{1.0, 0.5, 0.5}
\provideboolean{shownotes} 
\setboolean{shownotes}{true} 
\newcommand{\margnote}[1]{
\ifthenelse{\boolean{shownotes}}%
{\marginpar{\raggedright\tiny\texttt{#1}}}%
{}%
}
\newcommand{\hole}[1]{
\ifthenelse{\boolean{shownotes}}%
{\begin{center} \fbox{ \rule {.25cm}{0cm} \rule[-.1cm]{0cm}{.4cm}
\parbox{.85\textwidth}{\begin{center} \texttt{#1}\end{center}} \rule
{.25cm}{0cm}}\end{center}} {} }

\graphicspath{{axel/}}

\numberwithin{equation}{section}

\newtheorem{theorem}{Theorem}[section]
\newtheorem{lemma}{Lemma}[section]

\newtheorem{proposition}{Proposition}[section]
\newtheorem{remark}{Remark}[section]

\newcommand{\R}{\mathbb R}
\newcommand{\Rd}{\mathbb R^n}

\newcommand{\mc}{\mathcal C}

\newcommand{\bq}{\begin{equation}}
\newcommand{\eq}{\end{equation}}

\newcommand{\lt}{\left}
\newcommand{\rt}{\right}

\newcommand{\pa}{\partial}
\newcommand{\dt}{h}

\def\charf {\mbox{{\text 1}\kern-.30em {\text l}}}
\def\w{{\bf w}}

\title{Pressureless Euler Alignment system with control}

\author[G. Albi]{Giacomo Albi$^{(1)}$}
\address[1]{Department of Computer Science, University of Verona,
	Str. Le Grazie 15, Verona, IT-37131, Italy}
\email{giacomo.albi@univr.it}
\author[Y.-P. Choi]{Young-Pil CHOI$^{(2)}$}
\address[2]{Department of Mathematics and Institute of Applied Mathematics, Inha University
	Incheon 402--751, Republic of Korea}
\email{ypchoi@inha.ac.kr}

\author[A.-S. H\"ack]{Axel-Stefan H\"ACK$^{(3)}$}
\address[3]{Department of Mathematics, IGPM, RWTH Aachen University,
	Templergraben 55, Aachen, DE-52062, Germany}
\email{	haeck@igpm.rtwh-aachen.de}

\begin{document}


\date{\today}



\begin{abstract}
We study a non-local hydrodynamic system with control. First we characterize the control dynamics as a sub-optimal approximation to the optimal control problem constrained to the evolution of the pressureless Euler alignment system. We then discuss the critical thresholds that leading to global regularity or finite-time blow-up of strong solutions in one and two dimensions. Finally we propose a finite volume scheme for numerical solutions of the controlled system. Several numerical simulations are shown to validate the theoretical and computational results of the paper.
\end{abstract}

\maketitle 

\section{Introduction}
Over the last decades there has been a vigorous development of literature in applied mathematics and physics describing collective behaviors of 
multi-agent systems,~\cite{GC04,CS}, towards modeling phenomena in biology, such as cell aggregation and motility,~\cite{CDFSTB03,be07},  coordinated animal motion, \cite{MR2507454,BS12,BCCCCGLOPPVZ09}, or coordinated human,~\cite{crpito11,MR2438215,bellomo2011SR}.

The standard viewpoint of these branches of mathematical modeling of multi-agent systems is that the dynamics are based on the {\it free} interaction of the agents, to describe their self-organization in terms of the formation of complex  macroscopic patterns. 
On the other hand, most recently, researchers started to investigate those systems from the perspective of constrained interactions, in particular to study control mechanisms capable of enforcing desired global behaviors. 

Hence, control problems in multi-agent system have been developed in several directions in the mathematical community. Indeed, the solution to the control of collective behaviors has been studied in the microscopic setting for example in \cite{BFK,CFPT,BS15}, where, however, direct numerical solution can be challenging due to the high-dimensionality, and non-linearities of such systems. 
Therefore there has been a large effort in the derivation of mesoscopic, and macroscopic approximation, which represents a first step towards the reduction of the problem complexity, and the development of new numerical strategies, \cite{PTa,CFTV10,DM16,MR2425606}. In particular optimal control problem  for kinetic equations has been studied in \cite{APZa,CPRT17,ACFK17,AHP,FS13}, and macroscopic models in \cite{CP13,HKK15}, moreover systems where multiple scales are coupled have been investigated in \cite{ABCK15,crpito11,BFRS}. 

Concerning that, several studies have been carried out to derive consistent kinetic approximation of microscopic optimal control problems, for example via mean-field equations \cite{FS13,FPR,BFRS}, or  Boltzmann-type models \cite{AHP,PTa,ACFK17}. On the other hand, consistent and rigorous derivation of macroscopic models for  microscopic, and mean-field control problems, are still missing for general, and it poses extremely challenging problems due to difficulties in determining asymptotic equilibrium states. However, several studies have been proposed for hydrodynamic models of self-organized system, coupled to an optimal control problem, where the most common applications are in traffic flow, or crowd motion models \cite{CPT14,HKK15,BRSV11}.

In the current work, we focus on Euler type system for the velocity alignment behaviors. This type of model has been recently investigated in several papers, \cite{CCTT16,CCP16,TT14}, and it is referred to {\em pressureless Euler alignment} system,
where the non-local interaction process is inspired by microscopic multi-agent systems such as Cucker-Smale and Motsch-Tadmor models, \cite{CS,MT13}. 
Thus, we consider the evolution of density $\rho(x)\in\R$, and velocity field $u(x)\in\R^n$, both defined in $\R^n$, and whose dynamics are described by following system:
\begin{align}\label{mainintro}
\begin{aligned}
&\pa_t \rho + \nabla_x \cdot (\rho u) = 0,\quad x \in \R^n, \quad t > 0,\cr
&\pa_t u + u\cdot \nabla_x u  =\int_{\R^n} \psi(x-y)(u(y) - u(x))\rho(y)\,dy + \phi, 
\end{aligned}
\end{align}
where the right-hand side of the velocity equations account for the non-local alignment force with a positive communication function $\psi$, and a control term $\phi\in\R^n$ whose action aims to enforce certain desired behavior.


Typically, for velocity alignment dynamics, controls are used to enforce global convergence towards a flocking state, where the velocity field exhibits an uniform direction $\bar u$, this is usaully to prevent situations where the communication rate $\psi$ is weak and the initial state is not well-prepared, \cite{CS,HaHaKim}. 
For these applications, we can define the control $\phi$ as a solution of an optimal control problem  defined by the constrained minimization of functionals of the following type
\begin{align}\label{funintro}
\begin{aligned}
\phi^* = \arg\min_{\phi} J(\phi;\rho,u):=\int_0^T\int_{\R^n}\lt(\ell(x,\rho,u)+\gamma|\phi|^2\rt)\ dxdt,
\end{aligned}
\end{align}
with $\ell(\cdot)$ a specific target cost, to be specified according to applications, where $\gamma|\phi|^2$ is a quadratic regularization term, penalizing the control action with a parameter $\gamma > 0$. In our study, we will focus mainly on enforcing a global desired velocity field requiring $\ell(x,\rho,u):=|u(x)-\bar u|^2$.

Macroscopic models, such as \eqref{mainintro}, also require additional investigations to ensure the existence of solutions. It is well known that solutions to the pressureless Euler type system may develop a singularity such as a $\delta$-shock no matter how smooth the initial data are. For that reason, it is natural to take into account the measure-valued solutions for the global regularity. However, our main system \eqref{mainintro} includes a nonlocal dissipation and this allows us to have global-in-time strong solutions under smallness assumptions on the initial data \cite{HKK14}, see also \cite{Choipre} for the isothermal Euler alignment system. It is also even obtained a sharp critical threshold for the system \eqref{mainintro} without control, i.e., $\phi \equiv 0$, that leading to global regularity or finite-time blow-up of strong solutions \cite{CCTT16} in one dimension. In particular, other interactions forces, attractive or repulsive forces are also considered in \cite{CCTT16}. For two dimensional case, the critical thresholds are investigated in \cite{TT14}. We refer to \cite{CCP16} for the recent survey on attractive-repulsive hydrodynamic models for collective behaviors.

The paper is structured as follows: in Section \ref{sec2} we discuss our modelling setting, and we derive the control term $\phi$ as an instantaneous feedback control from a particular optimal control problem. Section \ref{sec3} provides critical thresholds results for the pressureless Euler alignment system with control. We give a sharp estimate for the one-dimensional case, and for the two-dimensional case we show novel bounds for the global existence of solutions. Section \ref{sec4} is devoted to the discretization of model \eqref{mainintro} in one and two dimensions. We employ a finite-volume type scheme for both cases. Finally, several numerical tests validate our theoretical findings and explore further features of the model.

%
%
%
%
\section{Control of macroscopic alignment systems}\label{sec2}
We consider the controlled pressureless Euler alignment system, defined as follows
\begin{align}\label{main-eq}
\begin{aligned}
&\pa_t \rho + \nabla_x \cdot (\rho u) = 0,\quad x \in \R^n, \quad t > 0,\cr
&\pa_t u + u\cdot \nabla_x u  = \int_{\R^n} \psi(x-y)(u(y) - u(x))\rho(y)\,dy + \phi, 
\end{aligned}
\end{align}
with compactly supported initial density and uniformly bounded initial velocity:
\bq\label{ini-main-eq}
(\rho(x,t),u(x,t))|_{t=0} = (\rho_0(x), u_0(x)) \quad \mbox{for} \quad x \in \R^n.
\eq
Here $\bar u \in \R^n$ and $\psi \in (W^{1,\infty}\cap \mc^1)(\R^n)$ is a communication weight satisfying
\bq\label{condi-psi}
0 \leq \psi(x) = \psi(-x) \quad \mbox{and} \quad \psi(x) \leq \psi(y) \quad \mbox{for}  \quad |x| \geq |y|.
\eq
Note that the Cucker-Smale model's communication weight function satisfies the condition \eqref{condi-psi}. Finally, we want to define the control, $\phi\in \R^n$, in such way that it leads system \eqref{main-eq} towards a preferred direction. Hence, we consider the following functional
\begin{equation}
\min_{\phi} J(\phi;\rho,u) :=\int_{0}^T\int_{\Rd} (|u(x,s)-\bar{u}|^2 + \gamma|\phi(x,s)|^2) \ dx ds
\label{eq:MaOC}
\end{equation}
 which has to be minimized in order to find a control field $\phi(\cdot,t):\Rd\to \Rd$, such that the velocity of the system $u(x,t)$ is steered toward the reference $\bar u \in\Rd$, with 
a quadratic convex penalization of parameter $\gamma>0$. 

Following standard adjoint calculus for infinite dimensional system, e.g. \cite{Trozlch}, we introduce the functions $p,q\in L^2$, which act as Lagrangian multipliers associated respectively to the density $\rho$ and velocity $u$.  Hence, we can derive the adjoint system of equations for the optimal control problem associated to \eqref{eq:MaOC} and constrained to evolution of \eqref{main-eq}, which reads as follows
\begin{align}
\label{eq:Adj}
&
\begin{aligned}
&\partial_t p+u\cdot\nabla_x p = \int_{\Rd}\psi(x-y)(u(y)-u(x))q(y)\ dy\\
&\partial_t q  - (u\cdot\nabla_x) q - L[q,u]= \int_{\Rd}\psi(x-y)(\rho(y)q(x)-\rho(x)q(y))\ dy - 2(\bar u-u) \cr
&\hspace{4cm} +\rho\nabla_x p,
\end{aligned}
\end{align}
where the operator  $L$ is explicitly given by
\begin{align}
(L[q,u])_{k} = \sum_{j=1}^n\lt(q_k\pa_j u_j - q_j\pa_k u_j\rt),\quad k=1,\ldots,n.
\end{align}
 This system is complemented with zero value terminal conditions $p(x,T)=0$ and $q(x,T) = 0$, and the additional condition for the control
\begin{align}
&\phi(x,t) = \frac{1}{2\gamma}q(x,t).\label{eq:Trasv}
\end{align}
Thus solutions $(\rho^*,u^*,\phi^*)$ of the optimal  control problem have to satisfy the optimality conditions system defined by \eqref{main-eq}, \eqref{eq:Adj}, and \eqref{eq:Trasv}.
In general, the numerical solution of this system is computationally very heavy, in particular for large time predictions or high-dimensional applications. 

In what follows we will show a reduction technique, which approximates the solution of the full optimal control problem deriving a sub-optimal control, characterized as an instantaneous controller, \cite{AHP,Hinze05}.

\subsection{Instantaneous feedback control}
We approximate the optimal control problem \eqref{main-eq}--\eqref{eq:MaOC} by minimizing a discretized functional over a set of sequential time sub-intervals. This strategy can be interpreted as a model-predictive control applied to the pressureless Euler alignment system, \cite{MRRS,MPC11}.

Consider the semi-implicit discrete functional for \eqref{eq:MaOC}, restricted to the time frame $[t,t+\dt]$ with the time step parameter $\dt$, as follows:
\begin{equation}
\begin{aligned}
\label{eq:dMOC1}
J_\dt(\phi) &:=\int_{\Rd}  \lt(|u(x,t+\dt)-\bar u(x)|^2 + \gamma_h|\phi_h(x,t)|^2\rt) dx,
\end{aligned}
\end{equation}
where the penalization term  $\gamma_h$ will be scaled later according to discretization parameter $\dt$. The functional \eqref{eq:dMOC1} is constrained to the evolution of the forward dynamics for the mass and velocity:
\begin{equation}
\begin{aligned}
\label{eq:forward}
&\frac{\rho(x,t+h) -\rho(x,t)}{h} + \nabla_x\cdot (\rho u)= O(\dt),\\
&\frac{u(x,t+h) - u(x,t)}{h} + u\cdot\nabla_x u = \int_{\Rd} \psi(x-y)(u(y,t)-u(x,t))\rho(y,t)\,dy\cr
&\hspace{5cm} +\phi_h(x,t) + O(\dt).
\end{aligned}
\end{equation}
In order to compute a minimizer for the reduced functional \eqref{eq:dMOC1}, it is enough to compute the vanishing points of the jacobian $D_{\phi}J_\dt(\phi)$, i.e.,
\begin{equation}
0=D_{\phi_h}J_\dt := \lim_{\epsilon\to 0}\frac{J_\dt(\phi_h+\epsilon\eta_h)-J_\dt(\phi_h)}{\epsilon},
\end{equation}
where $\eta_h$ is an admissible variation of the control $\phi_h$ such that $\phi_h+\epsilon\eta_h\in L^2$.  
Thus we compute the variations of the discrete \eqref{eq:dMOC1} with respect to $\phi_h$, 
\begin{align*}
&D_{\phi_h} J_\dt(\phi_h) \cr
&\quad =  D_{\phi_h}\left(\int_{\Rd}  \lt(|u(x,t+\dt)-\bar u(x)|^2 + \gamma_h|\phi_h(x,t)|^2\rt)dx \right) \\
&\quad = 2\dt\int_{\Rd} (u(x,t+h)-\bar u(x))\cdot \eta(x,t)\,dx+2\gamma_h\int_{\Rd} \left(\phi_h(x,t)\cdot \eta_h(x,t)\right) dx\\
&\quad = -2\dt\int_{\Rd} (\bar u(x)-u(x,t))\cdot \eta(x,t) \,dx+2\gamma_h\int_{\Rd} \left(\phi_h(x,t)\cdot \eta_h(x,t)\right) dx\cr 
&\qquad + 2\dt^2\int_{\Rd}\lt( -u\cdot\nabla_x u + \int_{\Rd} \psi(x-y)(u(y,t)-u(x,t))\rho(y,t)\, dy\rt)\cdot \eta_h(x,t)\,dx\cr
&\qquad + 2\dt^2\int_{\Rd} \phi_h(x,t) \cdot \eta_h(x,t) dx
\end{align*}
where we substituted the expression for $u(x,t+\dt)$ and we neglected terms of order $O(\dt^2)$. Finally, we have
\begin{align*}
D_{\phi_h} J_\dt(\phi_h) = 2\int_{\Rd}\lt[(\gamma_h+\dt^2)\phi_h(x,t) -\dt(\bar u(x)- u(x)) +O(\dt^2)\rt]\cdot \eta_h(x,t)\, dx.
\end{align*}
where the term $O(h^2)$ includes the additional contribution from the discretization of the velocity term. Thus we can conclude that for every test function $\eta$ the optimal one step control is given by 
\begin{align}
\label{eq:dIC}
\phi_\dt(x,t)  =\frac{\dt}{\gamma_h +\dt^2}(\bar u(x)-u(x,t))+O\lt(\frac{\dt^2}{\gamma_h +\dt^2}\rt).
\end{align}
Hence, assuming the penalization parameter is scaled as $\gamma_h = \dt \gamma$, in the limit $\dt\to 0$ we retrieve the instantaneous control
\begin{align}
\label{eq:IC}
\phi(x,t)  =\frac{1}{\gamma}(\bar u(x)-u(x,t)),
\end{align}
which acts as a relaxation towards the desired velocity field $\bar u(x)$.

In what follows we will assume that the control $\phi$ in \eqref{main-eq} is defined as the instantaneous control \eqref{eq:IC}.

\begin{remark}
Note that another choice for a possible functional might require the minimization of the momenta, $\rho u$, rather than the full velocity field, $u$. This can be expressed by defining, for example, the functional 
\begin{align}
J(\phi;\rho,u) = \int_0^T\int_{\Rd}\left(|u(x,t)-\bar u(x)|^2+\gamma|\phi(x,t)|^2\right)\rho(x,t)\, dx dt,
\end{align}
where we are intereseted in minimizing the $L^2$ distance between $u$ and the reference $\bar u$, only on the density $\rho$.
Following the instantaneous control approach we can define a discrete functional $J_h(\phi)$ equivalent to \eqref{eq:dMOC1} and by imposing $D_{\phi}J_\dt=0$, we have that the instantaneous minimizers are characterized by the following relations
\begin{align}
\label{eq:dICrho}
\phi_\dt(x,t)  =\frac{\rho}{\gamma_h\rho - h^2(\rho -\nabla_x\cdot(\rho u))}(\bar u(x)-u(x,t))+O(h^2).
\end{align}
Thus, in the limit for $\dt\to0$ it restitutes the same control as in \eqref{eq:IC}, but whose existence is defined only on the support of the density $\rho(x,t)$.
\end{remark}

%
%
%
%
\section{Critical thresholds for the controlled system \eqref{main-eq}} \label{sec3}
In this section, we study the critical thresholds for pressureless Euler alignment system \eqref{main-eq}  with instantaneous control defined in \eqref{eq:IC} in one and two dimensions, which characterizes the initial configurations for the global-in-time regularity and finite-time blow-up of solutions. For the one dimensional case, by using the fact that our control is a type of linear damping, we follow the strategy proposed in \cite{CCTT16}, which gives a sharp critical threshold estimate. For the two dimensional case, inspired by \cite{TT14}, we use the large-time behavior estimate of solutions to control differences of the velocities for investigating the subcritical region. In both cases we consider uniform reference velocity $\bar u$.

Then we begin by providing some preliminary results for the system \eqref{main-eq}-\eqref{ini-main-eq}.
\begin{lemma} Let $(\rho,u)$ be a global strong solution to the system \eqref{main-eq}-\eqref{condi-psi}. Then we have
	\begin{align*}
	\begin{aligned}
	&\frac{d}{dt}\int_{\R^n} \rho(x,t)\,dx = 0, \quad \frac{d}{dt}\int_{\R^n} (\rho u)(x,t)\,dx = 0,\cr
	&\frac{d}{dt}\int_{\R^n} (\rho|u|^2)(x,t)\,dx + \int_{\R^n \times \R^n} \psi(x-y)|u(x,t) - u(y,t)|^2 \rho(x,t)\rho(y,t)\,dxdy\cr
	&\qquad = \frac2\gamma\int_{\R^n} (\bar u - u(x,t))\cdot u(x,t) \rho(x,t)\,dx,
	\end{aligned}
	\end{align*}
	for $t \geq 0$.
\end{lemma}
\begin{proof} A straightforward computation together with using the symmetry assumption of $\psi$ yields the desired results.
\end{proof}
We next provide support estimates of the strong solutions to the system \eqref{main-eq}. For this, we introduce extremal functions:
\[
S(t):= \sup\{ |x-y| : x,y \in \mbox{supp}(\rho(t))\}
\]
and
\[
V(t) := \sup\{ |u(x,t)-u(y,t)| : x,y \in \mbox{supp}(\rho(t))\}.
\]

\begin{lemma}\label{lem_lt} Let $T \in [0,\infty]$ and $(\rho,u)$ be a strong solution to the system \eqref{main-eq}-\eqref{condi-psi} in $[0,T]$. Then we have
	\begin{align}\label{sddi_ineq}
	\begin{aligned}
	\frac{d}{dt}S(t) &\leq V(t),\cr
	\frac{d}{dt}V(t) &\leq -\|\rho_0\|_{L^1}\psi(S(t))V(t) - \frac1\gamma V(t),
	\end{aligned}
	\end{align}
	for $t \in [0,T]$. In particular, this gives
	\bq\label{est-comp}
	S(t) \leq D \quad \mbox{and} \quad V(t) \leq V_0, 
	\eq
	for all $t \in [0,T]$, where $D$ is a positive constant given by
	\[
	D := S_0 + \gamma V_0.
	\]
\end{lemma}
\begin{proof}By using the almost same argument as in \cite{TT14}, we can obtain the differential inequalities \eqref{sddi_ineq}. From that, we get $V(t) \leq V_0 e^{-t/\gamma}$ and subsequently this yields 
	\[
	S(t) \leq S_0 + \int_0^t V(s)\,ds \leq S_0 + \gamma V_0(1 - e^{-t/\gamma}) \leq D.
	\]
\end{proof}

%
%
%
%

\subsection{One dimensional case}
In this section, we study the critical thresholds for the system \eqref{main-eq} in one dimension. 
\begin{align}\label{1d-main-1}
\begin{aligned}
&\pa_t \rho + \pa_x (\rho u) = 0,\quad x\in \R, \quad t > 0,\cr
&\pa_t u + u \pa_x u = \int_{\R} \psi(x-y)(u(y) - u(x)) \rho(y) \,dy + \frac1\gamma(\bar u - u).
\end{aligned}
\end{align}
Differentiating the momentum equation in $\eqref{1d-main-1}$ with respect to $x$ together with setting $v = \pa_x u$ gives
\begin{align}\label{eq_v}
\begin{aligned}
&\pa_t \rho + u\pa_x \rho = - \rho v,\cr
&\pa_t v + u\pa_x v + v^2 = - u \int_\R \pa_x \psi(x-y) \rho(y)dy - \int_\R \psi(x-y) \pa_t \rho(y) dy \cr
&\hspace{2.8cm} - v \int_\R \psi(x-y) \rho(y) dy - \frac1\gamma v,
\end{aligned}
\end{align}
where we used the symmetry assumption on $\psi$ to get
\[
\int_\R \pa_x \psi(x-y) (u(y) - u(x))\rho(y)dy = - u(x) \int_\R \pa_x \psi(x-y) \rho(y)dy - \int_\R \psi(x-y) \pa_t \rho(y) dy.
\]
We next define the characteristic flow $\eta(x,t)$ by
\[
\frac{d}{dt} \eta(x,t) = u(\eta(x,t),t), \quad \mbox{with} \quad \eta(x,0) = x.
\]
and rewrite the system \eqref{eq_v} along that characteristic flow as
\begin{align*}
\begin{aligned}
&\pa_t\rho(\eta(x,t),t) = - \rho(\eta(x,t),t) v(\eta(x,t),t),\cr
&\pa_t (v(\eta(x,t),t) + (\psi \star \rho)(\eta(x,t),t)) = - v^2(\eta(x,t),t) - v(\eta(x,t),t)(\psi \star \rho)(\eta(x,t),t) \cr
&\hspace{5.6cm}  - \frac1\gamma v(\eta(x,t),t).
\end{aligned}
\end{align*}
Set $d := v + \psi \star \rho + 1/\gamma$. Then we finally have the following differential equations of $(\rho,d)$:
\begin{align}\label{eq-prop1}
\begin{aligned}
& \rho^\prime = - \rho v,\cr
& d^\prime = - v\lt(v + \psi \star \rho + \frac1\gamma\rt) = - d\lt(d - \psi \star \rho - \frac1\gamma\rt).
\end{aligned}
\end{align}
From the above equations, we easily get the blow up estimates of solutions and the uniform boundedness of solutions in time. We provide the details of that in the proposition below even though it can be obtained by using almost same argument as in \cite{CCTT16}.
\begin{proposition}\label{prop1} Consider the system \eqref{eq-prop1}. Then we have
	\begin{itemize}
		\item[(i)] If $d_0 < 0$, then $d \to - \infty$ and $\rho \to +\infty$ in finite time.
		\item[(ii)] If $d_0 = 0$, then $d(t) = 0$ for all $t \geq 0$.
		\item[(iii)] If $d_0 > 0$, then $d(t) \to \psi \star \rho(t) + 1/\gamma$ as $t \to \infty$.
	\end{itemize}
\end{proposition}
\begin{proof} Set $\beta = d/\rho$, then we can easily find 
	\[
	\beta^\prime = \frac{d^\prime \rho - d \rho^\prime}{\rho^2} = \frac{1}{\rho^2}\lt(-v\lt(v + \psi \star \rho + \frac1\gamma\rt)\rho + dv\rho\rt) = 0.
	\]
	Thus we get $\beta(t) = \beta_0$ for all $t \geq 0$. Since the proof of (ii) is almost trivial, we only deal with (i) and (iii). 
	
	(i). It is clear to have that  if $d_0 < 0$, then $d(t) \leq 0$ for all $t \geq 0$. Then it follows from $\eqref{eq-prop1}_2$ that $d^\prime \leq -d^2$, and this gives
	\[
	d(t) \leq \frac{d_0}{t + d_0}.
	\]
	Hence $d(t)$ will be blow up until $t \leq -d_0$. Furthermore since $\rho(t) \beta_0 = d(t)$ with $\beta_0 < 0$, we also find $\rho \to + \infty$ until $t \leq -d_0$.
	
	(iii). Note that if $d(t) \in (0, \psi \star \rho(t)+1/\gamma)$, then $d^\prime(t) > 0$ thus $d(t)$ is increasing up to $\psi \star \rho(t) + 1/\gamma$. On the other hand, if $d(t) > \psi \star \rho(t) + 1/\gamma$, then $d(t)$ is decreasing up to $\psi \star \rho(t) + 1/\gamma$.
\end{proof}
\begin{remark}
	We do not need to have the large-time behavior estimate obtained in Lemma \ref{lem_lt}.
\end{remark}
As a direct consequence of Proposition \ref{prop1}, we have the following complete description of the critical thresholds for the system \eqref{1d-main-1}.
\begin{theorem}\label{thm1} Consider the one-dimensional pressureless Euler equations with the alignment force for velocities \eqref{1d-main-1} and the control. Then we have
	\begin{itemize}
		\item (Supercritical region) If there exists an $x$ such that $\pa_x u_0(x) < -\psi \star \rho_0(x) - 1/\gamma $, then the solution blows up in a finite time.
		\item (Subcritical region) If $\pa_x u_0(x) \geq -\psi \star \rho_0(x) - 1/\gamma$ for all $x \in \R$, then the system has a global strong solution, $(\rho,u) \in \mc(\R_+;L^\infty(\R)) \times \mc(\R_+; \dot{W}^{1,\infty})$.
	\end{itemize}
\end{theorem}
\begin{proof} Since $\|\psi \star \rho\|_{L^\infty} \leq \|\psi\|_{L^\infty}\|\rho_0\|_{L^1} < \infty$, all results are obtained from Proposition \ref{prop1}.
\end{proof}
\begin{remark}It follows from Theorem \ref{thm1} that the parameter $\gamma$, the strength of control, plays an important role in determining the regions for sup- and sub-critical regions; if we consider a strong control, i.e., $\gamma$ is large enough, then the subcritical region is larger. On the other hand, if $\gamma$ is small enough, i.e, a weak control is taken into account, then the subcritical region is smaller.
\end{remark}

%
%
%
\subsection{Two dimensional case}\label{2dtheo}
In this part, we study the critical thresholds for the system \eqref{main-eq} in two dimensions. For this, we employ the refined estimate in \cite{CCTT16} together with the idea used in \cite{TT14}, which provides slightly better results compared to \cite{TT14}.

More precisely, we are concerned with the system:
\begin{align}\label{eqn-2d}
\begin{aligned}
&\pa_t \rho + \nabla_x \cdot (\rho u) = 0,\quad x \in \R^2, \quad t >0,\cr
&\pa_t u + u \cdot \nabla_x u = \int_{\R^2} \psi(x-y)(u(y) - u(x)) \rho(y)\, dy +\frac1\gamma(\bar u - u).
\end{aligned}
\end{align}
Taking $\nabla_x$ to the system $\eqref{eqn-2d}_2$ and setting 
\[
M := \left( \begin{array}{cc}
\pa_1u^1 & \pa_2 u^1 \\
\pa_1 u^2 & \pa_2 u^2 \\
\end{array} \right),
\]
yield that $M$ satisfies
\[
\pa_t M + (u\cdot \nabla_x )M + M^2 + \lt(\frac1\gamma + (\psi \star \rho)\rt) M = F,
\]
where $F$ is given by $F = (F_{ij})_{1 \leq i,j \leq 2}$ with
\[
F_{ij} = \int_{\R^2} \pa_j\psi(x-y)(u^i(y) - u^i(x))\rho(y)\,dy.
\]
Set
\[
v:= \pa_1 u^1 + \pa_2 u^2, \quad q := \pa_1 u^1 - \pa_2 u^2, \quad r := \pa_2 u^1, \quad \mbox{and} \quad s:=\pa_1 u^2.
\]
Then we find that
\begin{align*}
\begin{aligned}
v^\prime + \frac{v^2 + \eta^2}{2} &= \int_{\R^2} \nabla_x \psi(x-y) \cdot (\rho(y)u(y)) dy - u \cdot \int_{\R^2} \nabla_x \psi(x-y) \rho(y) dy \cr
&\quad - v(\psi \star \rho) - \frac1\gamma v\cr
&= -(\psi \star \rho)^\prime - v(\psi \star \rho) - \frac1\gamma v.
\end{aligned}
\end{align*}
Here $\eta$ is a spectral gap $\eta = \lambda_2 - \lambda_1$ where $\lambda_i,i=1,2$ are two eigenvalues of the matrix $\nabla_x u$ given by
\[
\lambda_1 = \frac12\lt(v - \sqrt{\Gamma} \rt), \quad \lambda_2 = \frac12\lt( v + \sqrt{\Gamma}\rt), \quad \Gamma = q^2 + 4rs.
\]
Thus if we again set $d := v + (\psi \star \rho) + 1/\gamma$, then $d$ satisfies
\bq\label{eqn-d2}
d^\prime = -\frac{v^2 + \eta^2}{2} - v(\psi \star \rho) -\frac1\gamma v= -\frac12\lt(d - \psi \star \rho - \frac1\gamma\rt)\lt(d + \psi \star \rho + \frac1\gamma\rt) - \frac{\eta^2}{2}.
\eq
We also obtain that
\begin{align}\label{eqn-qrs}
\begin{aligned}
&q^\prime + qd = \int_{\R^2} \lt(\pa_1 \psi(x-y)(u^1(y) - u^1(x)) - \pa_2 \psi(x-y)(u^2(y) - u^2(x))\rt)\rho(y) dy \cr
&\hspace{1cm} =: Q_{11} - Q_{22},\cr
&r^\prime + rd = \int_{\R^2} \pa_2\psi(x-y)(u^1(y) - u^1(x))\rho(y) dy =: Q_{12},\cr
&s^\prime + sd = \int_{\R^2} \pa_1\psi(x-y)(u^2(y) - u^2(x))\rho(y) dy =: Q_{21}.\cr
\end{aligned}
\end{align}
We next set $Q := \max_{1\leq i,j \leq 2} |Q_{ij}|$. Then it easily follows from the fact \eqref{est-comp} that $Q \leq \|\nabla_x \psi\|_{L^{\infty}}
\|\rho_0\|_{L^1} V_0 =:\tilde{Q}$.
We now loosely follow the argument in \cite[Lemma 4.4]{TT14} to have the uniform boundedness of $\eta(t)$ in time.
\begin{lemma}\label{lem:2d}Let $(q,r,s)$ be the strong solutions to the system \eqref{eqn-qrs} with the initial data $(q_0,r_0,s_0)$. Suppose that 
	\[
	d(t) \geq \frac{2\tilde Q}{\max\{|q_0|, 2|r_0|,2|s_0|\}} \quad \mbox{for} \quad t \in [0,T].
	\]
	Then $(q,r,s)$ remain bounded 
	\[
	\max\{|q(t)|, 2|r(t)|,2|s(t)|\} \leq \max\{|q_0|, 2|r_0|,2|s_0|\} \quad \mbox{for} \quad t \in [0,T].
	\]
	Thus the spectral gap $|\eta(t)| \leq \sqrt{2}\max\{|q_0|, 2|r_0|,2|s_0|\}$ is bounded for $t \in [0,T]$.
\end{lemma}
Then we are now ready to provide the initial configurations for the global regularity of solutions.
\begin{proposition}\label{prop-d2-2}Suppose that 
	\[
	\psi(D)\|\rho_0\|_{L^1}  + \frac1\gamma \geq \frac{1}{\max\{|q_0|, 2|r_0|,2|s_0|\}}\sqrt{4\tilde Q^2 + 2\lt(\max\{|q_0|, 2|r_0|,2|s_0|\}\rt)^4},
	\]
	where $D$ is the positive constant appeared in Lemma \ref{lem_lt}. If $d_0 \geq 2\tilde{Q}/\max\{|q_0|, 2|r_0|,2|s_0|\}$, then we have
	\[
	d(t) \geq 2\tilde{Q}/\max\{|q_0|, 2|r_0|,2|s_0|\} \quad \mbox{for} \quad t \in [0,T],
	\]
	and subsequently we have
	\[
	\max\{|q(t)|, 2|r(t)|,2|s(t)|\} \leq \max\{|q_0|, 2|r_0|,2|s_0|\} \quad \mbox{for} \quad t \in [0,T].
	\]
\end{proposition}
\begin{proof} For notational simplicity, we denote by $\delta_0 := \max\{|q_0|, 2|r_0|,2|s_0|\}$. Note that by Lemma \ref{lem:2d} $|\eta(t)| \leq \sqrt 2\delta_0$ as long as $d(t) \geq 2\tilde Q/\delta_0$. On the other hand, if $|\eta(t)| \leq \sqrt 2\delta_0$, then it follows from \eqref{eqn-d2} that
	$$\begin{aligned}
	d' &= -\frac12\lt(d^2 - \lt(\psi \star \rho + \frac1\gamma\rt)^2 \rt) -\frac{\eta^2}{2} \geq -\frac12 d^2 + \frac12 \lt(\psi(D)\|\rho_0\|_{L^1}  + \frac1\gamma\rt)^2 - \delta_0^2 \cr
	&= -\frac12(d - c_0)(d+c_0),
	\end{aligned}$$
	where $c_0$ is a positive constant given by
	\[
	c_0 = \sqrt{\lt(\psi(D)\|\rho_0\|_{L^1}  + \frac1\gamma\rt)^2 - 2\delta_0^2}.
	\]
	On the other hand, we find $c_0 \geq 2\tilde Q/\delta_0$ and this gives that $d'(t) \geq 0$ if $d_0 \in [2\tilde Q/\delta_0, c_0)$. This completes the proof.
\end{proof}
We next investigate the initial configurations for finite-time blow-up of solutions.
\begin{proposition}\label{prop-d2-1} Consider the system \eqref{eqn-d2}. Suppose that 
	\[
	\min\{r_0, s_0 \} \geq \frac{\tilde Q}{(\|\psi\|_{L^\infty}\|\rho_0\|_{L^1}  + 1/\gamma)} \quad \mbox{and} \quad d_0 < -\|\psi\|_{L^\infty}\|\rho_0\|_{L^1}  -\frac1\gamma.
	\]
	Then $d \to -\infty$ in finite time.
\end{proposition}
\begin{proof} Set
	\[
	\mathcal{T} := \lt\{ t \in [0,\infty): r(\tau) > 0 \quad \mbox{and} \quad s(\tau) > 0 \quad \mbox{for} \quad \tau \in [0,t) \rt\}.
	\]
	It is clear from the continuity of the functions $r(t)$, $s(t)$, and the assumption that $\mathcal{T} \neq \emptyset$. Thus we can set $\mathcal{T}^\infty := \sup \mathcal{T}$. We then claim 
	\[
	\mathcal{T}^\infty \geq -2\lt(d_0 + \|\psi\|_{L^\infty}\|\rho_0\|_{L^1}  + \frac1\gamma\rt).
	\]
	Let us assume that the above claim is not correct, i.e., $\mathcal{T}^\infty < -2\lt(d_0 + \|\psi\|_{L^\infty}\|\rho_0\|_{L^1}  + 1/\gamma\rt)$, and we have either 
	\bq\label{eqn_claim}
	\lim_{t \to \mathcal{T}^\infty -} r(t) = 0 \quad \mbox{or} \quad \lim_{t \to \mathcal{T}^\infty -} s(t) = 0.
	\eq
	On the other hand, it follows from \eqref{eqn-d2} that for $t < \mathcal{T}^\infty$
	\[
	d ' \leq -\frac12\lt(d - \|\psi\|_{L^\infty}\|\rho_0\|_{L^1} - \frac1\gamma\rt)\lt(d + \|\psi\|_{L^\infty}\|\rho_0\|_{L^1}  + \frac1\gamma\rt).
	\]
	This and using the continuity argument with the assumption $d_0 < -\|\psi\|_{L^\infty}\|\rho_0\|_{L^1}  -1/\gamma$ imply
	\[
	d' \leq -\frac12\lt(d + \|\psi\|_{L^\infty}\|\rho_0\|_{L^1} +\frac1\gamma \rt)^2 \quad \mbox{for} \quad t \in \lt[0,\mathcal{T}^\infty\rt).
	\]
	Thus we obtain
	\bq\label{eqn_dblow}
	d \leq -\|\psi\|_{L^\infty}\|\rho_0\|_{L^1}  - \frac1\gamma + \frac{1}{(\tilde d_0)^{-1} + t/2} \quad \mbox{for} \quad t \in \lt[0,\mathcal{T}^\infty\rt),
	\eq
	where $\tilde d_0 = d_0 + \|\psi\|_{L^\infty}\|\rho_0\|_{L^1}  + 1/\gamma < 0$. This together with \eqref{eqn-qrs} yields
	\bq\label{eq_r}
	r ' = -rd + Q_{12} \geq r\lt(\|\psi\|_{L^\infty}\|\rho_0\|_{L^1}  + \frac1\gamma - \frac{1}{(\tilde d_0)^{-1} + t/2} \rt) - \tilde Q.
	\eq
	Since
	$$\begin{aligned}
	&\exp\lt(-\int_0^t \lt(\|\psi\|_{L^\infty}\|\rho_0\|_{L^1}  + \frac1\gamma - \frac{1}{(\tilde d_0)^{-1} + s/2} \rt)\,ds \rt) \cr
	&\quad = e^{-(\|\psi\|_{L^\infty}\|\rho_0\|_{L^1} +1/\gamma)t} e^{2\lt(\ln|(\tilde d_0)^{-1} + t/2| - \ln|(\tilde d_0)^{-1}|\rt)}\cr
	&\quad =e^{-(\|\psi\|_{L^\infty}\|\rho_0\|_{L^1} +1/\gamma)t}\lt((\tilde d_0)^{-1} + t/2\rt)^2 \tilde d_0^2\cr
	&\quad \leq e^{-(\|\psi\|_{L^\infty}\|\rho_0\|_{L^1} +1/\gamma)t},
	\end{aligned}$$
	we get from \eqref{eq_r} that 
	\[
	\lt(r e^{-(\|\psi\|_{L^\infty}\|\rho_0\|_{L^1} +1/\gamma)t}\lt((\tilde d_0)^{-1} + t/2\rt)^2 \tilde d_0^2\rt)^\prime \geq -\tilde Q  e^{-(\|\psi\|_{L^\infty}\|\rho_0\|_{L^1} +1/\gamma)t}.
	\]
	Integrating the above differential inequality in time, we find
	$$\begin{aligned}
	&r e^{-(\|\psi\|_{L^\infty}\|\rho_0\|_{L^1} +1/\gamma)t}\lt((\tilde d_0)^{-1} + t/2\rt)^2 \tilde d_0^2 \cr
	&\qquad \geq r_0 + \frac{\tilde Q}{\|\psi\|_{L^\infty}\|\rho_0\|_{L^1}  + 1/\gamma}\lt(e^{-(\|\psi\|_{L^\infty}\|\rho_0\|_{L^1} +1/\gamma)t} - 1 \rt).
	\end{aligned}$$
	We now take the limit $t \to \mathcal{T}^\infty-$ to the above inequality to obtain
	$$\begin{aligned}
	0 &= \lim_{t \to \mathcal{T}^\infty -} r e^{-(\|\psi\|_{L^\infty}\|\rho_0\|_{L^1} +1/\gamma)t}\lt((\tilde d_0)^{-1} + t/2\rt)^2 \tilde d_0^2 \cr
	&\geq r_0 + \frac{\tilde Q}{\|\psi\|_{L^\infty}\|\rho_0\|_{L^1}  + 1/\gamma}\lt(e^{-(\|\psi\|_{L^\infty}\|\rho_0\|_{L^1} +1/\gamma)\mathcal{T}^\infty} - 1 \rt) \cr
	&\geq \frac{\tilde Q}{\|\psi\|_{L^\infty}\|\rho_0\|_{L^1}  + 1/\gamma}e^{-(\|\psi\|_{L^\infty}\|\rho_0\|_{L^1} +1/\gamma)\mathcal{T}^\infty}\cr
	&>0.
	\end{aligned}$$
	Similarly, by symmetry, we can show that $\lim_{t \to \mathcal{T}^\infty-} s(t) \geq 0$. This is a contradiction to \eqref{eqn_claim}, Thus we have $r(t) \geq 0$ and $s(t) \geq 0$ for $t \leq -2(d_0 + \|\psi\|_{L^\infty}\|\rho_0\|_{L^1}  + 1/\gamma)$, and this further yields from \eqref{eqn_dblow} that $d(t) \to -\infty$ until $t \leq -2(d_0 + \|\psi\|_{L^\infty}\|\rho_0\|_{L^1}  + 1/\gamma)$.
\end{proof}
Summarizing the above discussion, we conclude the following theorem.
\begin{theorem}\label{thm2dblowup} Consider the system \eqref{eqn-d2}. If the initial configurations satisfy the following 
	\[
	\psi(D)\|\rho_0\|_{L^1}  + \frac1\gamma \geq \frac{\sqrt{4 \|\nabla_x \psi\|_{L^{\infty}}^2\|\rho_0\|_{L^1} ^2V_0^2 + 2\lt(\max\{|\pa_1 u^1_0 - \pa_2 u^2_0|, 2|\pa_2 u^1_0|,2|\pa_1 u^2_0|\}\rt)^4}}{\max\{|\pa_1 u^1_0 - \pa_2 u^2_0|, 2|\pa_2 u^1_0|,2|\pa_1 u^2_0|\}},
	\]
	for all $x \in supp(\rho_0)$, where $D$ is the positive constant appeared in Lemma \ref{lem_lt} and 
	\[
	d_0 \geq \frac{2\|\nabla_x \psi\|_{L^{\infty}}\|\rho_0\|_{L^1} V_0}{\max\{|\pa_1 u^1_0 - \pa_2 u^2_0|, 2|\pa_2 u^1_0|,2|\pa_1 u^2_0|\}},
	\]
	then $\nabla_x u(x,t)$ remains uniformly bounded for all $(x,t) \in supp(\rho(t))$. On the other hand, if there exists an $x$ such that
	\[
	\min\{\pa_2 u^1_0(x), \pa_1 u^2_0(x) \} \geq \frac{\|\nabla_x \psi\|_{L^{\infty}}\|\rho_0\|_{L^1} V_0}{(\|\psi\|_{L^\infty}\|\rho_0\|_{L^1}  + 1/\gamma)} \quad \mbox{and} \quad d_0(x) < -\|\psi\|_{L^\infty}\|\rho_0\|_{L^1}  -\frac1\gamma,
	\]
	then there is a finite-time blow-up of solutions such that $\inf_{x \in supp(\rho(t))} \nabla_x \cdot u(x,t) \to -\infty$ in finite time.
\end{theorem}

%
%
%
%

\section{Numerical experiments}\label{sec4}
In this section we will present  numerical solutions to the pressureless Euler alignment system \eqref{main-eq} in one and two space dimensions. In order to discretize such system we will employ a finite volume approach based on Kurganov-Tadmor scheme, \cite{KT02}; among the several works on numerical solutions for pressureless Euler type system we refer to \cite{Bouchut94,CKR07,YDS13} for further insights. 
\subsection{Numerical Scheme}\label{scheme}
We introduce the following notation 
\begin{subequations}
	\begin{align}
	&\w = \begin{pmatrix} \rho \\ u \end{pmatrix},\qquad F(\w) = \begin{pmatrix}  u \\ u  \end{pmatrix}  \otimes \begin{pmatrix} \rho  \\  \frac{   u}{2} \end{pmatrix}   ,\qquad \Phi =\begin{pmatrix} 0\\\phi\end{pmatrix}, \label{conslaw}\\ 
	&G(\w) = \left(0,\int \psi(x-y)(u(y)-u(x))\rho(y)\ dy\right)^T.\label{convo}
	\end{align}
\end{subequations}
Note that $u$ and $\phi$ consist of scalar quantities in the one-dimensional case and the tensor product in \eqref{conslaw} becomes the  inner  product. 

Employing \eqref{conslaw} and \eqref{convo} we can rewrite \eqref{main-eq} as follows
\begin{align}\label{consform}
\pa_t \w  + \nabla_x \cdot F(\w) = G(\w)+\Phi.
\end{align}
Thus \eqref{consform} renders a (one- or two-dimensional) conservation law with the source term $ G(\w)+\Phi$.\par
In two spatial dimensions we get from \eqref{conslaw} that
\begin{align}\label{consform2d}
\pa_t \w  + \pa_x \cdot F_1(\w) +  \pa_y \cdot F_2(\w) = G(\w)+\Phi,
\end{align}
where $y$ is the second space dimension
$$ 
F_1=\begin{pmatrix} \rho u_1 \\[4pt] \frac{u_1^2}{2} \\[4pt] \frac{u_1 u_2}{2} \end{pmatrix}, \qquad  F_2=\displaystyle\begin{pmatrix} \rho u_2 \\[4pt] \frac{u_1 u_2}{2} \\[4pt] \frac{u_2^2}{2}  \end{pmatrix} \quad \mbox{with} \quad u = \begin{pmatrix}  u_1 \\ u_2  \end{pmatrix}.
$$

We will use a {\it Finite Volume} scheme to calculate an approximated solution of the macroscopic system \eqref{main-eq}. The discretization will consist of uniform cells $\mathcal{C}_{i,j}$ with its center located at $(x_i,y_j)$, the $x$-diameter of $\Delta x>0$ and the $y$-diameter of $\Delta y>0$.
A possible semi-discrete scheme for the two-dimensional case reads
\begin{align}\label{eq:macro_2d_disc}
\begin{aligned}
\frac{d}{dt} \w_{i,j}  &=  -\frac{1}{\Delta x}\left[F^*_{1\: i+\frac{1}{2},j}(t)-F^*_{1\: i-\frac{1}{2},j}(t)\right]-\frac{1}{\Delta y}\left[F^*_{2\: i,j+\frac{1}{2}}(t)-F^*_{2\: i,j-\frac{1}{2}}(t)\right]\cr
&+ G^*_{i,j}(t)+\Phi^*_{i,j}(t),
\end{aligned}
\end{align}
where the numerical fluxe $F_{1\: i\pm\frac{1}{2}, j}^*(t)$ is the one in $x$-direction of the cell $i,j$ through the interfaces located at $x_{i\pm\frac{1}{2}}$ and analog for the $y$-direction fluxes $F_{2\: i, j\pm\frac{1}{2}}^*(t)$. As the scheme to evaluate the numerical fluxes $F_{1\: i\pm\frac{1}{2}, j}^*(t)$ and $F_{2\: i, j\pm\frac{1}{2}}^*(t)$ we will use the {\it (genuinely multidimensional) Kurganov-Tadmor Central Scheme} as suggested in \cite{KT02}. We therefore have 
\begin{align}
F^*_{1\: i-\frac{1}{2},j} = & a^{+}_{i-\frac{1}{2},j}\frac{  F\left(\w^{EN}_{1\: i-1,j}\right) +   F\left(\w^{ES}_{1\:i-1,j}\right) }{2  \left[a^{+}_{i-\frac{1}{2},j} - a^{-}_{i-\frac{1}{2},j} \right]} - a^{-}_{i-\frac{1}{2},j} \frac{  F\left(\w^{WN}_{1\: i,j}\right) +   F\left(\w^{WS}_{1\: i,j}\right) }{2  \left[a^{+}_{i-\frac{1}{2},j} - a^{-}_{i-\frac{1}{2},j} \right]} \nonumber \\
&\ + \frac{a^{+}_{i-\frac{1}{2},j} a^{-}_{i-\frac{1}{2},j}}{2  \left[a^{+}_{i-\frac{1}{2},j} - a^{-}_{i-\frac{1}{2},j} \right]}\left(   \w^{WN}_{i,j} +\w^{WS}_{i,j} -  \w^{EN}_{i-1,j} -\w^{ES}_{i-1,j} \right) \label{2d_numFlux1}
\shortintertext{and}
F^*_{2\: i,j-\frac{1}{2}} =& b^{+}_{i,j-\frac{1}{2}}\frac{  F\left(\w^{WN}_{2\: i,j-1}\right) +   F\left(\w^{EN}_{2\: i,j-1}\right) }{2  \left[b^{+}_{i,j-\frac{1}{2}} - b^{-}_{i,j-\frac{1}{2}} \right]} - b^{-}_{i,j-\frac{1}{2}} \frac{  F\left(\w^{WS}_{2\: i,j}\right) +   F\left(\w^{ES}_{2\: i,j}\right) }{2  \left[b^{+}_{i,j-\frac{1}{2}} - b^{-}_{i,j-\frac{1}{2}} \right]} \nonumber \\
&\  + \frac{b^{+}_{i,j-\frac{1}{2}} b^{-}_{i,j-\frac{1}{2}}}{ 2  \left[b^{+}_{i,j-\frac{1}{2}} - b^{-}_{i,j-\frac{1}{2}} \right]}\left(  \w^{WS}_{i,j} + \w^{ES}_{i,j}  -  \w^{WN}_{i-1,j} - \w^{EN}_{i-1,j}  \right).\label{2d_numFlux2}
\end{align}
as second order-accurate approximations. Here $a^\pm$ is the largest (lowest, respectively) characteristic speed at the denoted interface of $F_1$, and $b^\pm$ is defined similarly with respect to the $y$-flux $F_2$. We denote by $\w^{EN},\w^{WN},\w^{ES},\w^{WS}$ the interpolation of the linear reconstruction of the cell-vertex values, generated by the slope reconstruction gathered by a MUSCL scheme. For remaining open questions on those numerical fluxes we again refer to  \cite{KT02}.
The spatial second order approximation of the numerical fluxes associated to $G(\w)$ and $\Phi$ comes for free since we do finite volume. We employ a midpoint-rule here. However, we cannot simply merge the numerical fluxes of the source term to a classical finite volume scheme, as indicated in \eqref{eq:macro_2d_disc}. This is due to the fact that for small control parameters $\gamma$ the control term $\Phi$ will render the associated terms to be stiff. For that reason, we rewrite \eqref{eq:macro_2d_disc} in the following way
\begin{align}\label{nonstiffstiff}
\pa_t \w = \mathcal{F}(\w) + \frac{1}{\gamma} \mathcal{R}(\w),
\end{align}
where we collect the non-stiff terms by 
\begin{align}\label{nonstiff}
\mathcal{F}(\w) =  -\pa_xF_1(\w)-\pa_yF_2(\w) + G(\w)
\end{align}
and
\begin{align}\label{stiff}
\mathcal{R}(\w) = \gamma \Phi = (0,\gamma \phi_{u_1},\gamma \phi_{u_2})^T =(0,\bar{{u}}_1-{u_1},\bar{{u}}_2-{u_2})^T.
\end{align}
From this decomposition, it is trivial to see that we can apply an IMEX Runge-Kutta scheme as presented in \cite{PR05}.\\
For the following numerical experiments we implemented a second order strong-stability-preserving (SSP) scheme, with the associated Butcher tableaus (implicit and explicit integrators) as displayed in Figure \ref{Butcher}.\par
\begin{figure}
	\centering
	\begin{tabular}[l]{c|cc}
		0 & 0 & 0 \\
		1 & 1 & 0\\
		\hline
		& $\frac{1}{2}$ & $\frac{1}{2}$
	\end{tabular}
	\qquad
	\begin{tabular}[l]{c|cc}
		$\beta$ & $\beta$ & 0 \\
		$1-\beta$ & $1-2\beta$ & $\beta$\\
		\hline
		& $\frac{1}{2}$ & $\frac{1}{2}$
	\end{tabular}
	\qquad $\beta=1-\dfrac{1}{\sqrt{2}}$
	\caption{Butcher tableaus of second order SSP IMEX Runge-Kutta scheme.\newline Left: Explicit integrator of $\mathcal{F}$. Right: Implicit integrator of $\frac{1}{\gamma}\mathcal{R}$. } \label{Butcher}
\end{figure}
By plugging together the Kurganov-Tadmor Central scheme and the second order SSP IMEX for the time integrations, we end up to have established a globally second order scheme that solves the two-dimensional case of \eqref{main-eq}.
\begin{remark}[One-dimensional scheme]
	For the one-dimensional case we do not have a flux in $y$-direction. Hence, we set $F_2$ zero (as well as $u_2\equiv 0$). In doing so, we already almost recovered the corresponding one-dimensional scheme. Note that the time integration is formulated independently form the space-dimension. \par
	The only left task at hand is the numerical flux $F_1^*$. In \eqref{2d_numFlux1}, we use the trapezoid rule (in $y$-direction) along the cell's interface located at $x_{i\pm \frac{1}{2}}$ to evaluate the flux across those boundaries. 
	An alternative numerical flux at interface $\left(i-\frac{1}{2},j\right)$, introduced in \cite{KT02} as well reads as follows:
	\begin{align}\label{1d_num_Flux}
	F_{i-\frac{1}{2},j} = \frac{ a^{+}_{i-\frac{1}{2},j} F\left(\w^E_{i-1,j}\right) -  a^{-}_{i-\frac{1}{2},j} F\left(\w^W_{i,j}\right) }    {a^{+}_{i-\frac{1}{2},j} - a^{-}_{i-\frac{1}{2},j} } + \frac{a^{+}_{i-\frac{1}{2},j} a^{-}_{i-\frac{1}{2},j}}{a^{+}_{i-\frac{1}{2},j} - a^{-}_{i-\frac{1}{2},j}}\left(  \w^W_{i,j}-\w^E_{i-1,j}  \right) ,
	\end{align}
	where $\w^E$ and $\w^W$ are the (one-dimensional) linear reconstructions of the interface values of $\w$. Again, we refer to \cite{KT02} for more details. 
\end{remark}
\subsection{Numerical experiments in one dimension.}\label{1dnum}
We implemented the 1D case of \eqref{main-eq} following the scheme introduced in Section \ref{scheme}.
We set, as computational domain, $x\in [-L,L]$ with $L=1$, and the time frame is set to $[0,T_f]$, $T_f=3$.  The 
initial data is defined as follows,
\begin{align}\label{ID}
\rho_0(x) = \chi_{\left[-\frac{2L}{3},-\frac{L}{6}\right]}(x) + \chi_{\left[\frac{L}{6},\frac{2L}{3}\right]}(x),
\qquad u_0(x) = -\sin\left(\frac{\pi x}{L}\right),
\end{align}
and periodic boundary conditions. We consider a communication function, $\psi$ in \eqref{condi-psi}, of the Cucker-Smale type,
\bq\label{CS}
\psi(x) = \frac{1}{(\zeta + \|x\|)^\beta},
\eq
with $\zeta=1$ and $\beta=10$. The spatial discretization is set, such that we have an uniform grid with $\Delta x = 0.01$, and as the flux limiter we chose \texttt{minmod}. For the IMEX scheme with $\frac{\Delta t}{\lambda\Delta x}=CFL\leq1$, we set the CFL--condition to $CFL=0.95$.\par

In Figure \ref{fig:0} we present the uncontrolled case ($\phi(x,t)\equiv 0$) to gain an understanding of the varying influence of the control.

\begin{figure}[t]
	\centering
	\begin{tabular}{@{}c@{\hspace{1mm}}c@{\hspace{1mm}}c@{\hspace{1mm}}c@{}}
		\hline
		\\
		\includegraphics[width=0.34\textwidth]{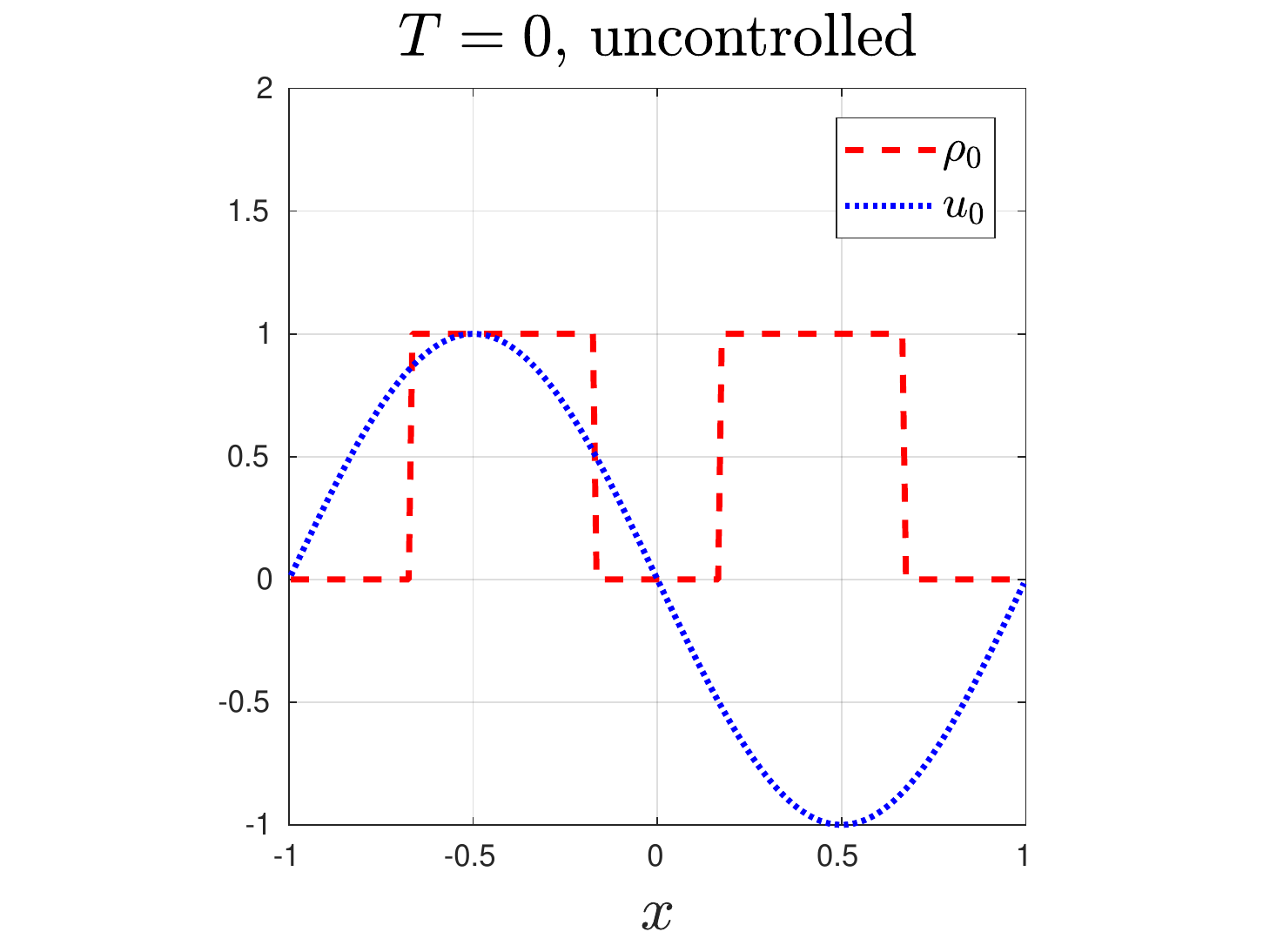}\hfill
		\includegraphics[width=0.34\textwidth]{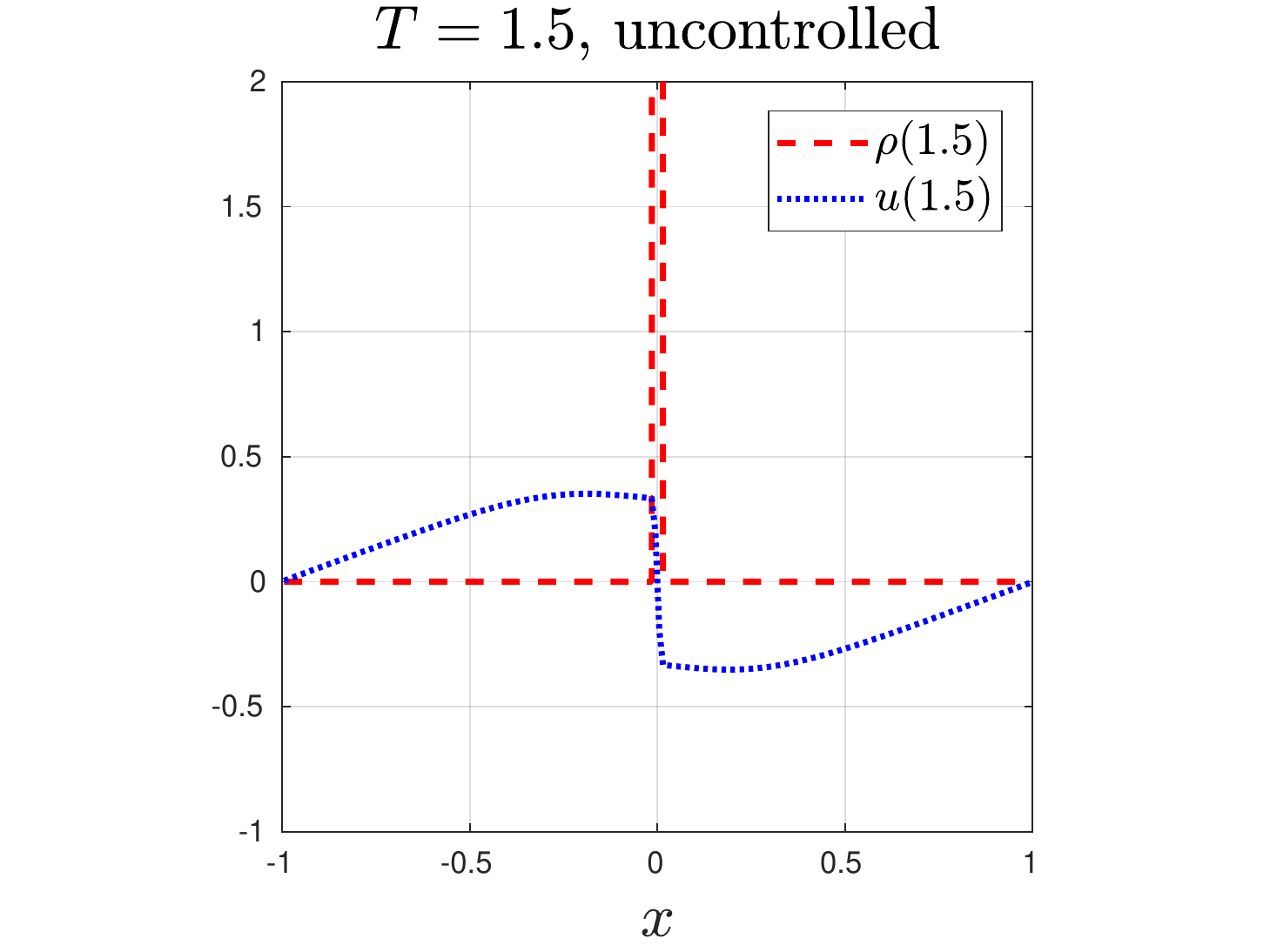}\hfill
		\includegraphics[width=0.34\textwidth]{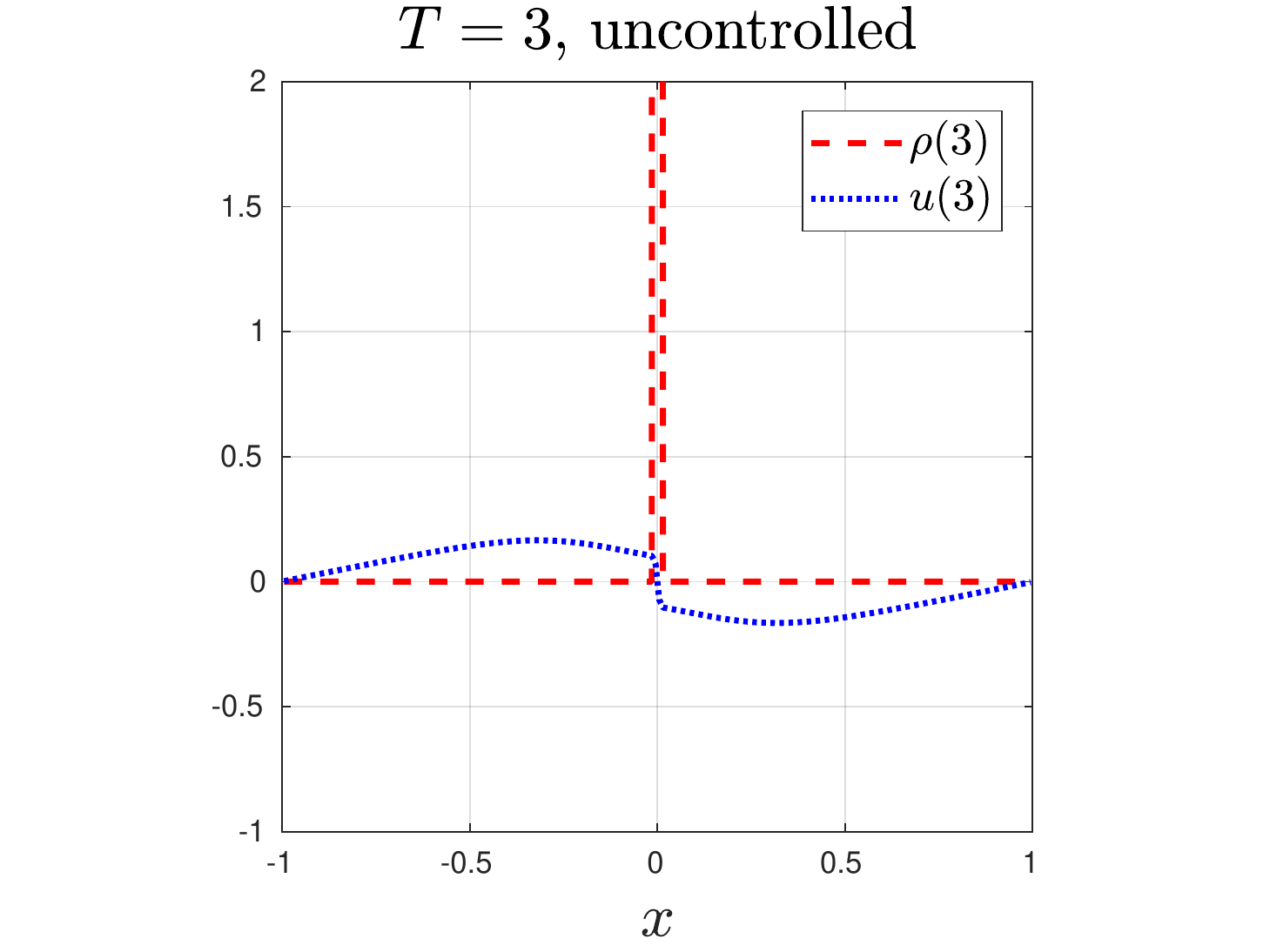}\\
		\hline
	\end{tabular}
	\caption{{ (Test 1D Uncontrolled case)}. Evolution from left side to right side of system \eqref{main-eq} without control in the time interval $[0,3]$.}\label{fig:0}
\end{figure}

\paragraph{Test 1D: Homogeneous desired state.}
For the controlled cases we consider first a  homogeneous case desired state, such that
\bq\label{homocont}
\phi(x,t) = \frac{1}{\gamma}(\bar u - u(x,t))
\eq
for varying values of the parameter $\gamma \in\{0.1,1,10\}$, and different desired states $\bar u$.

In Figure \ref{fig:1} we apply the same initial condition as above \eqref{ID} and instantaneous control \eqref{homocont} with desired homogeneous velocity $\bar{u}\equiv0.$
\begin{figure}[t]
	\centering
	\begin{tabular}{@{}c@{\hspace{1mm}}c@{\hspace{1mm}}c@{\hspace{1mm}}c@{}}
		\hline
		
		\includegraphics[width=0.34\textwidth]{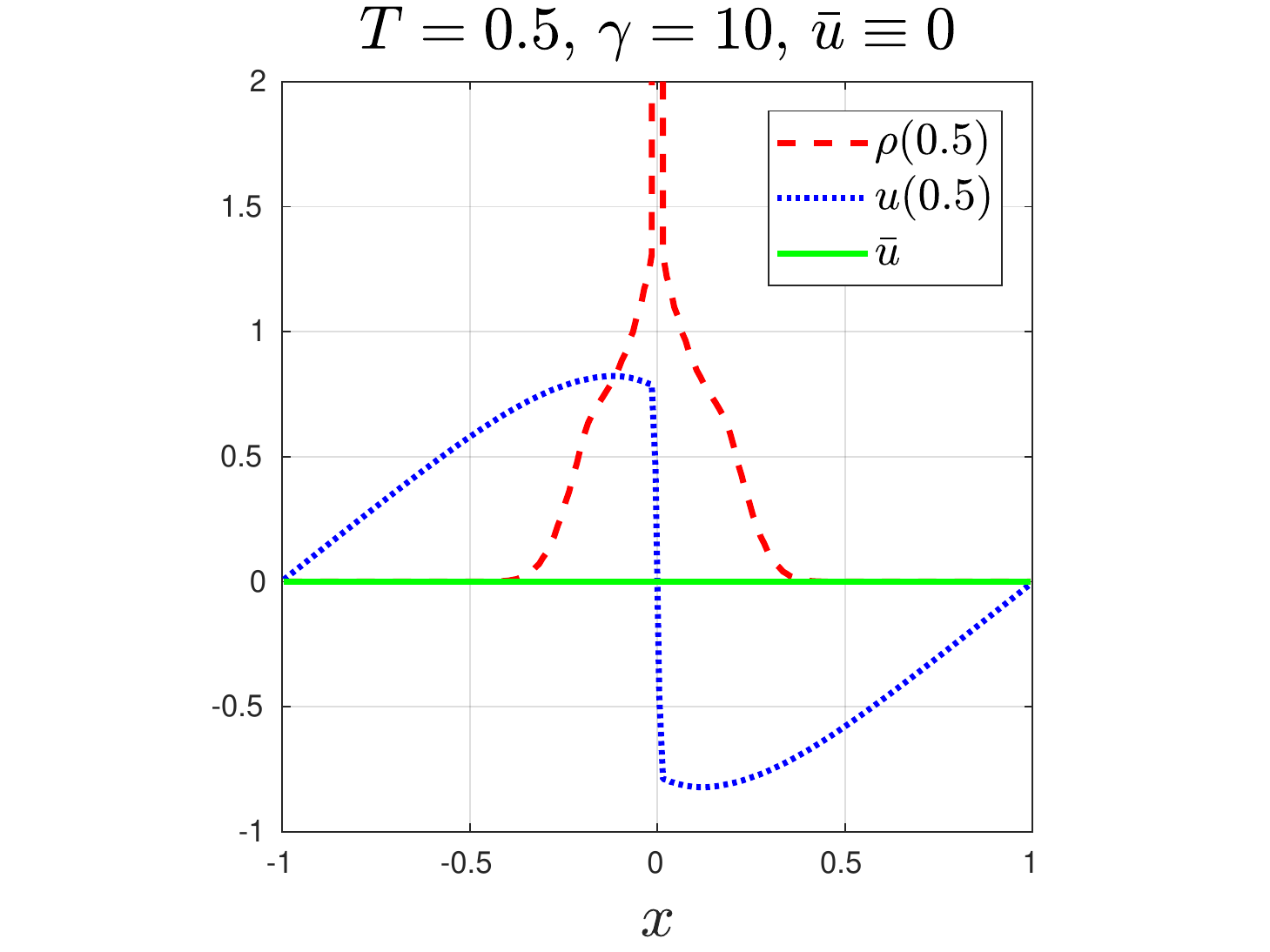}\hfill
		\includegraphics[width=0.34\textwidth]{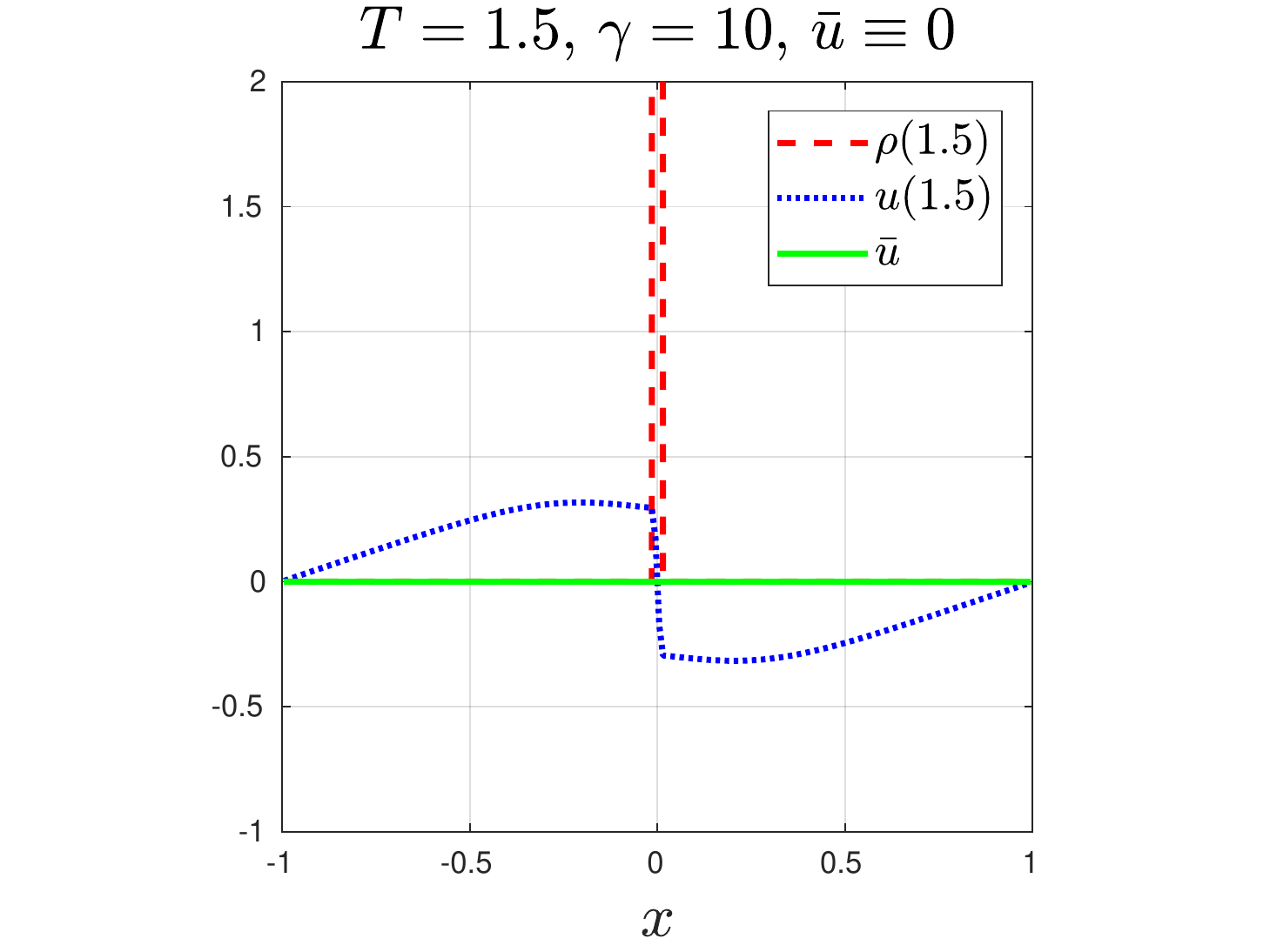}\hfill
		\includegraphics[width=0.34\textwidth]{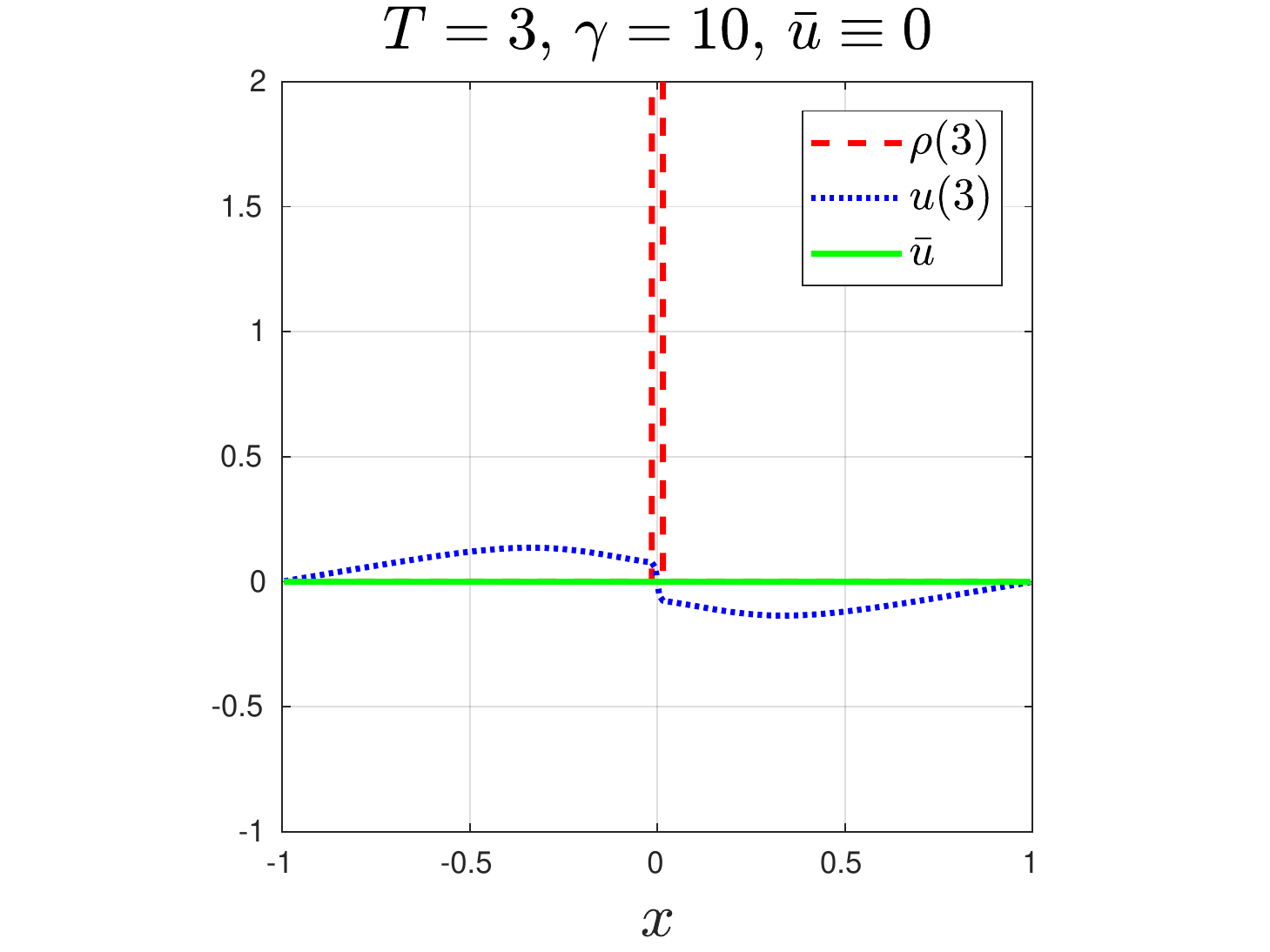}\\
		\hline\\
		\hline\\
		\includegraphics[width=0.34\textwidth]{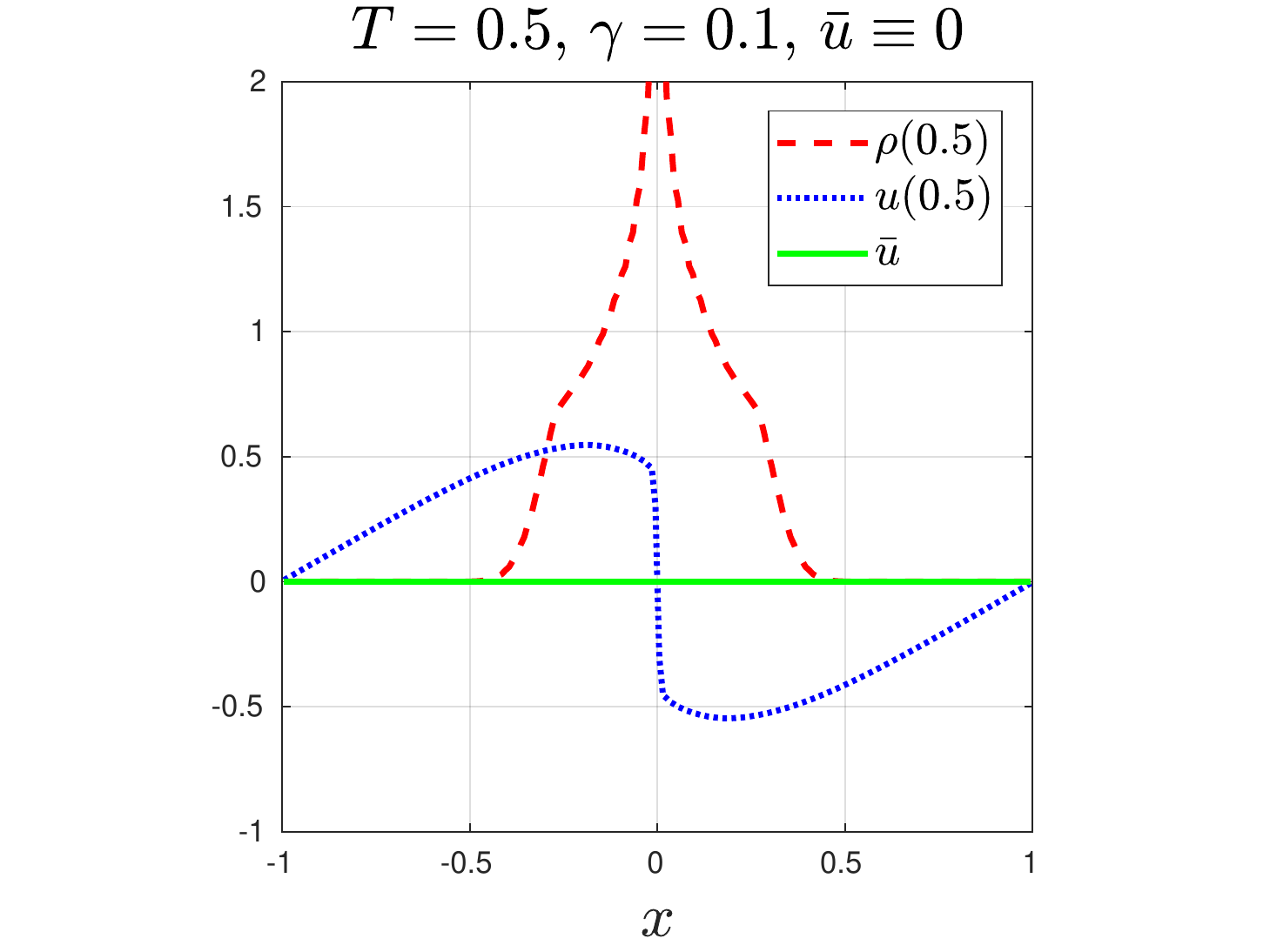}\hfill
		\includegraphics[width=0.34\textwidth]{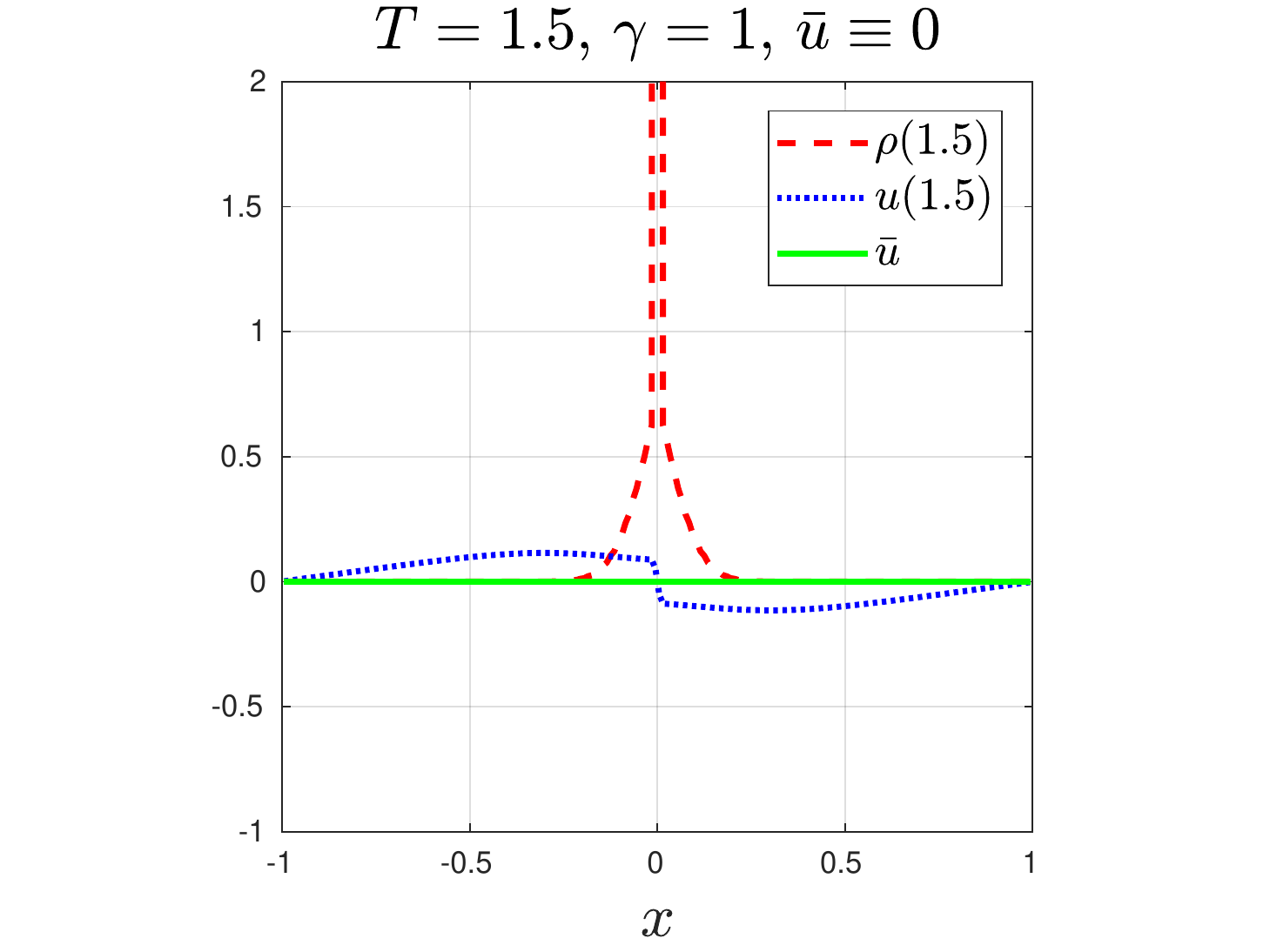}\hfill
		\includegraphics[width=0.34\textwidth]{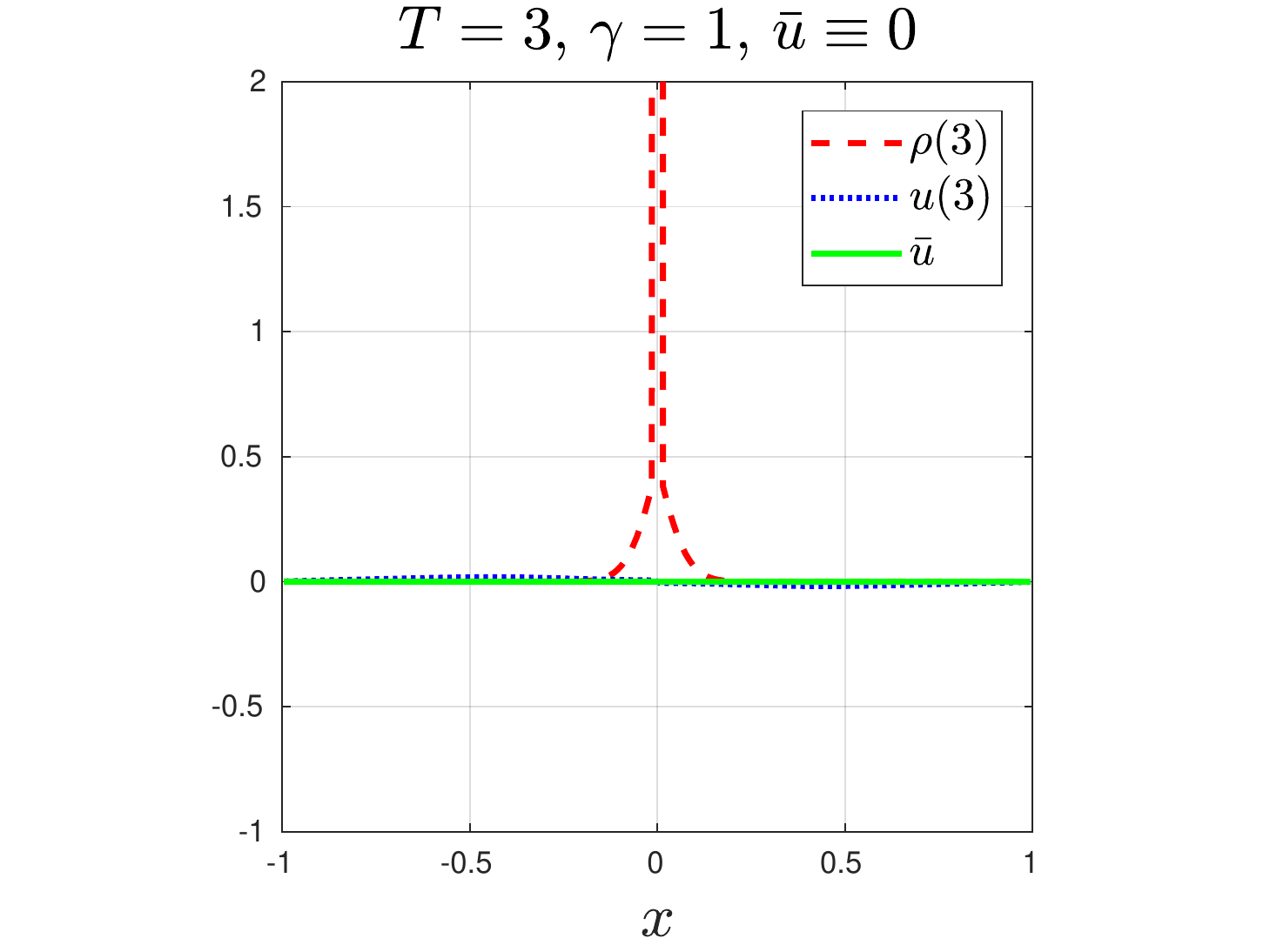}\\
		\hline\\
		\hline\\
		\includegraphics[width=0.34\textwidth]{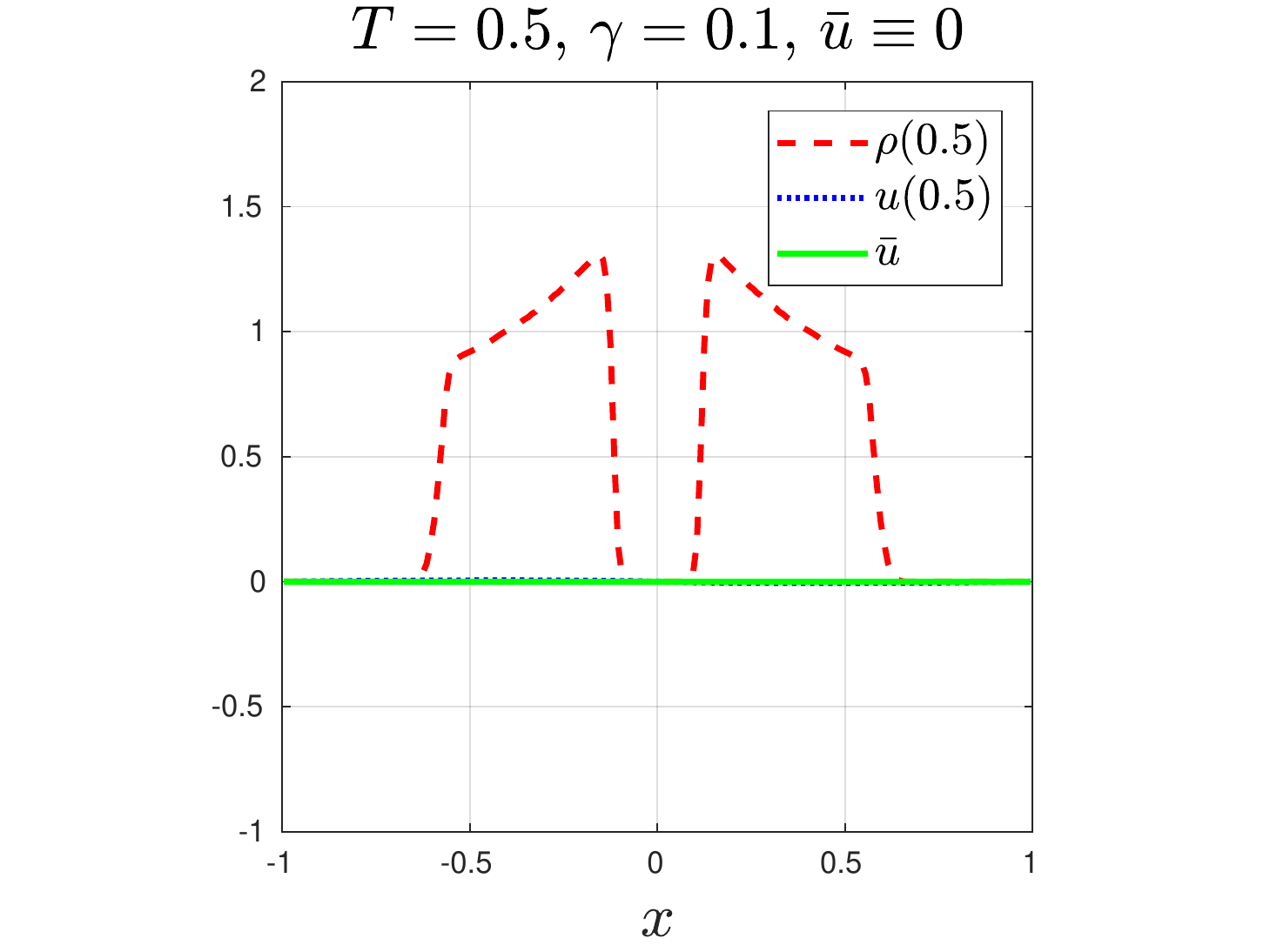}\hfill
		\includegraphics[width=0.34\textwidth]{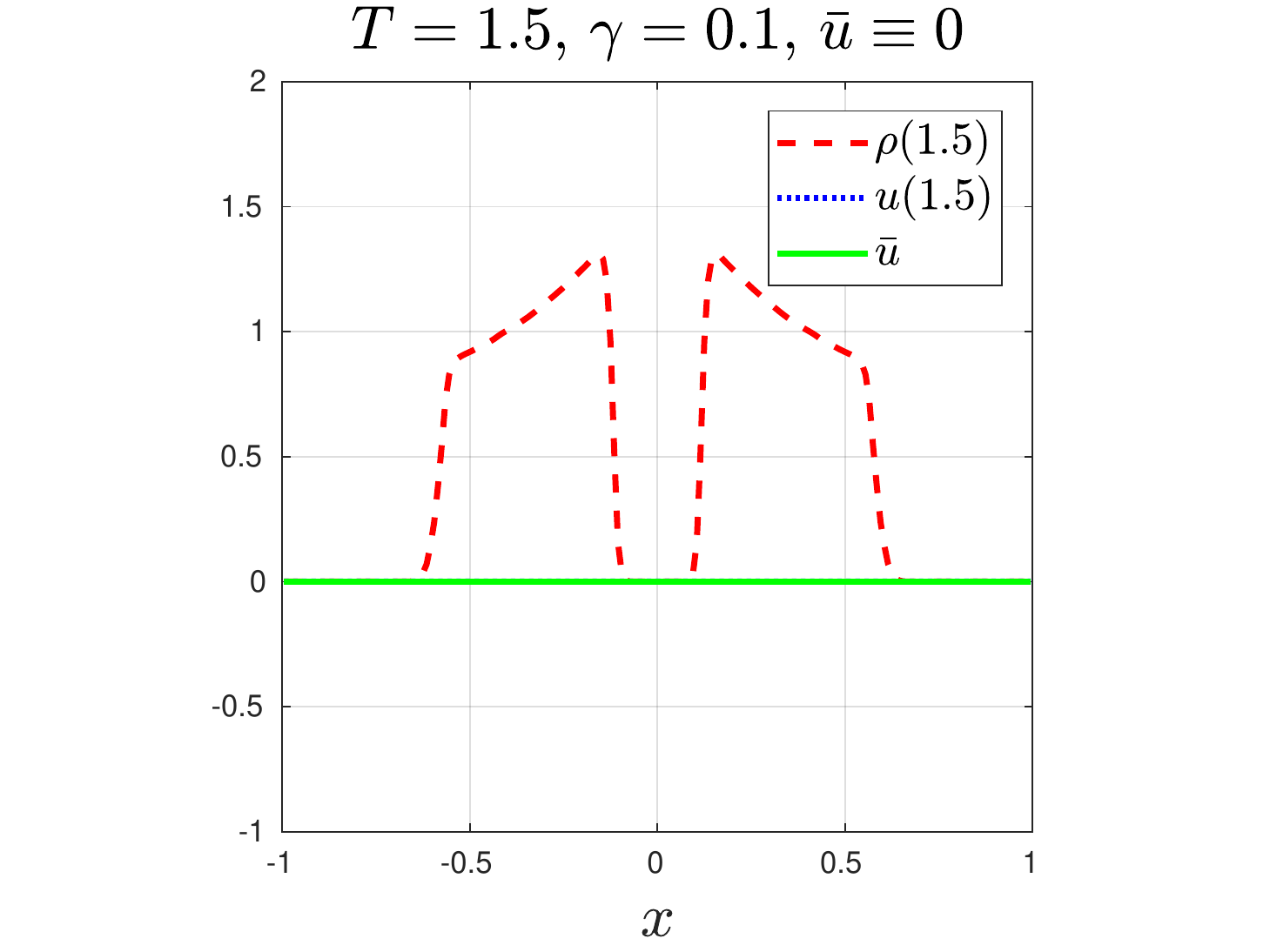}\hfill
		\includegraphics[width=0.34\textwidth]{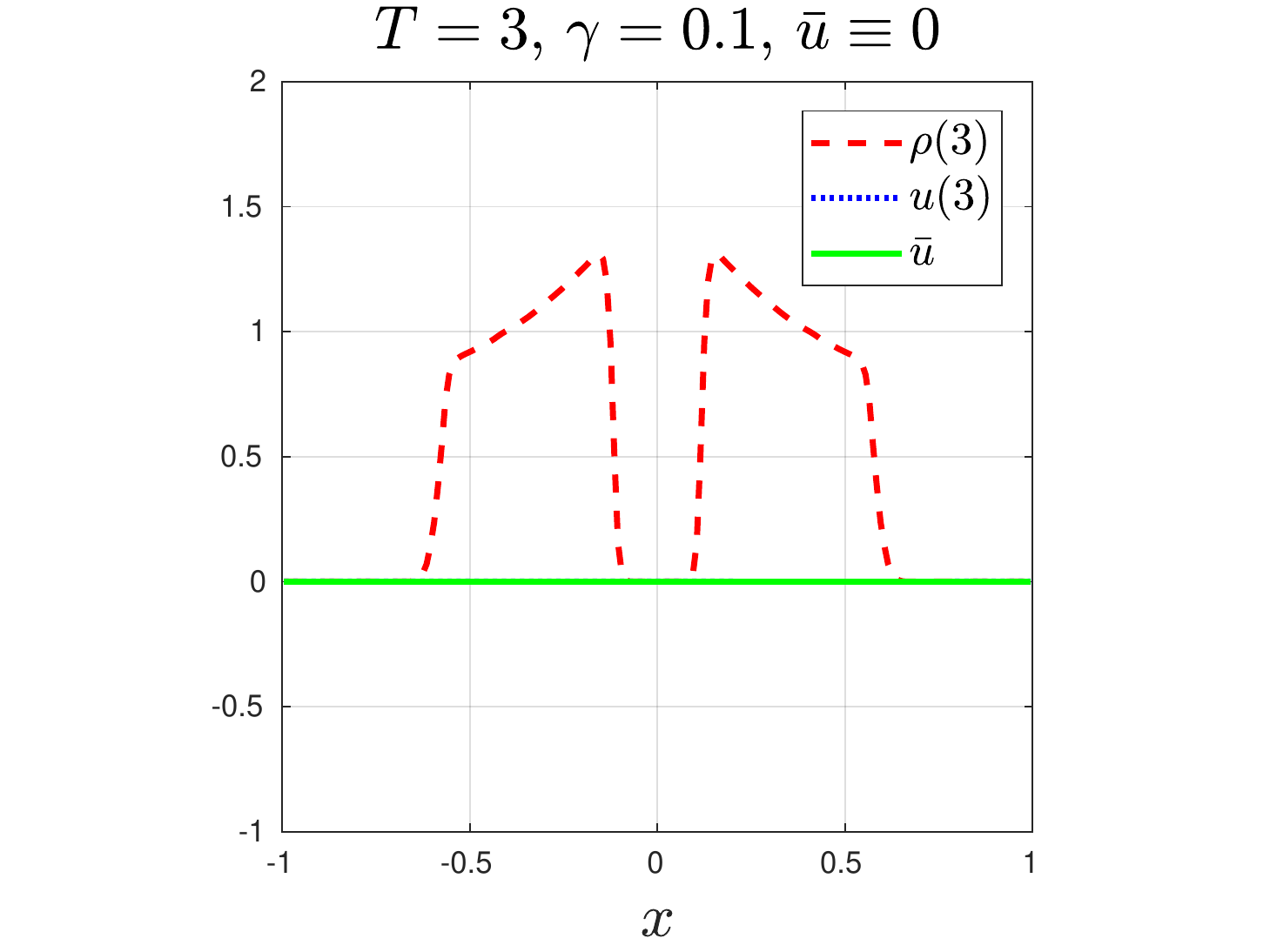}\\
		\hline
	\end{tabular}
	\caption{{( Test 1D: Homogeneous desired state).} From top to bottom evolution of system \eqref{main-eq} with $\gamma\in\{10,1,0.1\}$, and $\bar{u}\equiv0$.}\label{fig:1}
\end{figure}\

If we compare Figure \ref{fig:0} and Figure \ref{fig:1}, we see that in case of the control parameter $\gamma=10$, the control has almost no effect. For $\gamma=1$, however, we observe a significant quicker diminishing of the velocity ($u$ approaches $\bar{u}$) than the alignment of the uncontrolled case in Figure \ref{fig:0}. This yields to a not completed clustering of mass at the origin. For $\gamma=0.1$, we observe that the dynamics of the Cucker-Smale come to a halt almost instantly. Hence, the control is enforced stronger than in the other two cases. 

In Figure \ref{fig:2} we display the same setting as in Figure \ref{fig:1}, but with desired velocity $\bar{u}\equiv0.5$. Differently from the previous cases we observe that once the desidered velocity is reached the mass $\rho$ is transported in the right side direction.
\begin{figure}[t]
	\centering
	\begin{tabular}{@{}c@{\hspace{1mm}}c@{\hspace{1mm}}c@{\hspace{1mm}}c@{}}
		\hline
		
		\includegraphics[width=0.34\textwidth]{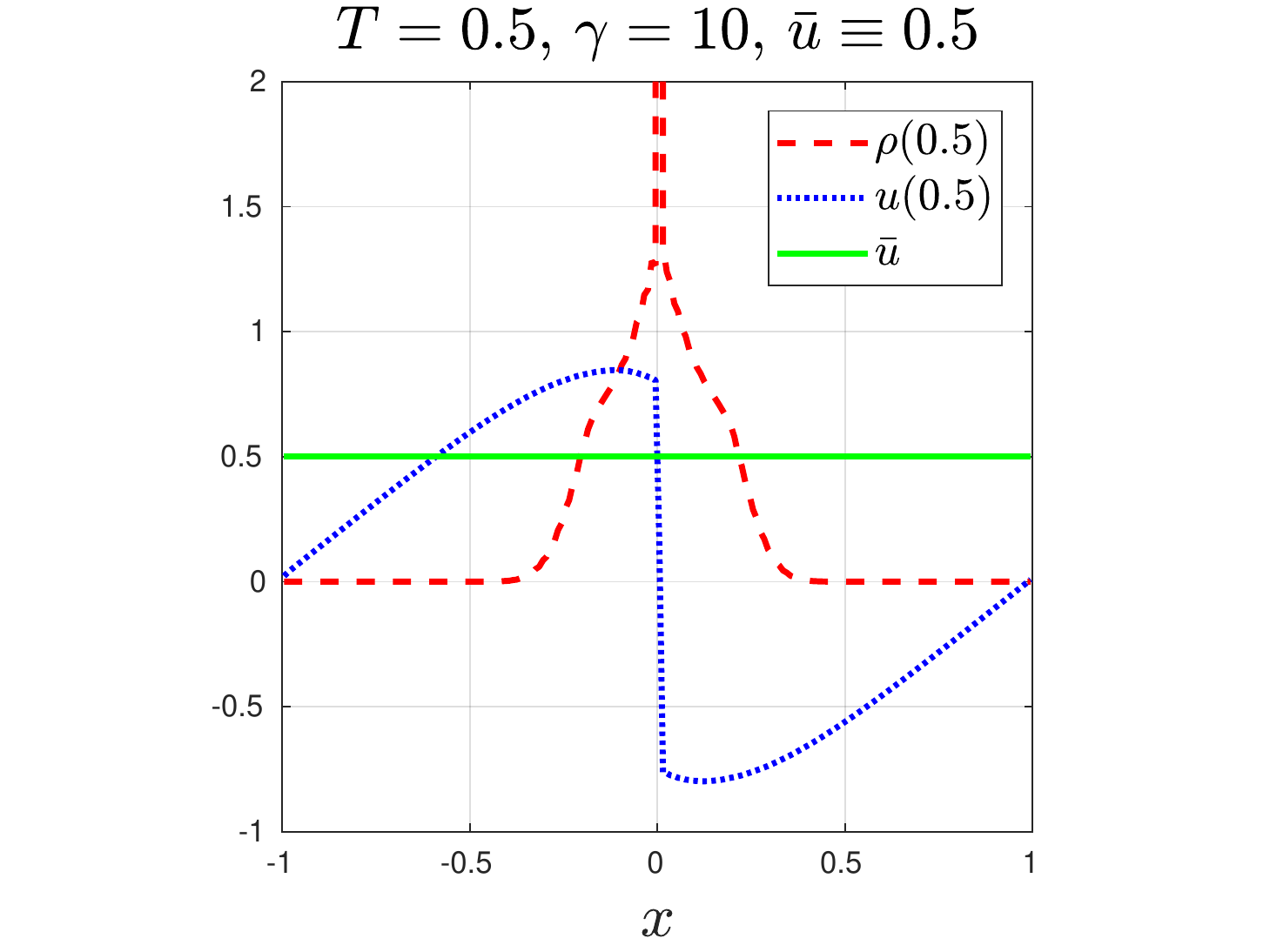}\hfill
		\includegraphics[width=0.34\textwidth]{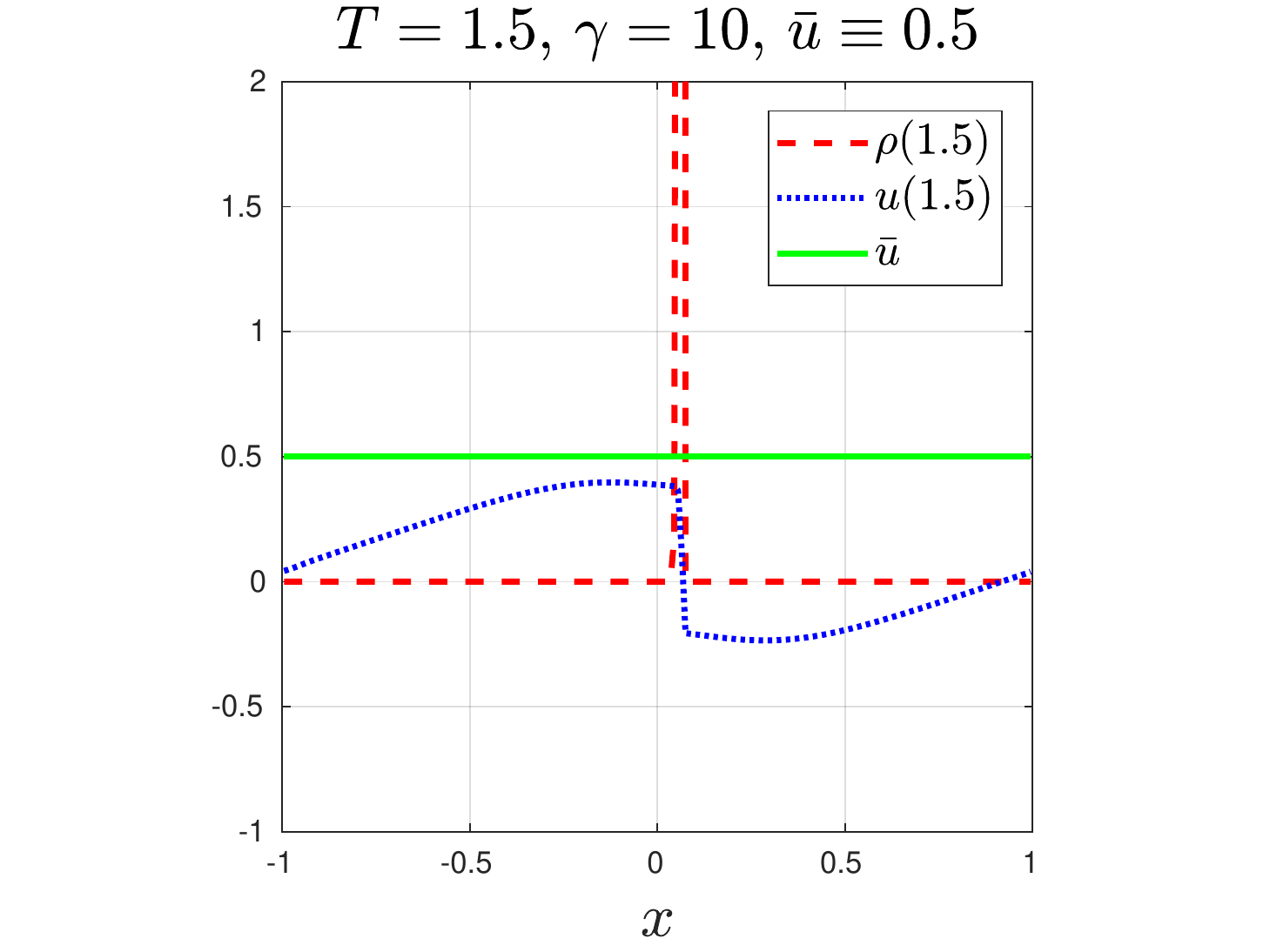}\hfill
		\includegraphics[width=0.34\textwidth]{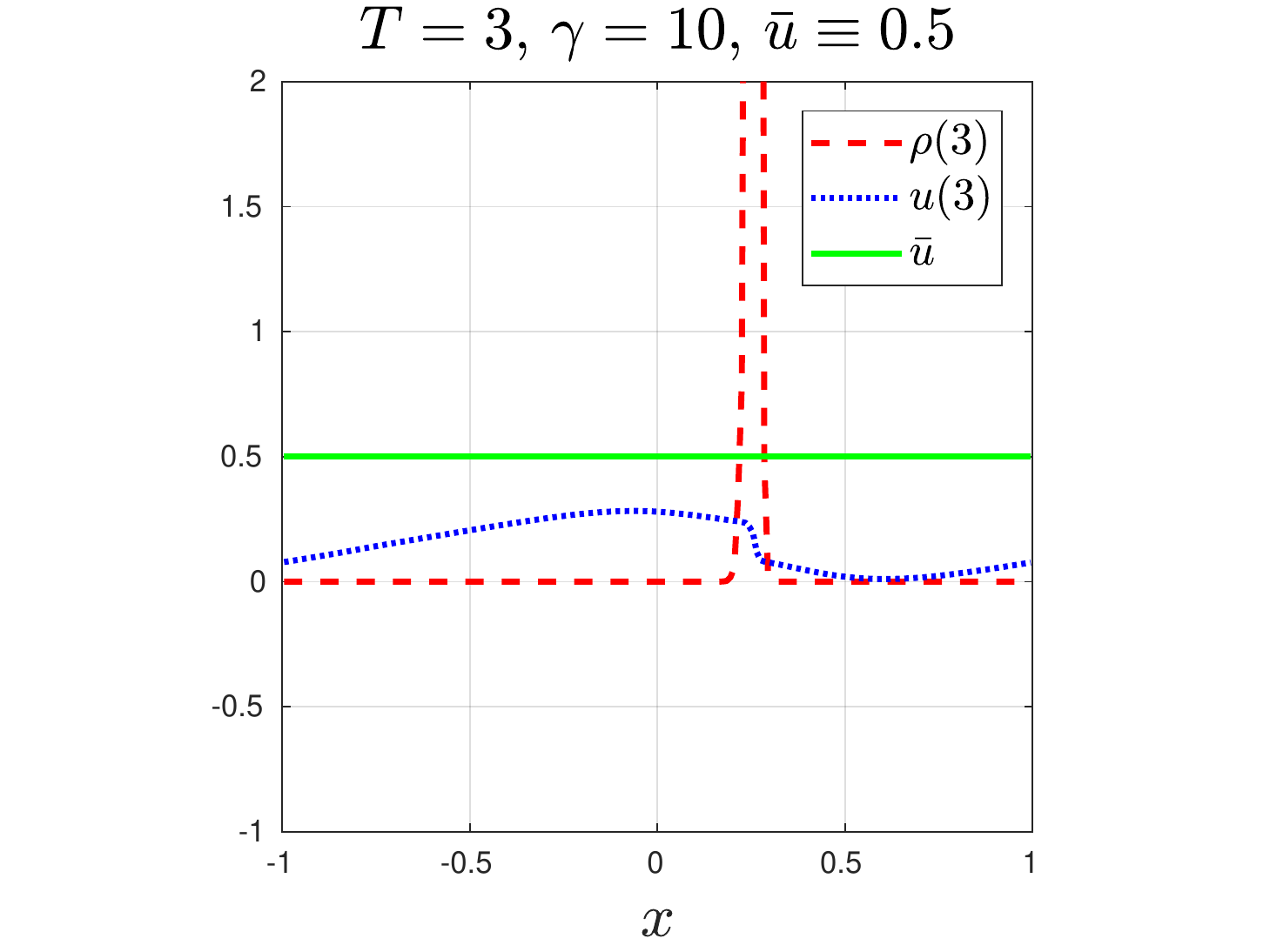}\\
		\hline\\
		\hline\\
		\includegraphics[width=0.34\textwidth]{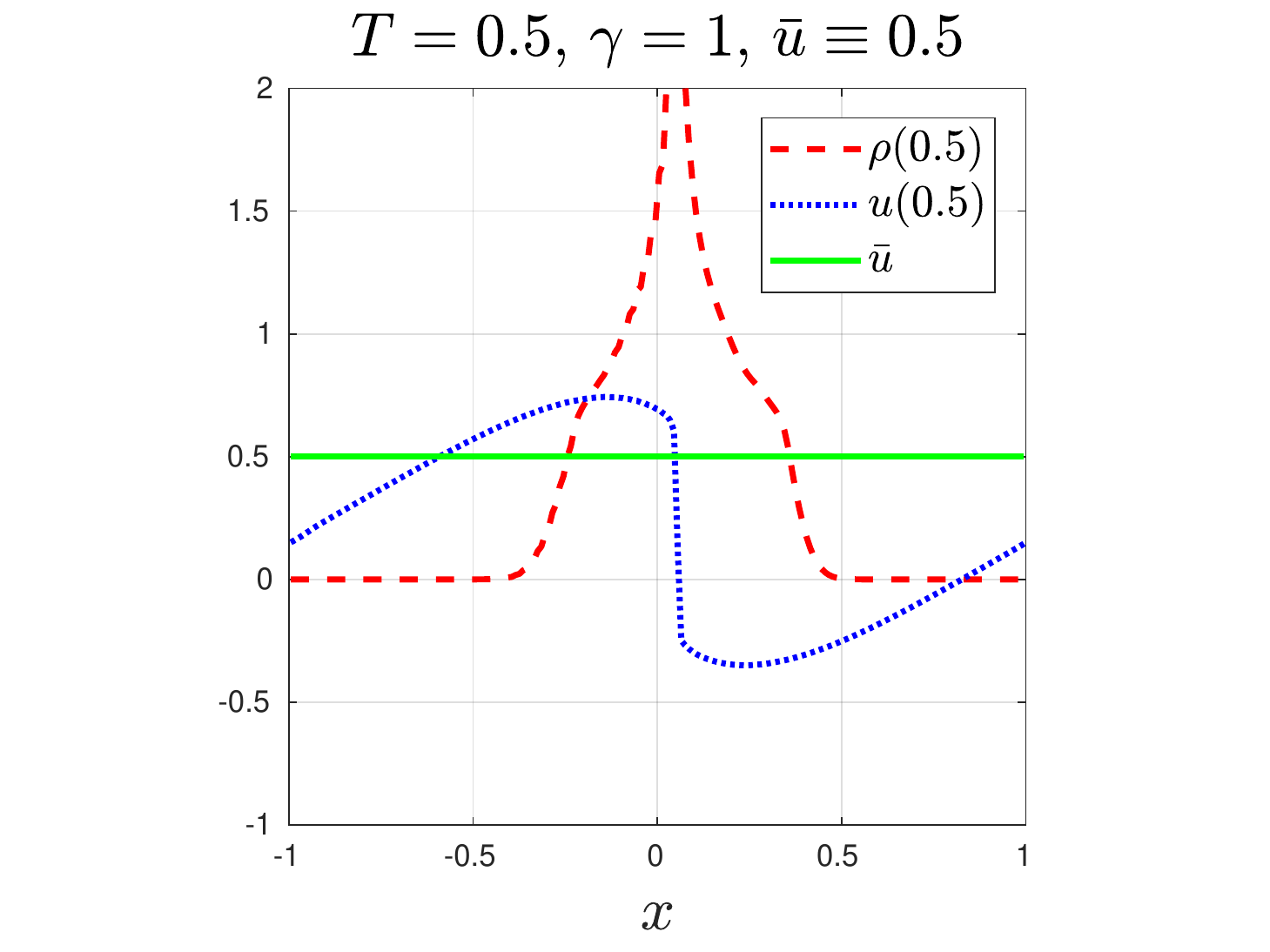}\hfill
		\includegraphics[width=0.34\textwidth]{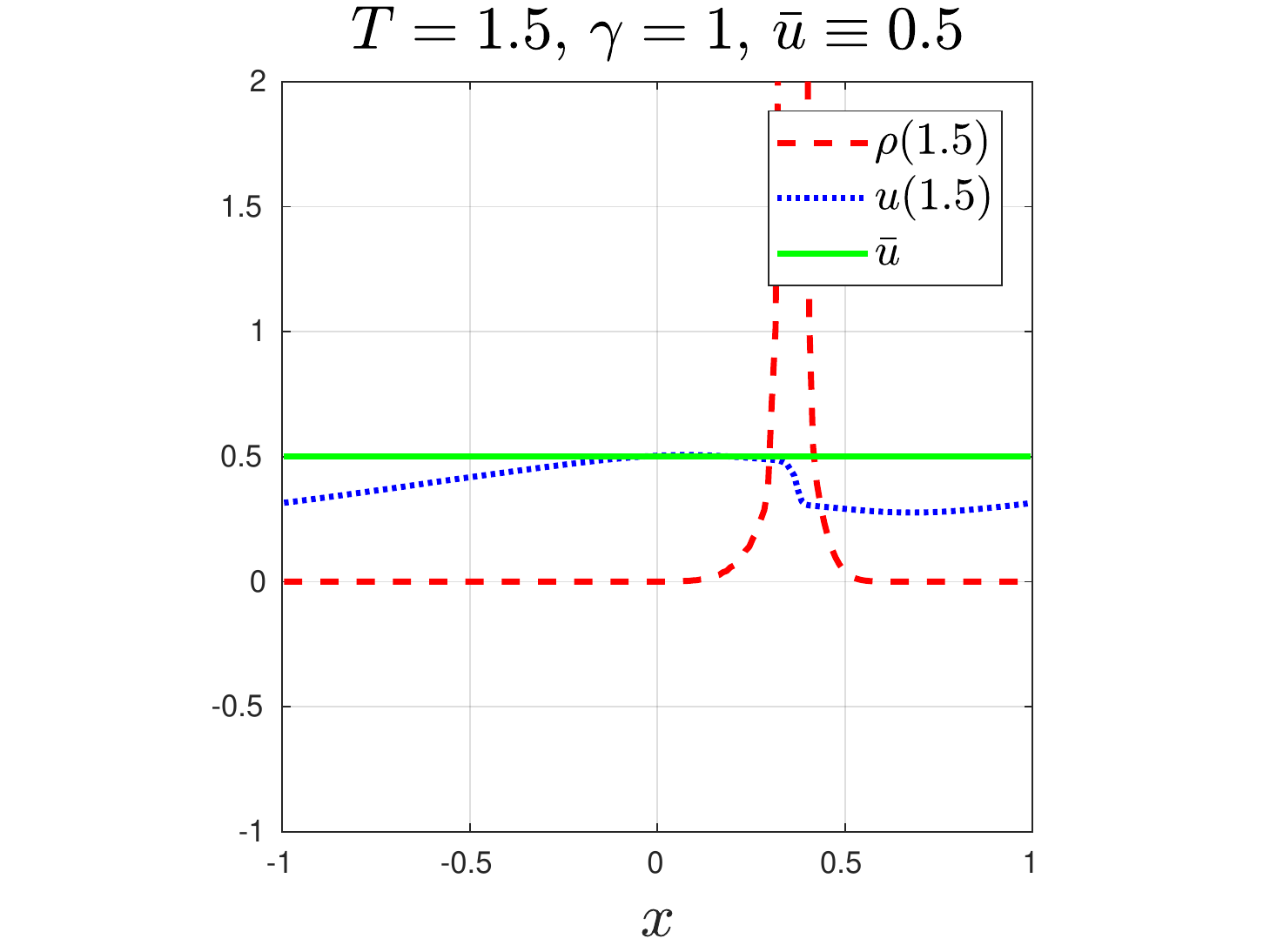}\hfill
		\includegraphics[width=0.34\textwidth]{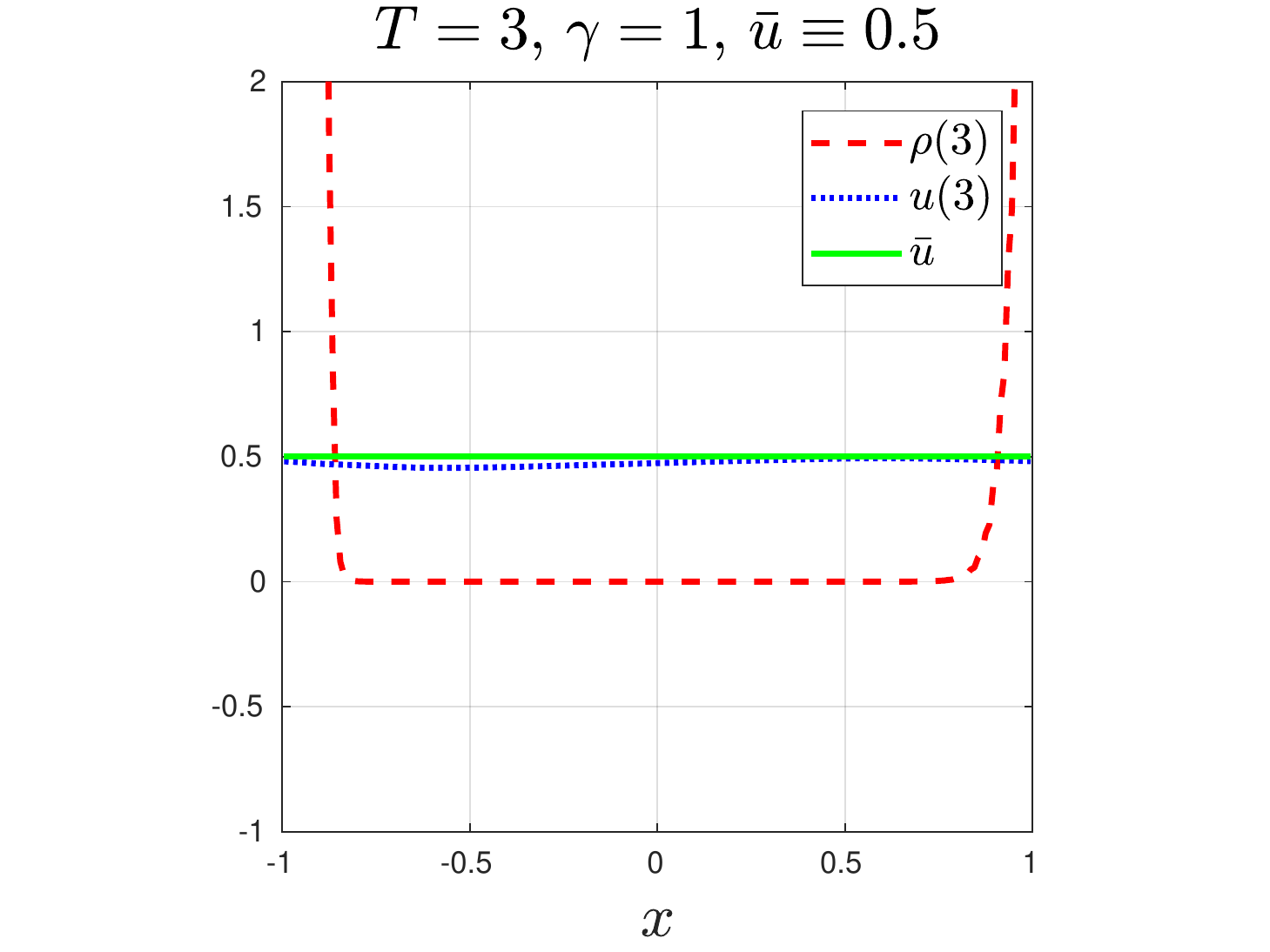}\\
		\hline\\
		\hline\\
		\includegraphics[width=0.34\textwidth]{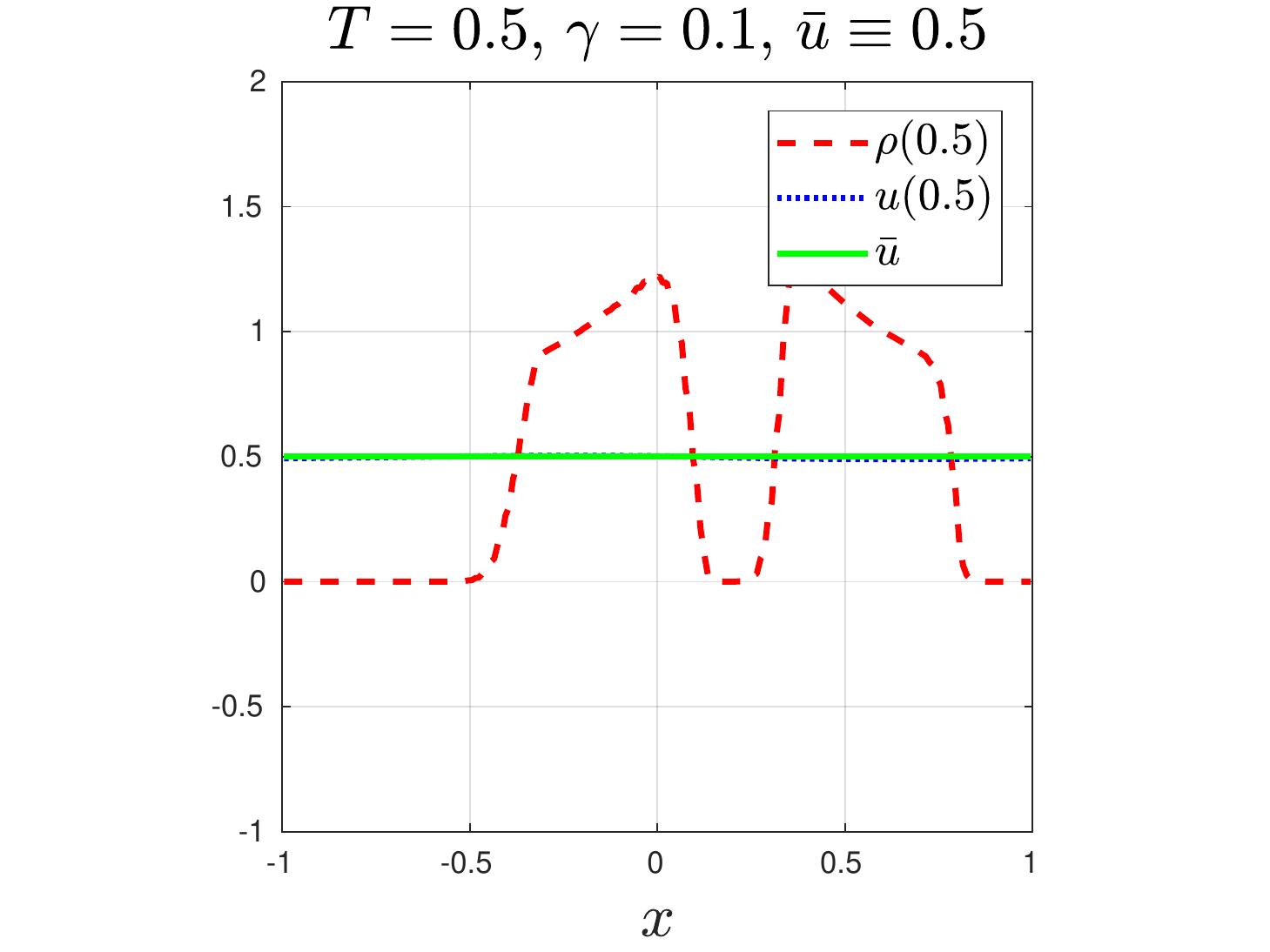}\hfill
		\includegraphics[width=0.34\textwidth]{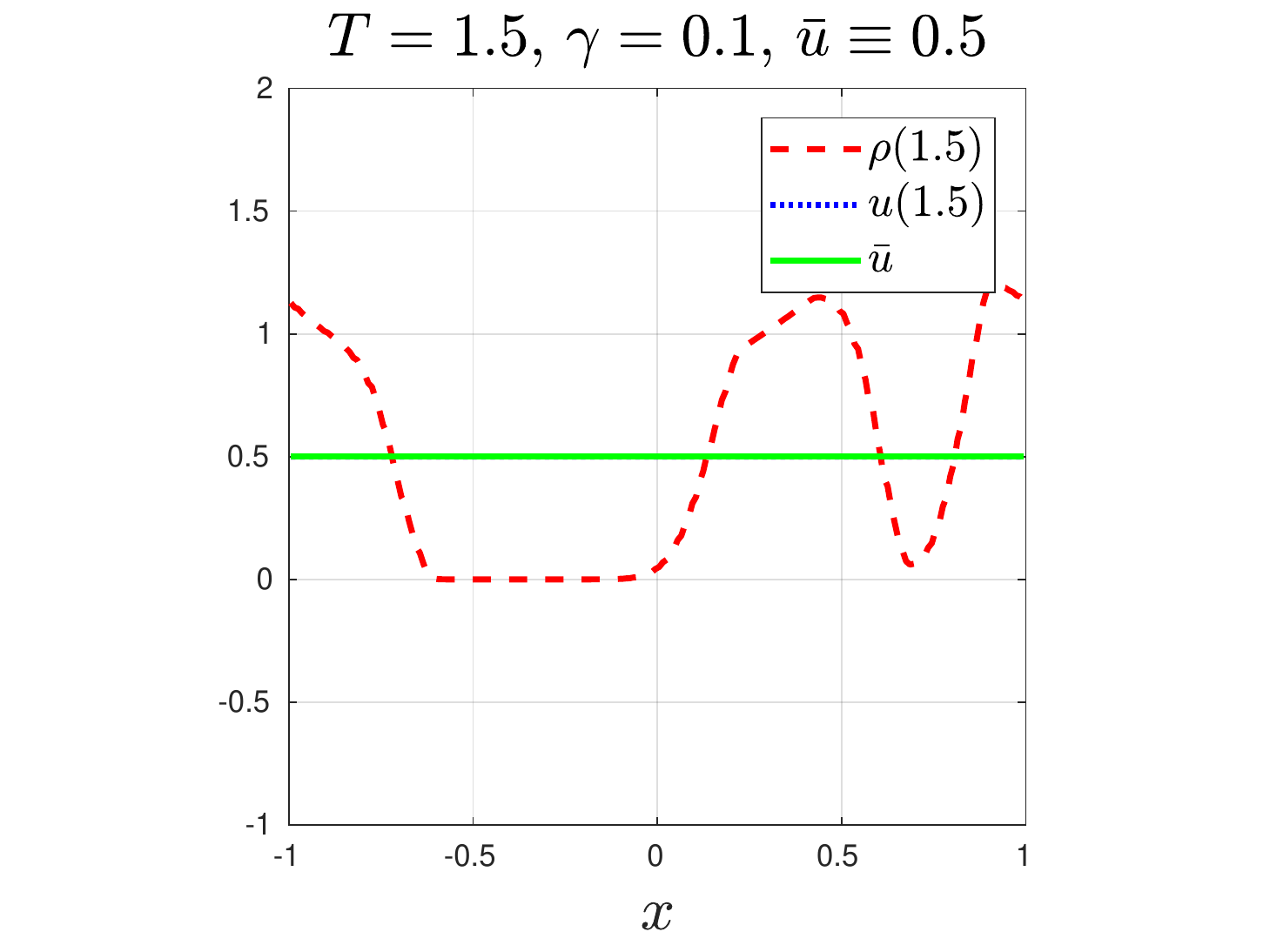}\hfill
		\includegraphics[width=0.34\textwidth]{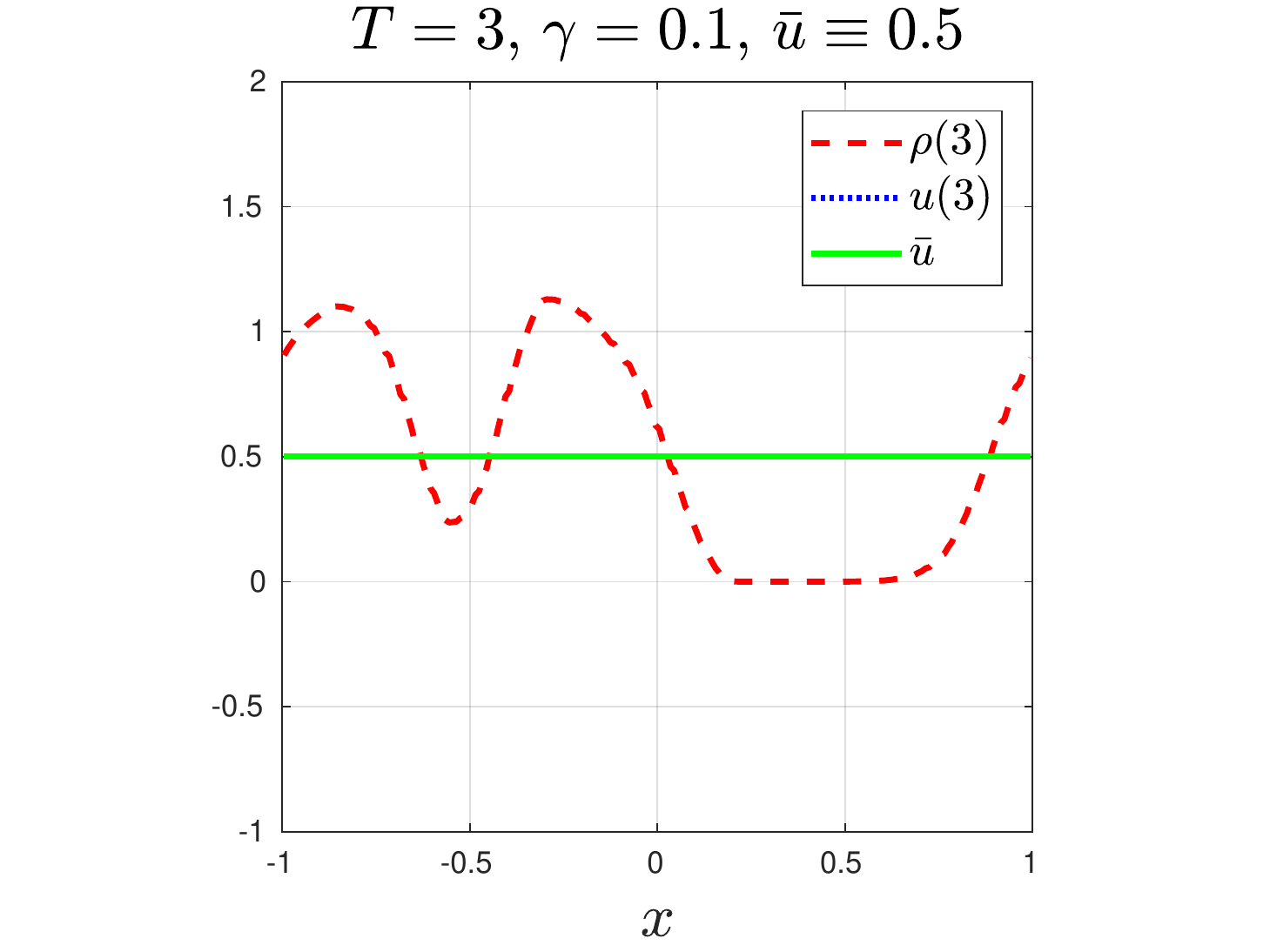}\\
		\hline
	\end{tabular}
	\caption{ (Test 1D : Homogeneous desired state). From top to bottom evolution of system \eqref{main-eq} with $\gamma\in\{10,1,0.1\}$, and $\bar{u}\equiv0.5$.}\label{fig:2}
\end{figure}


\paragraph{Inhomogeneous desired state.}
Let us consider the same scenario as \eqref{ID}, but where we apply an inhomogeneous desired state control, namely
$$ \bar{u}(x) = \sin\left( \frac{\pi x}{L} \right).$$
In Figure \ref{fig:2ih} we display the associated evolutions, where for different intensity of the instantaneous control (i.e. smaller values of $\gamma$) the two intial densities, which are initially pushed toward the center of the domain, change direction, and merge at boundary of the periodic domain.
\begin{figure}[t]
	\centering
	\begin{tabular}{@{}c@{\hspace{1mm}}c@{\hspace{1mm}}c@{\hspace{1mm}}c@{}}
		\hline
		\includegraphics[width=0.34\textwidth]{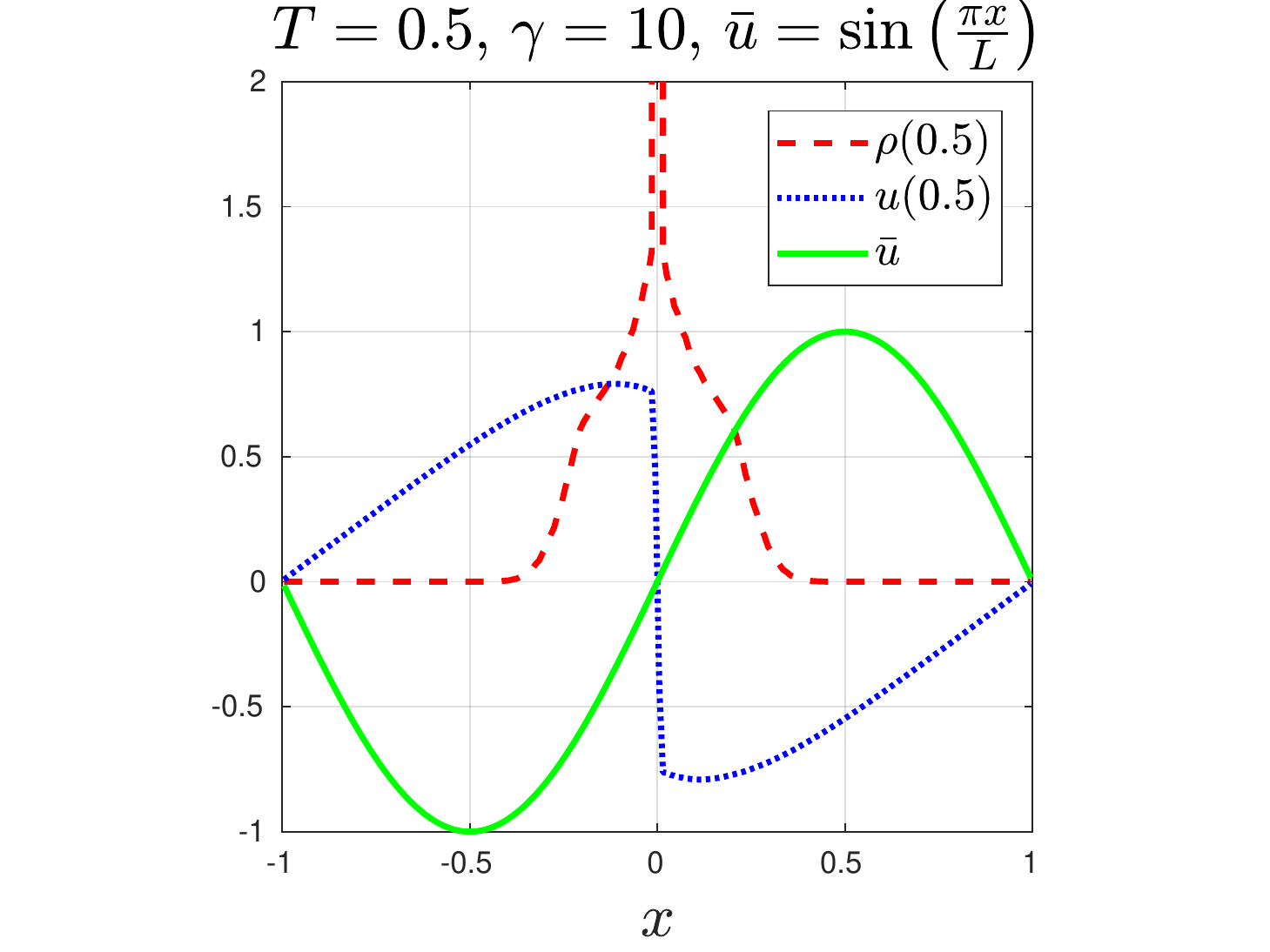}\hfill
		\includegraphics[width=0.34\textwidth]{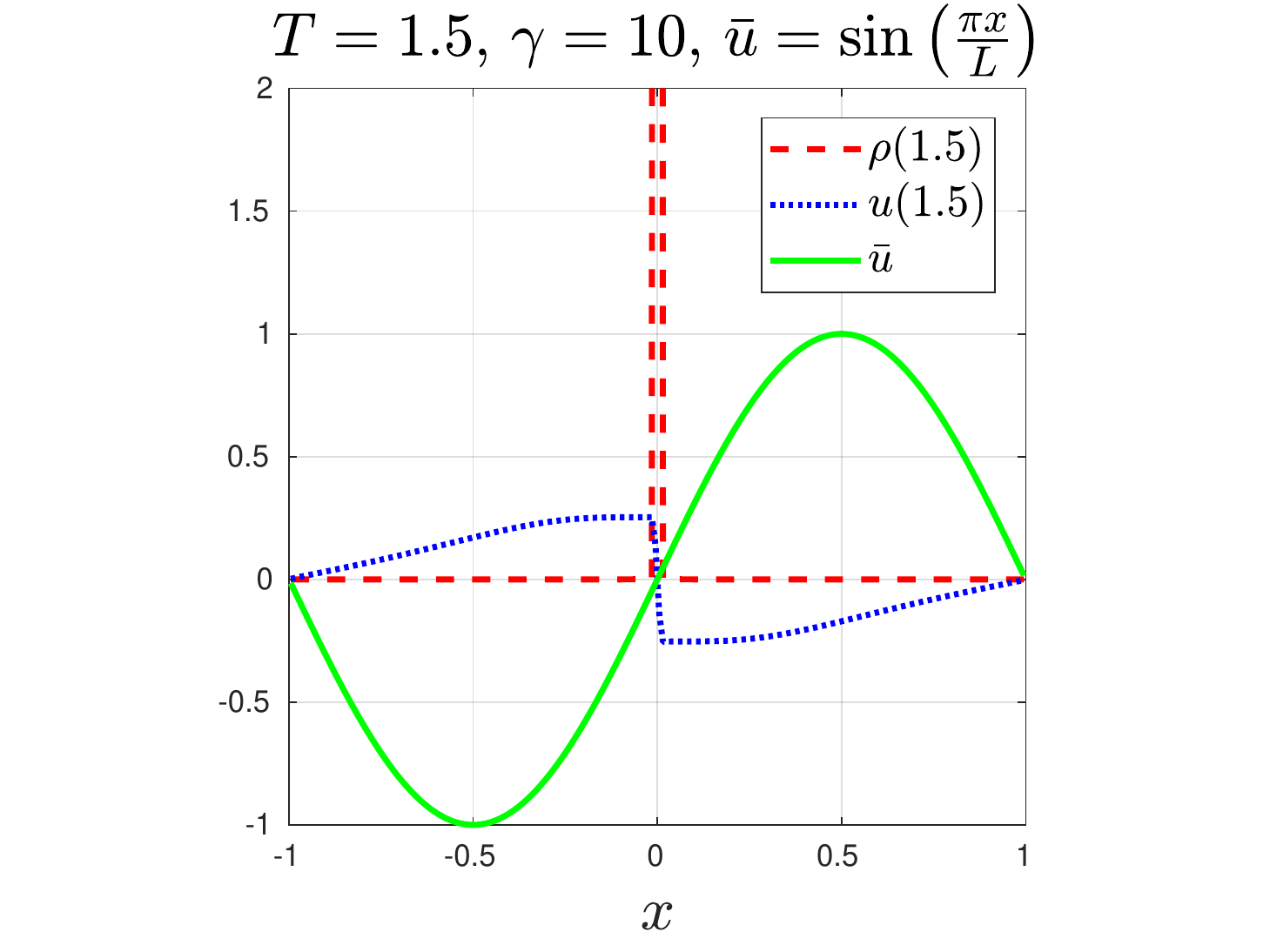}\hfill
		\includegraphics[width=0.34\textwidth]{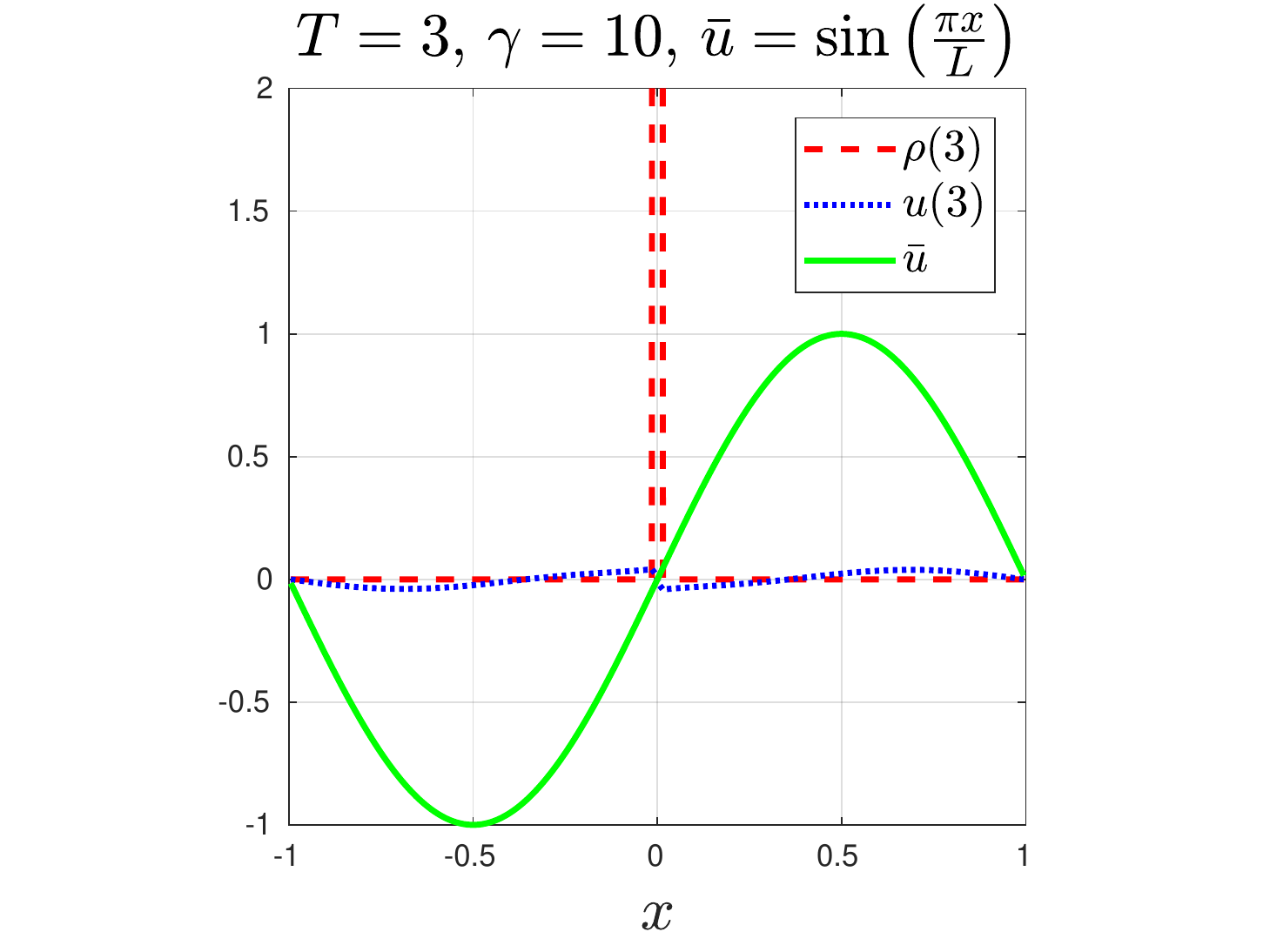}\\
		\hline\\
		\hline\\
		\includegraphics[width=0.34\textwidth]{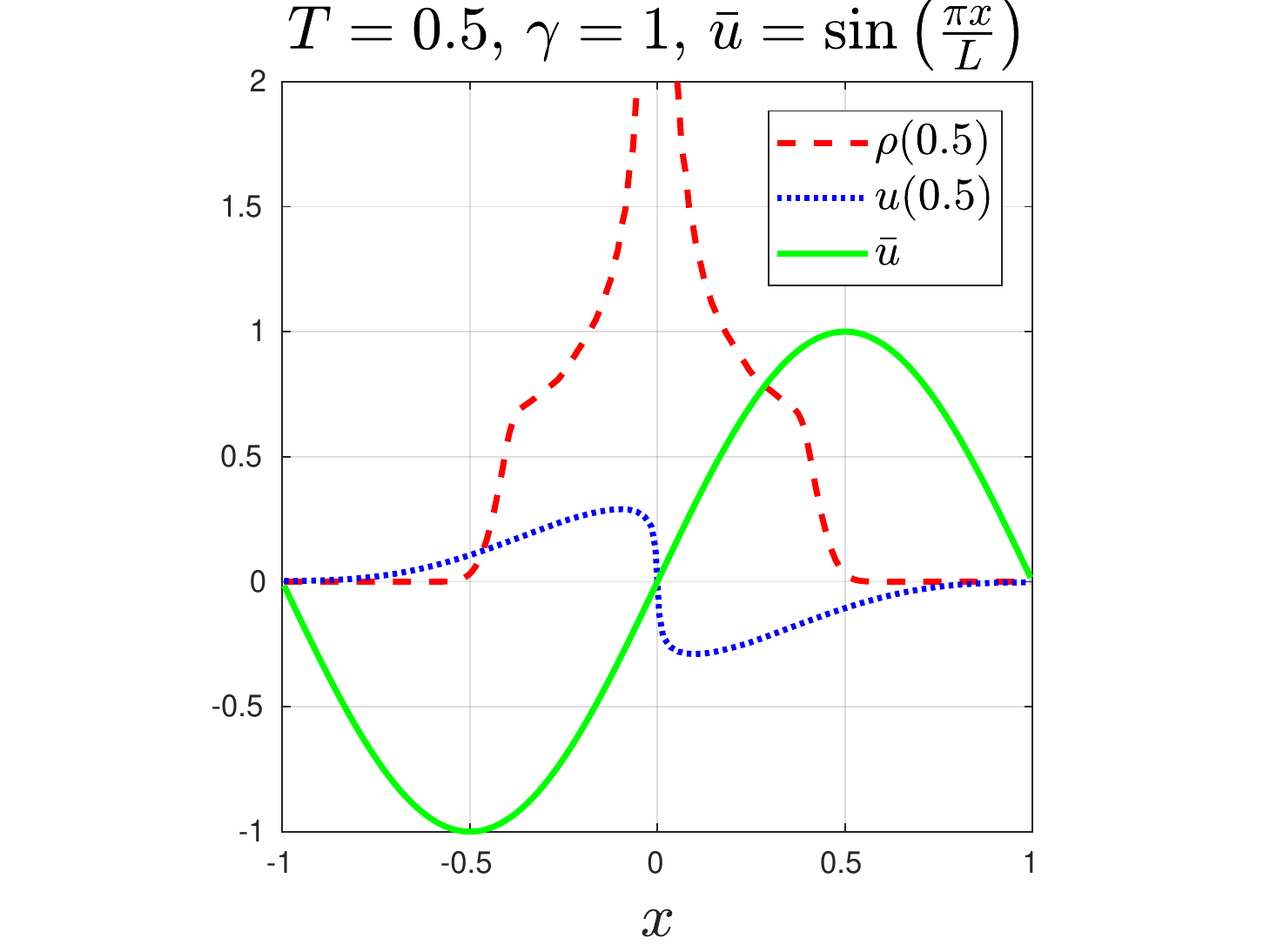}\hfill
		\includegraphics[width=0.34\textwidth]{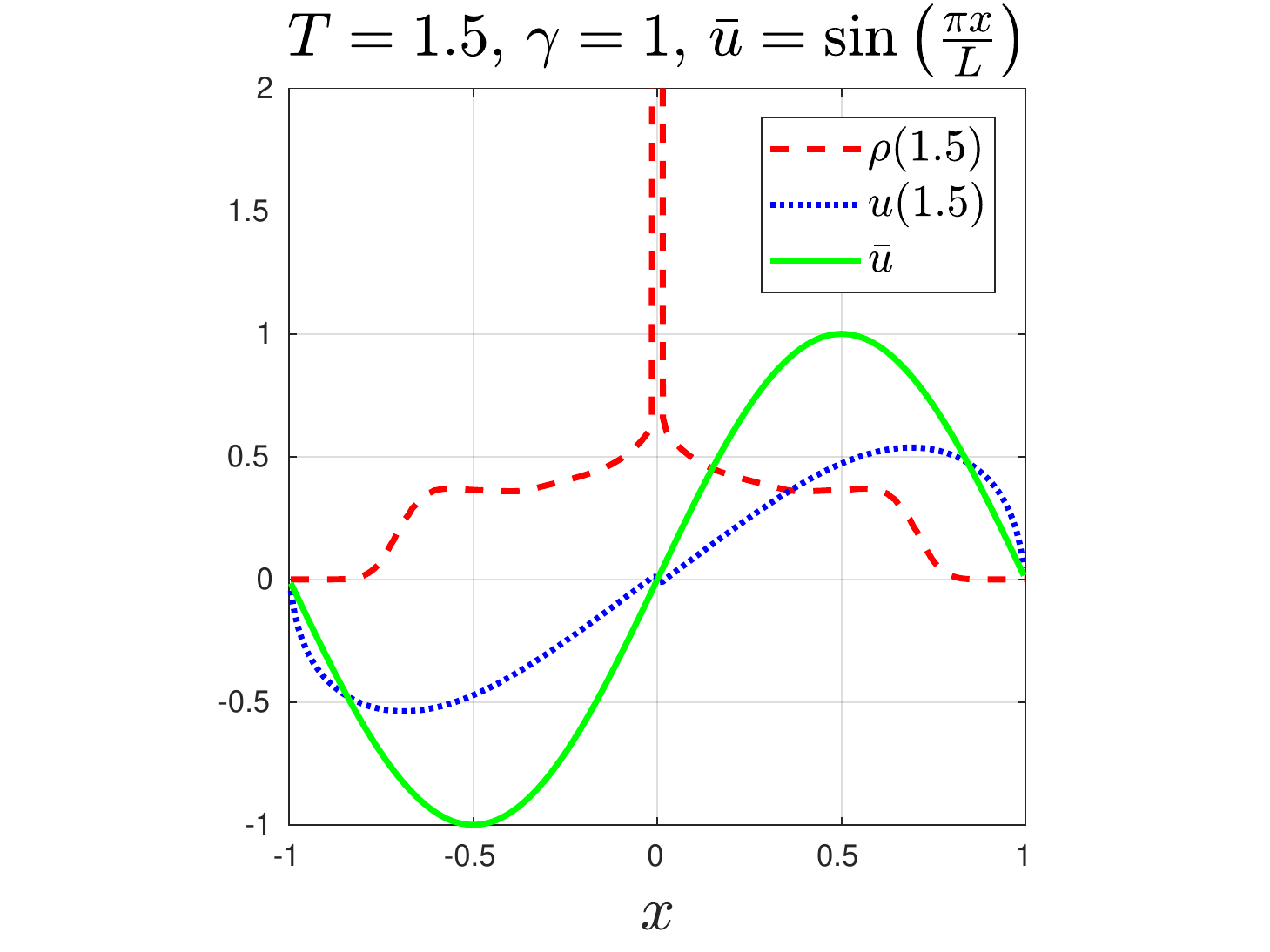}\hfill
		\includegraphics[width=0.34\textwidth]{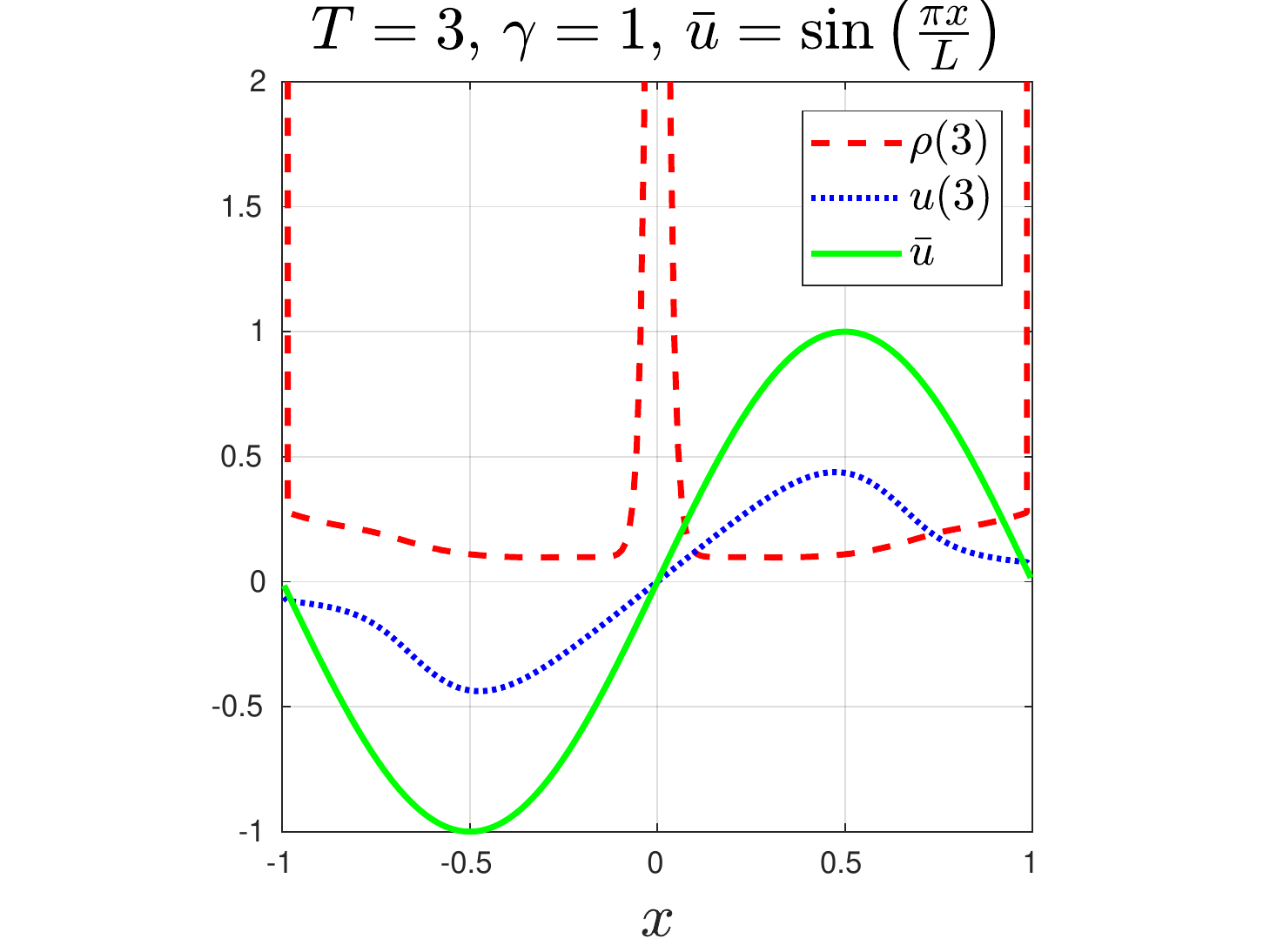}\\
		\hline\\
		\hline\\
		\includegraphics[width=0.34\textwidth]{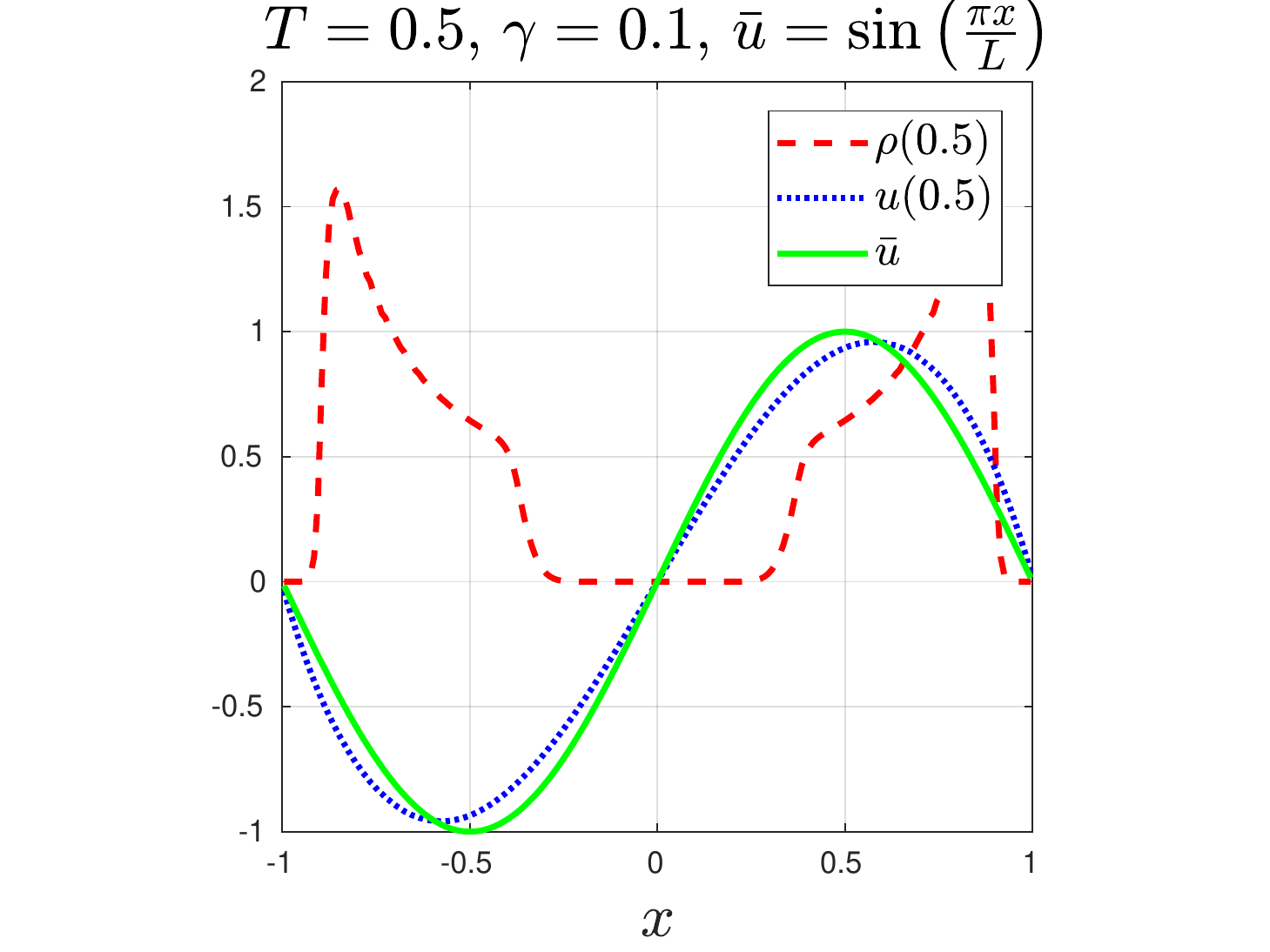}\hfill
		\includegraphics[width=0.34\textwidth]{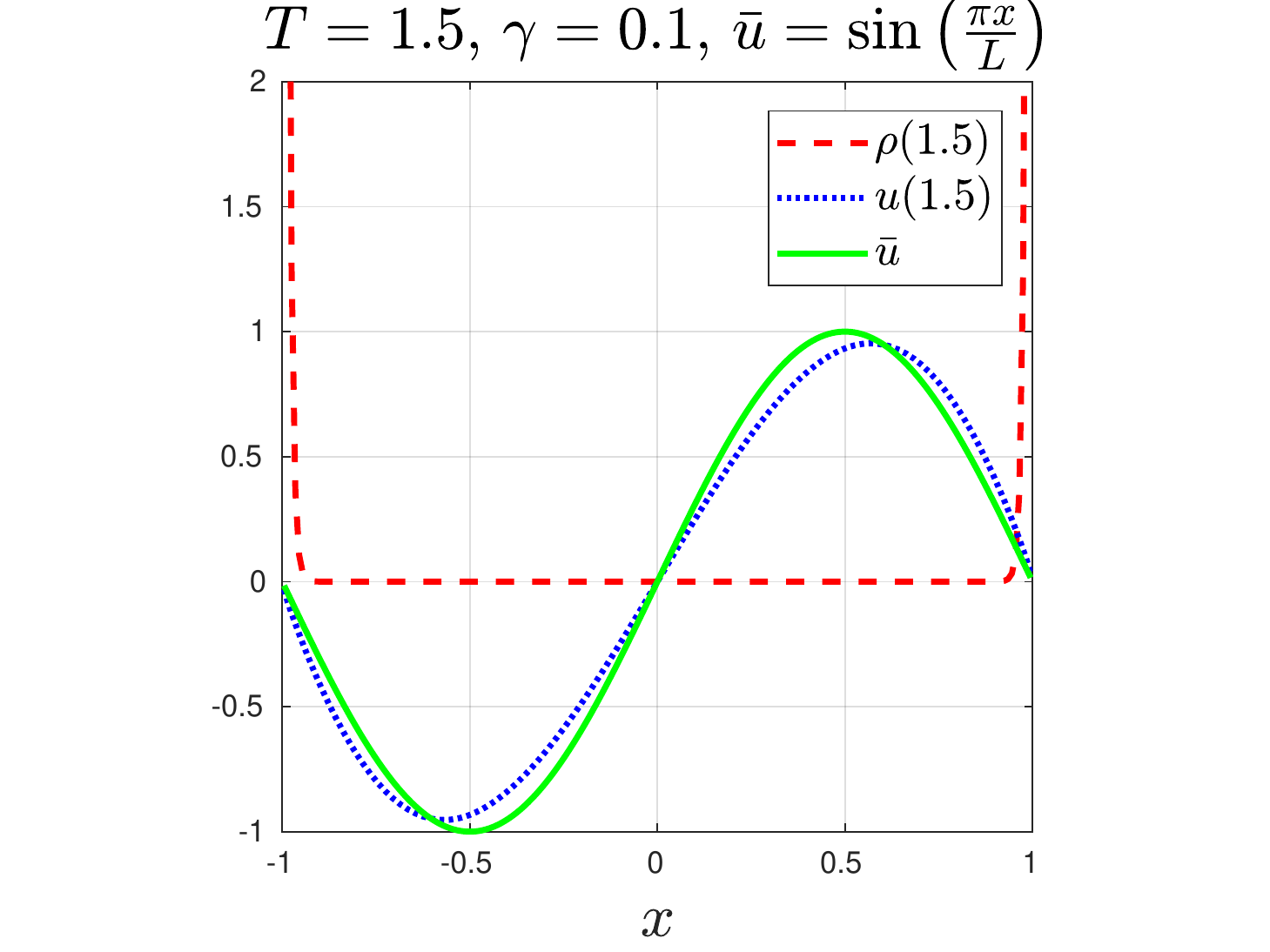}\hfill
		\includegraphics[width=0.34\textwidth]{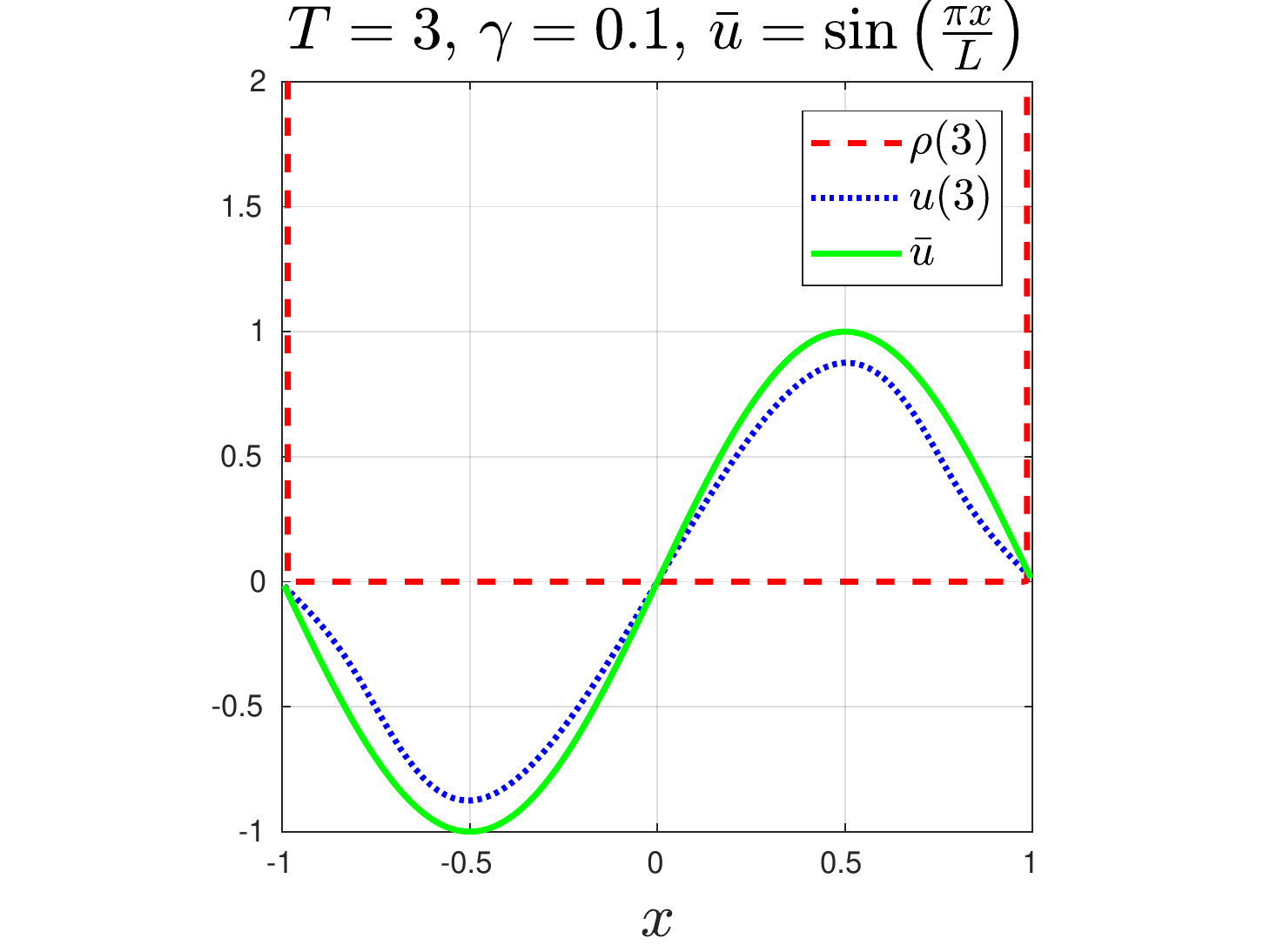}\\
		\hline\\
	\end{tabular}
	\caption{  (Test 1D : Inhomogeneous desired state). From top to bottom evolution of system \eqref{main-eq} with $\gamma\in\{10,1,0.1\}$, and $\bar{u}=\sin\left( \frac{\pi x}{L} \right)$.}\label{fig:2ih}
\end{figure}

\subsection{Numerical experiments in two dimensions.}\label{2dnum}
We now investigate the two dimensional case of the Euler alignment dynamics $\eqref{main-eq}$ with Cucker-Smale type of interactions \eqref{CS}, and for different choice of the control.

For the numerical scheme, we set all parameters analog to the case in one dimension in Section \ref{1dnum}. The only difference is that we lowered the resolution to $[64\times64]$ cells of the special domain $(x,y)\in [-L,L]\times[-L,L]$, $L=1$. Therefore, we have $\Delta x = \Delta y = 0.03125.$ For all cases, we will give a comparison of the uncontrolled case ($\phi(x,t)\equiv 0$) with the corresponding controlled cases.
\begin{remark}[2D graphical representation]
	In all following 2D plots  we will represent the density $\rho$, and the momenta $\rho u$, and occasionally the velocity field $u$.
	
	In particular, we display the momenta-vector $\rho u$ by black quiver plot of varying orientation and length corresponding the values of the vector at a given position, we will plot the $2$-norm of this field as a contour plot. In this way, we give the reader an immediate insight of the absolute magnitude of the underlying vector field rather than a relative information on this.
	
	In a second step we will drop the presentation of the velocity field $u$ and refer to the initial data, since we are only interested in the support of $\rho$, and its the evolution.	
\end{remark}

\paragraph{Test 2D: Uncontrolled symmetric heaps.}
We consider an initial symmetric scenario by setting first the quantities
\[
\rho_0^\pm(x,y)  = \max\left\{\exp\left[ -10\left(x \mp \frac{1}{2} \right)^2 -10\left(y \mp \frac{1}{2} \right)^2 \right] -\frac{1}{5},0\right\},
\]
and then definining the initial density and velocity as follows
	\begin{align}\label{s2Din}
&\rho_0(x,y) = \rho_0^+(x,y)+\rho^-(x,y),\quad
u_0(x,y) = (2H(x)-1,2H(y)-1)^T ,
	\end{align}
where $H(\cdot)$ denotes the {\it Heaviside Function}. Note that we expect an alignment at velocity $(0,0)^T$ due to the symmetry. In Figure \ref{fig:6} we show the evolution up to $T=5$, reporting in the first row of the plots the density $\rho$, in the second row the velocity $u$, and in the last row the momenta, $\rho u$. Indeed we observe alignment toward zero velocity, since the magnitude of the velocity $\|u\|_2$ shrink to zero, whereas the density concentrates in zero.
	\begin{figure}[t]
		\centering
		\begin{tabular}{@{}c@{\hspace{1mm}}c@{\hspace{1mm}}c@{\hspace{1mm}}c@{}}
			\hline
			\includegraphics[width=0.34\textwidth]{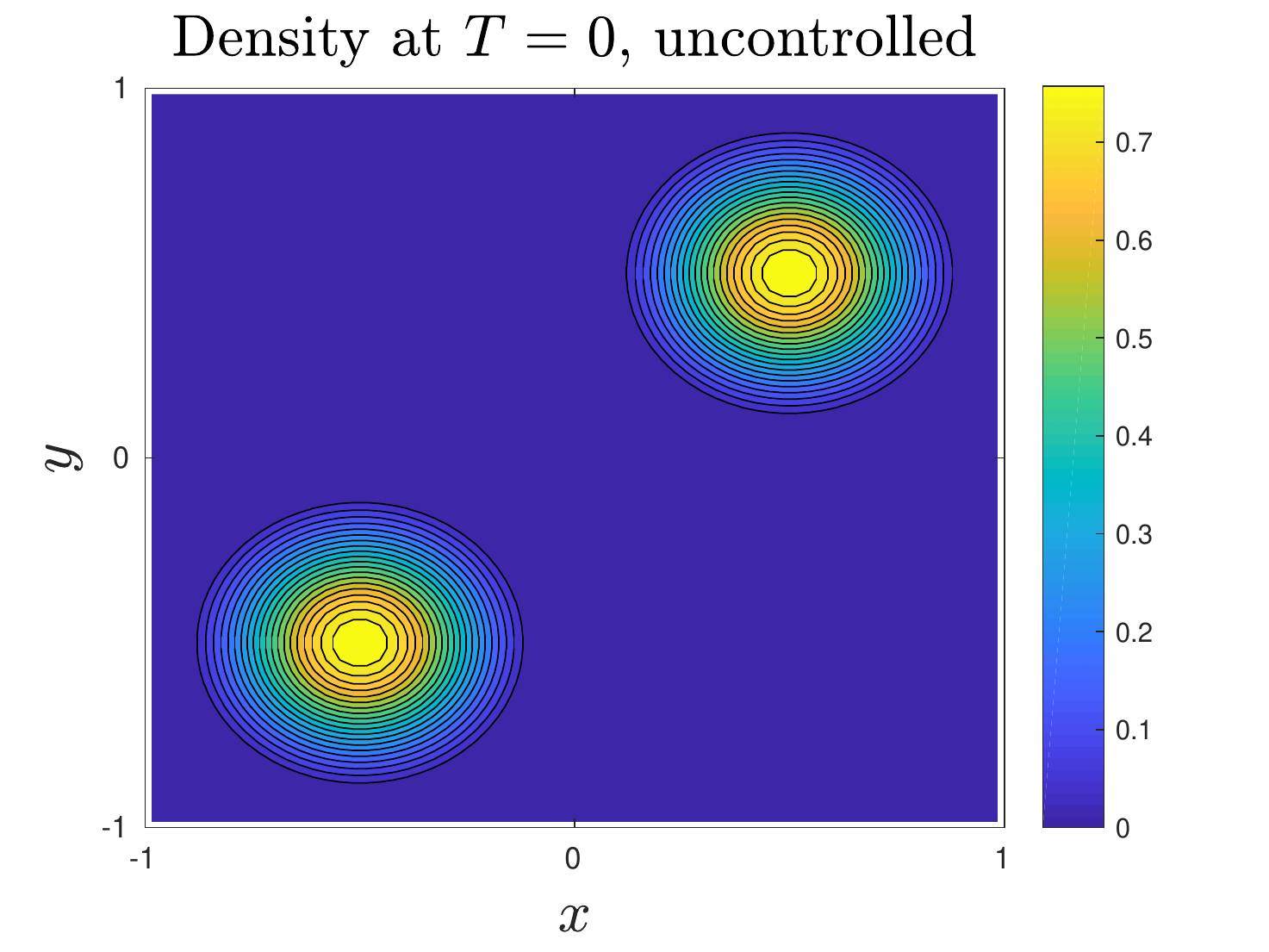}\hfill
			\includegraphics[width=0.34\textwidth]{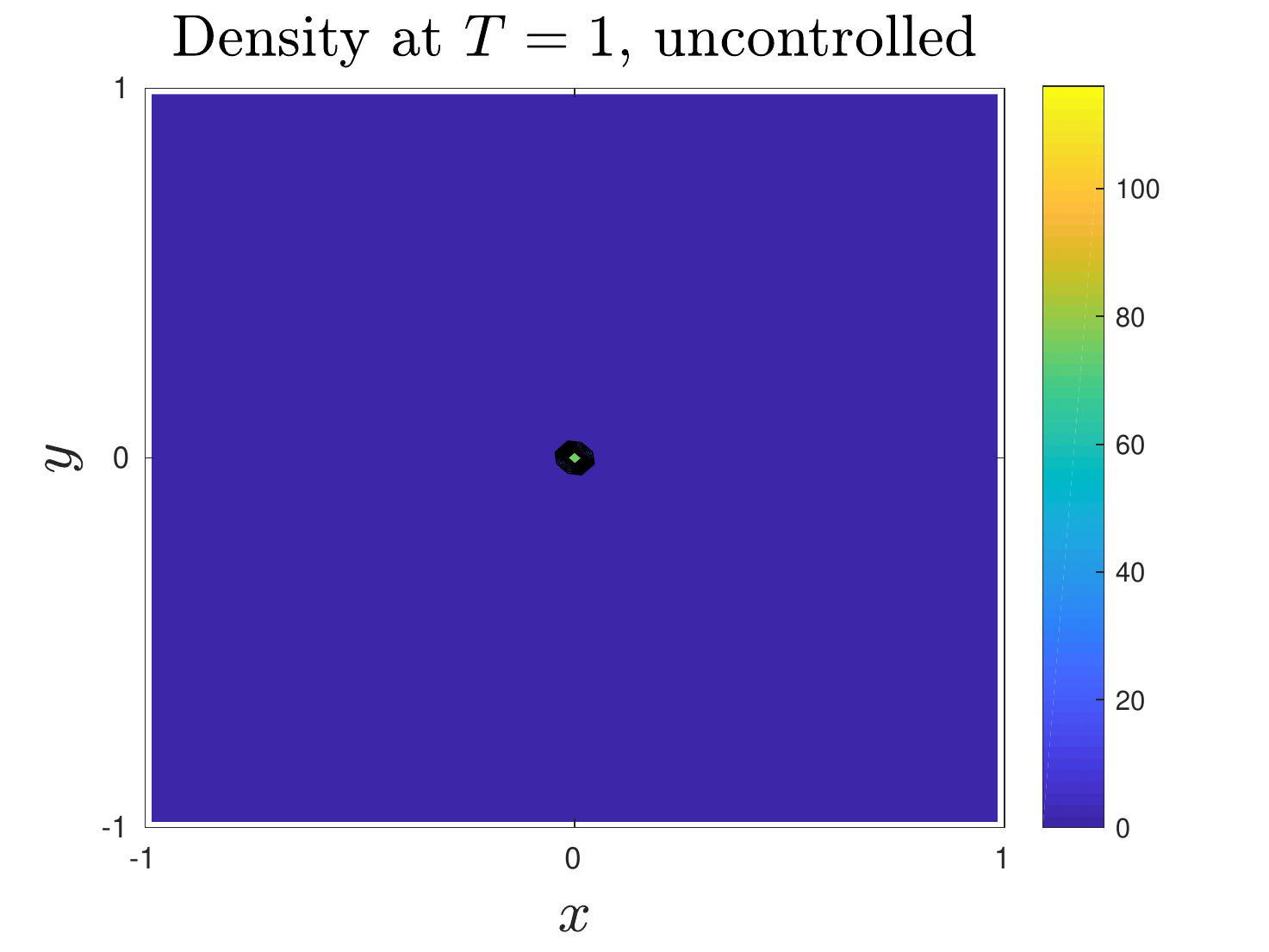}\hfill
			\includegraphics[width=0.34\textwidth]{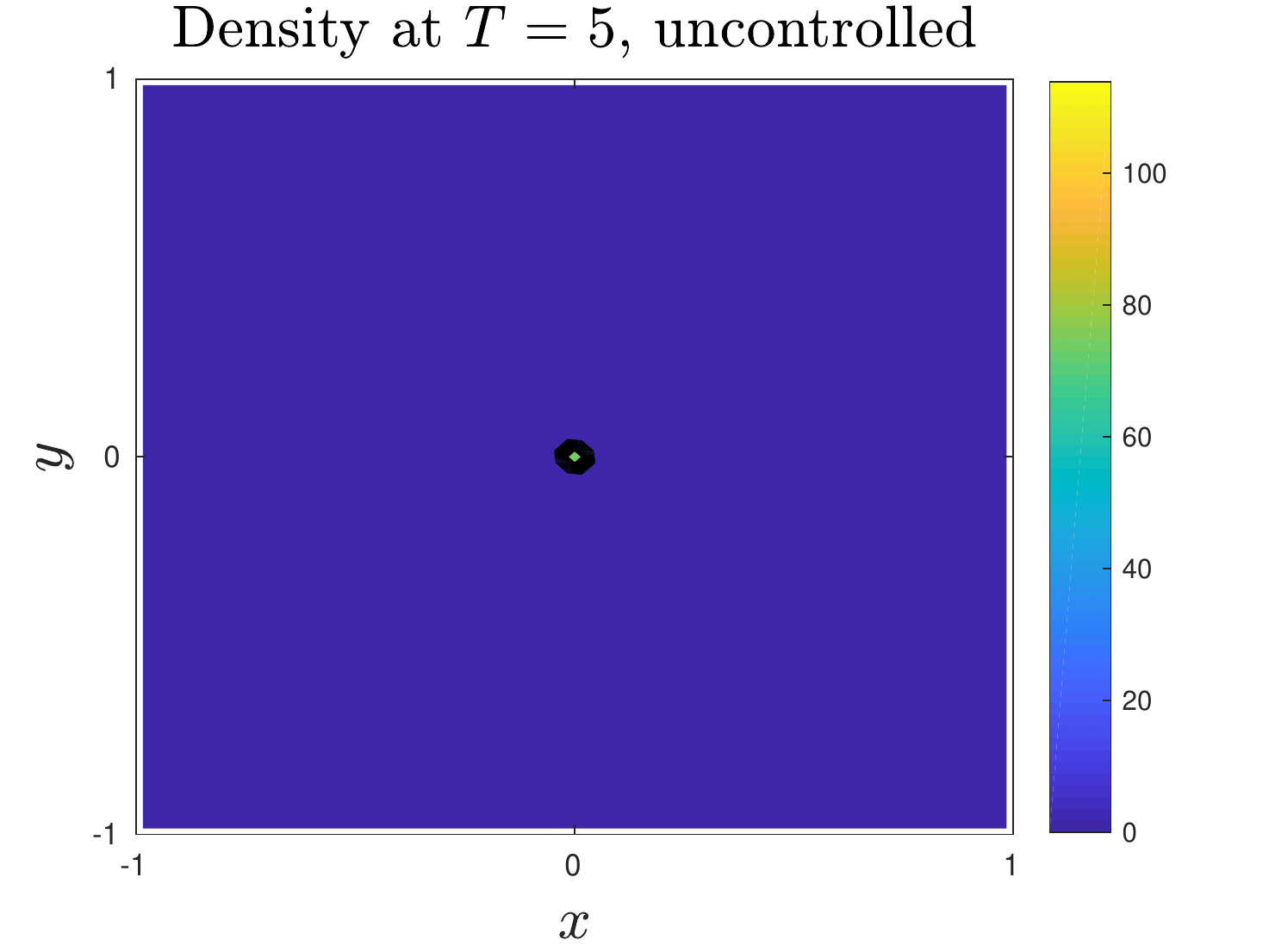}\\
			\hline\\ 
					\includegraphics[width=0.34\textwidth]{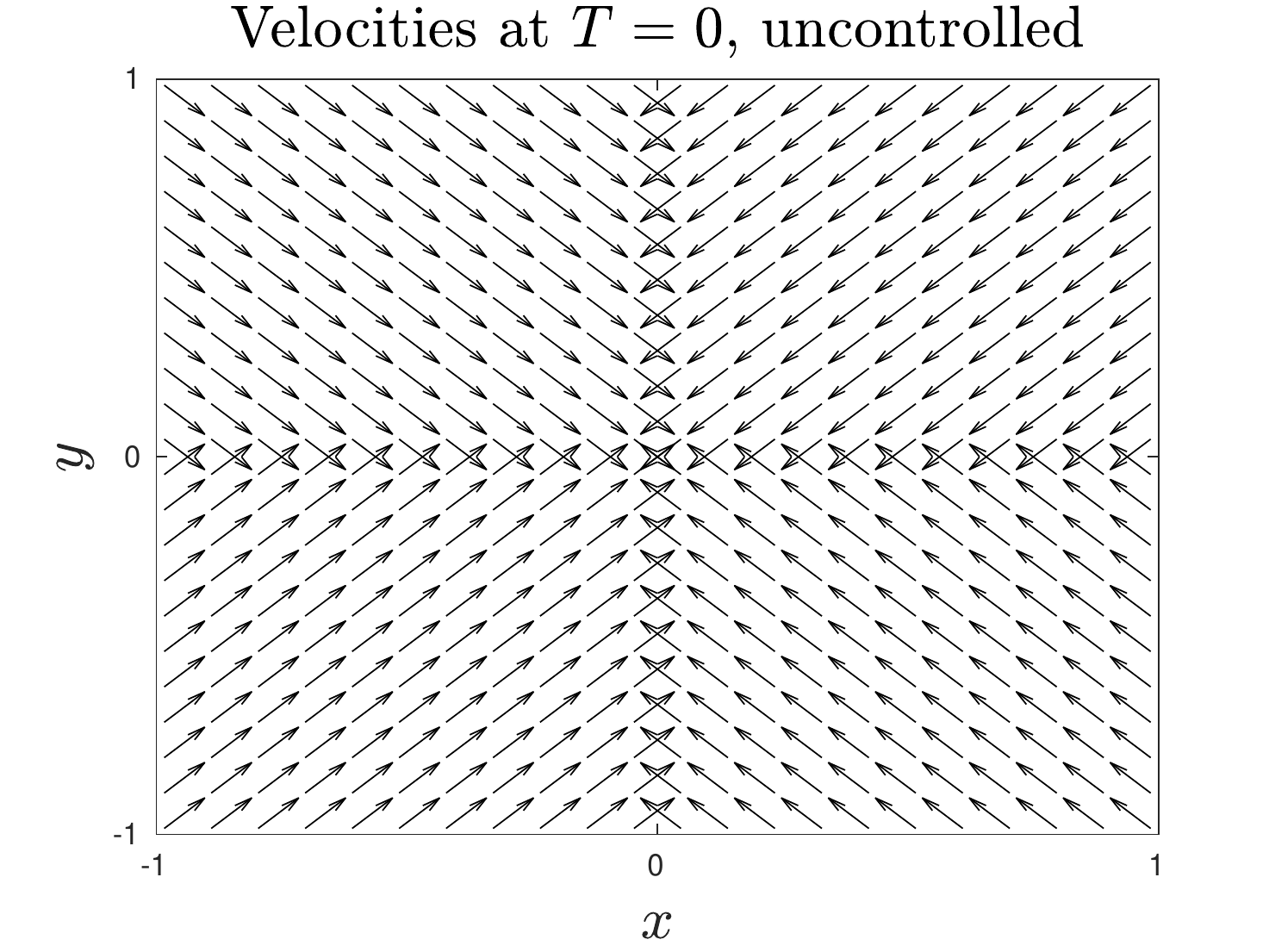}\hfill
					\includegraphics[width=0.34\textwidth]{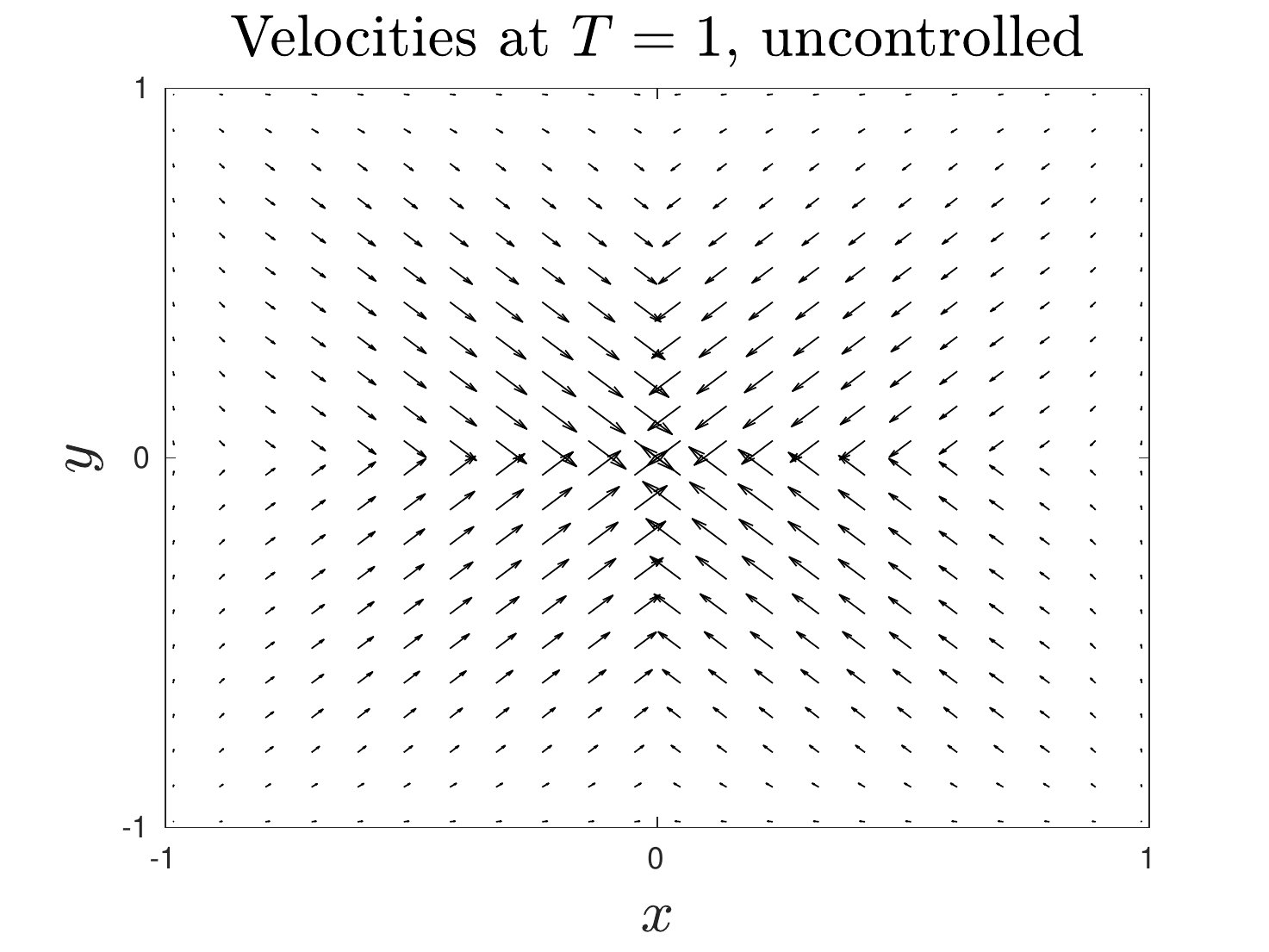}\hfill
					\includegraphics[width=0.34\textwidth]{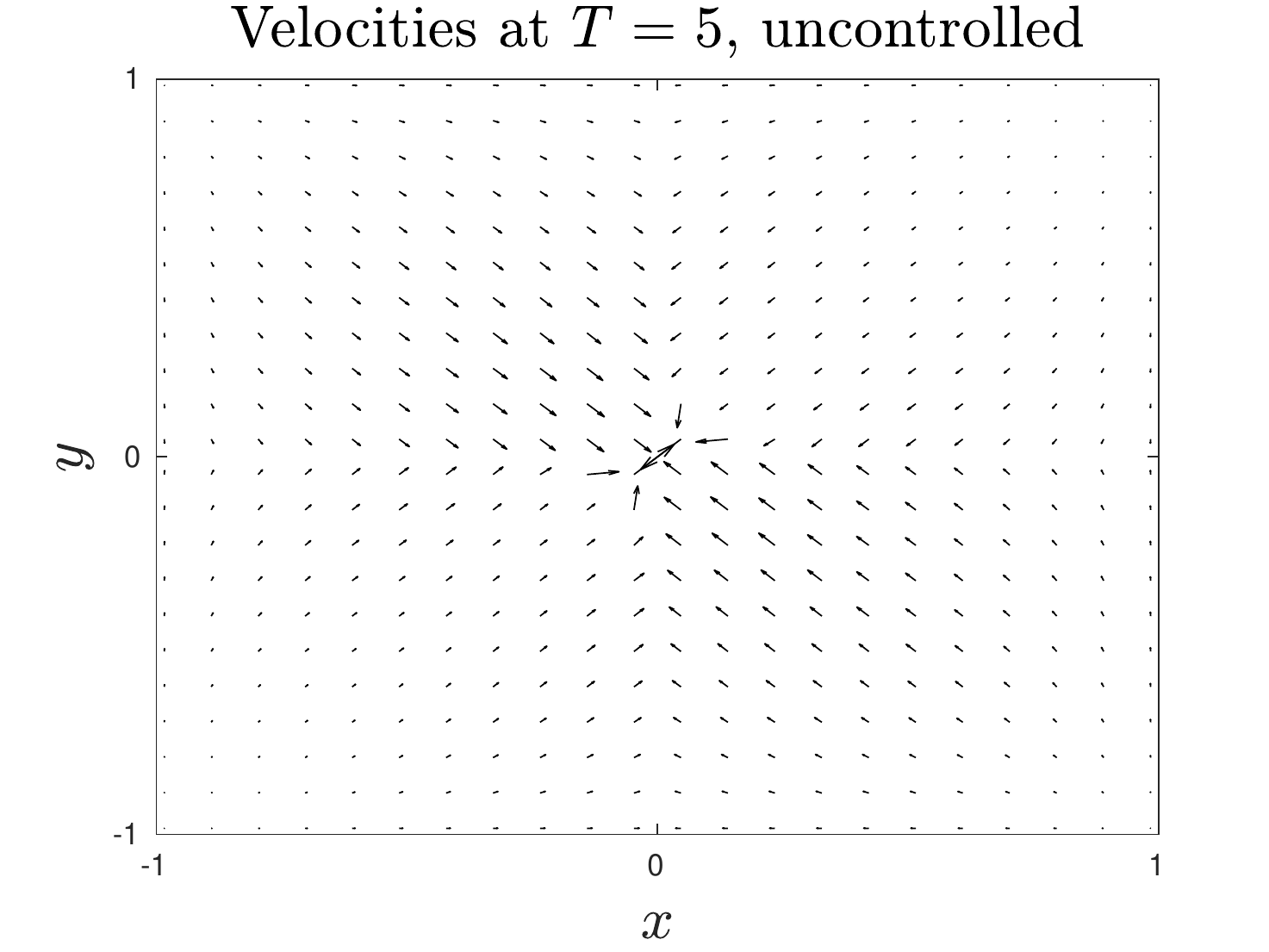}\\
			\hline\\
			\includegraphics[width=0.34\textwidth]{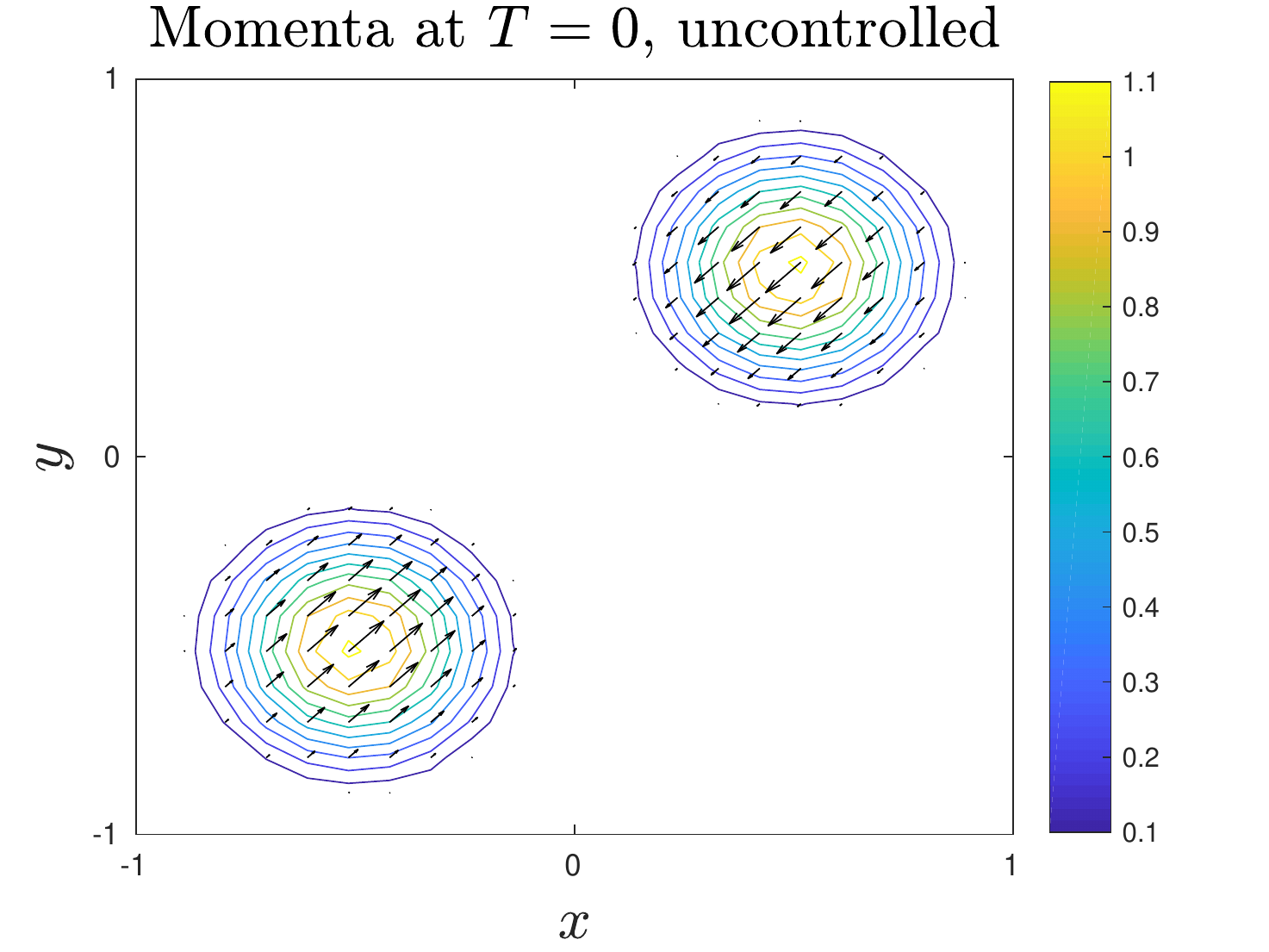}\hfill
			\includegraphics[width=0.34\textwidth]{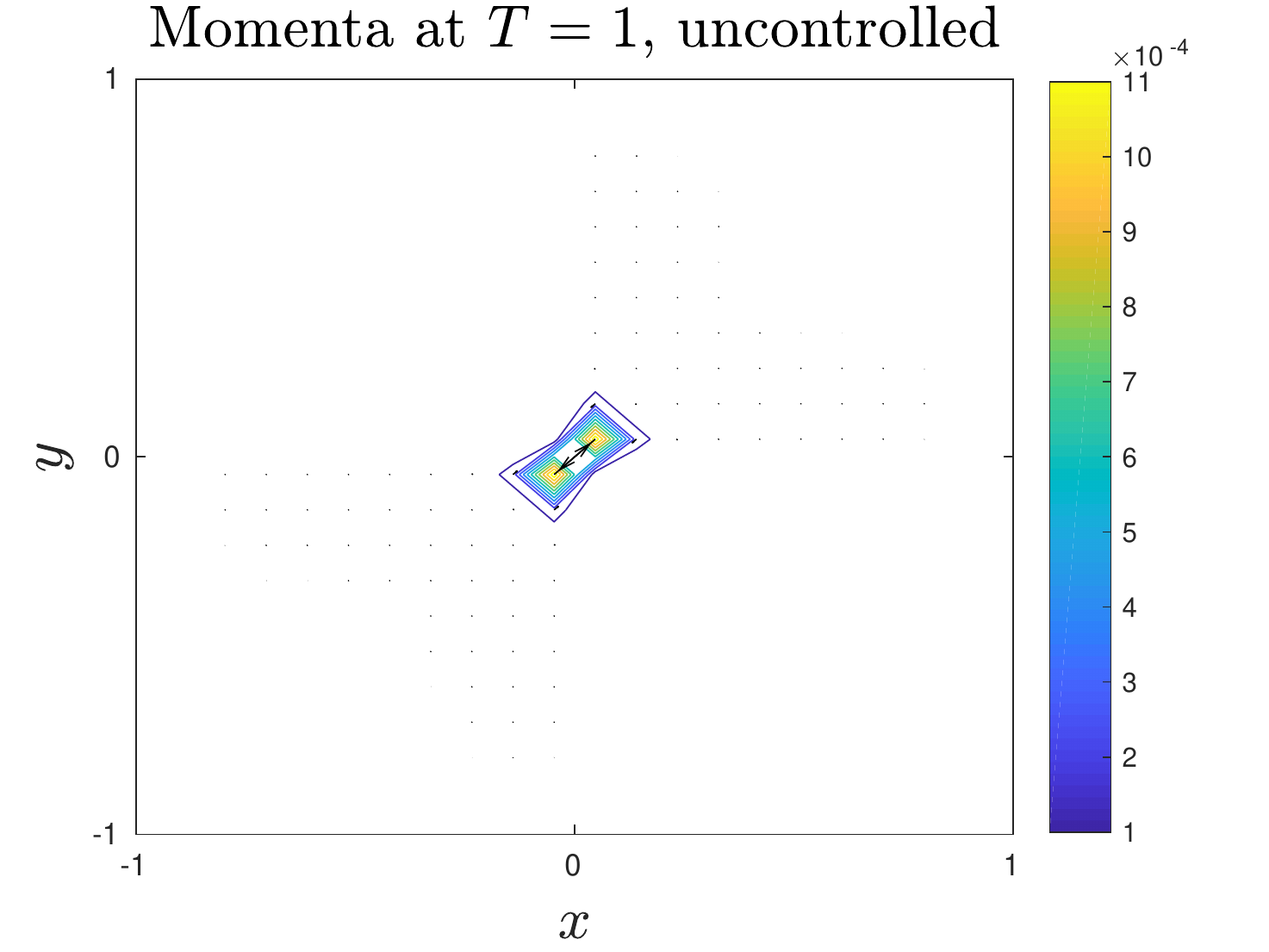}\hfill
			\includegraphics[width=0.34\textwidth]{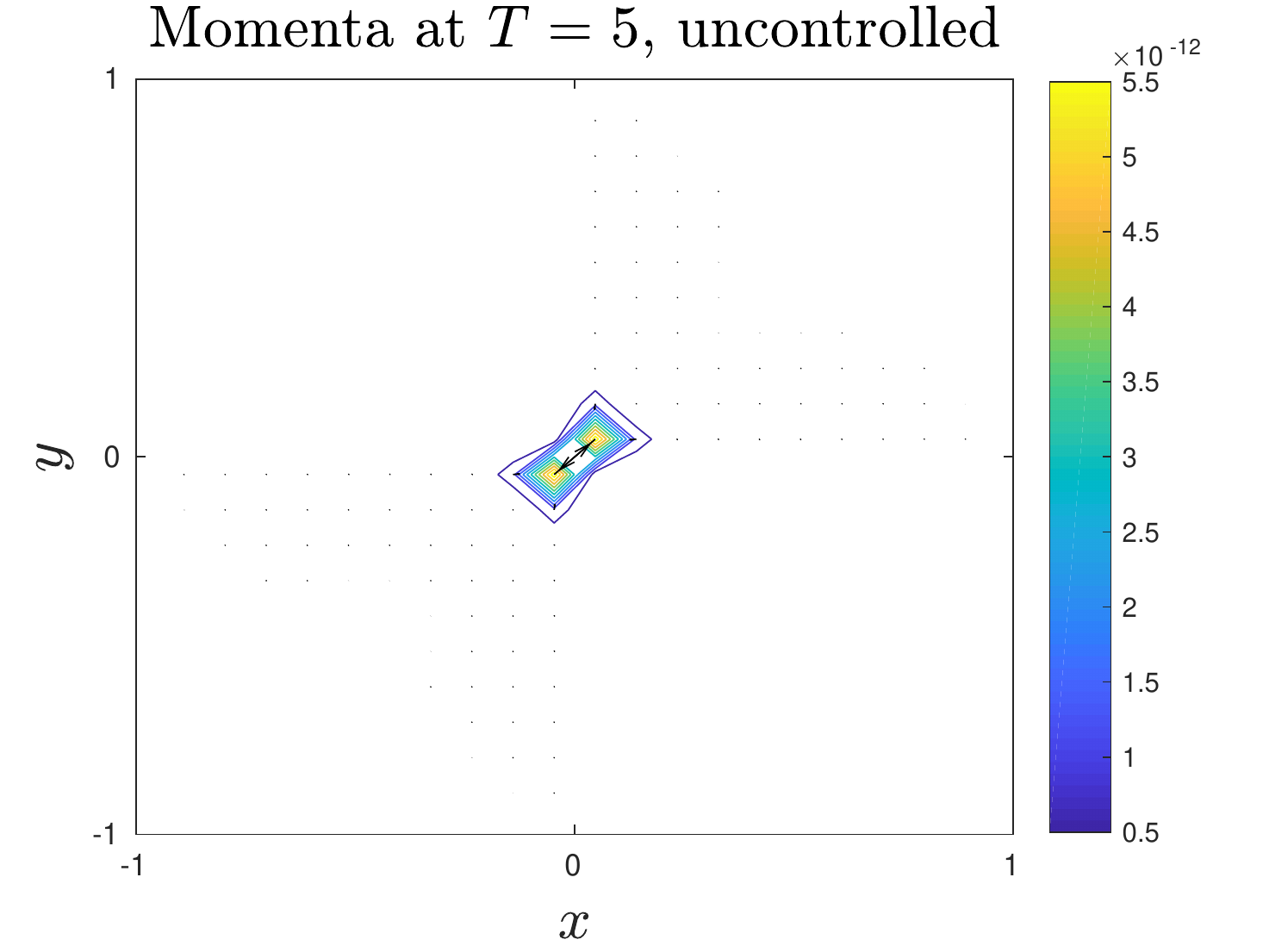}\\
			\hline
		\end{tabular}
		\caption{(Test 2D: Uncontrolled symmetric heaps). Alignment in the two-dimensional case for initial data \eqref{s2Din}. Evolution of the density is reported in the first row, whereas the second row shows the velocity $u$, whereas the corresponding momenta $\rho u$ is reported.} \label{fig:6}
	\end{figure} 
\paragraph{Test 2D: Asymmetric heaps.}
We consider the previous example, where the initial data \eqref{s2Din} accounts now at time zero a biased density $\rho$, as follows
\begin{align}\label{ID2d1}
\rho_0(x,y) =\rho^+(x,y) + 2\rho^-(x,y), \quad u_0(x,y) = (2H(x)-1,2H(y)-1)^T.
\end{align}
For comparison, we first consider the uncontrolled case. In Figure \ref{fig:4}, we observe that the mass at the initial velocity $ (1,1)^T$ assigns a larger momentum to the cluster at the bottom left. This larger momentum dominates the Cucker-Smale dynamics and hence the resulting alignment of the global mass-distribution follows this larger initial momentum (see the momenta at $T=1$ ). Furthermore, we see that the initial velocity field has an influence at $x=0$ or $y=0$. If the support of $\rho$ is transported to this area, the mass is accelerated and clusters in the center $ (0,0)^T$. 
\begin{figure}[t]
	\centering
	\begin{tabular}{@{}c@{\hspace{1mm}}c@{\hspace{1mm}}c@{\hspace{1mm}}c@{}}
	
		\includegraphics[width=0.34\textwidth]{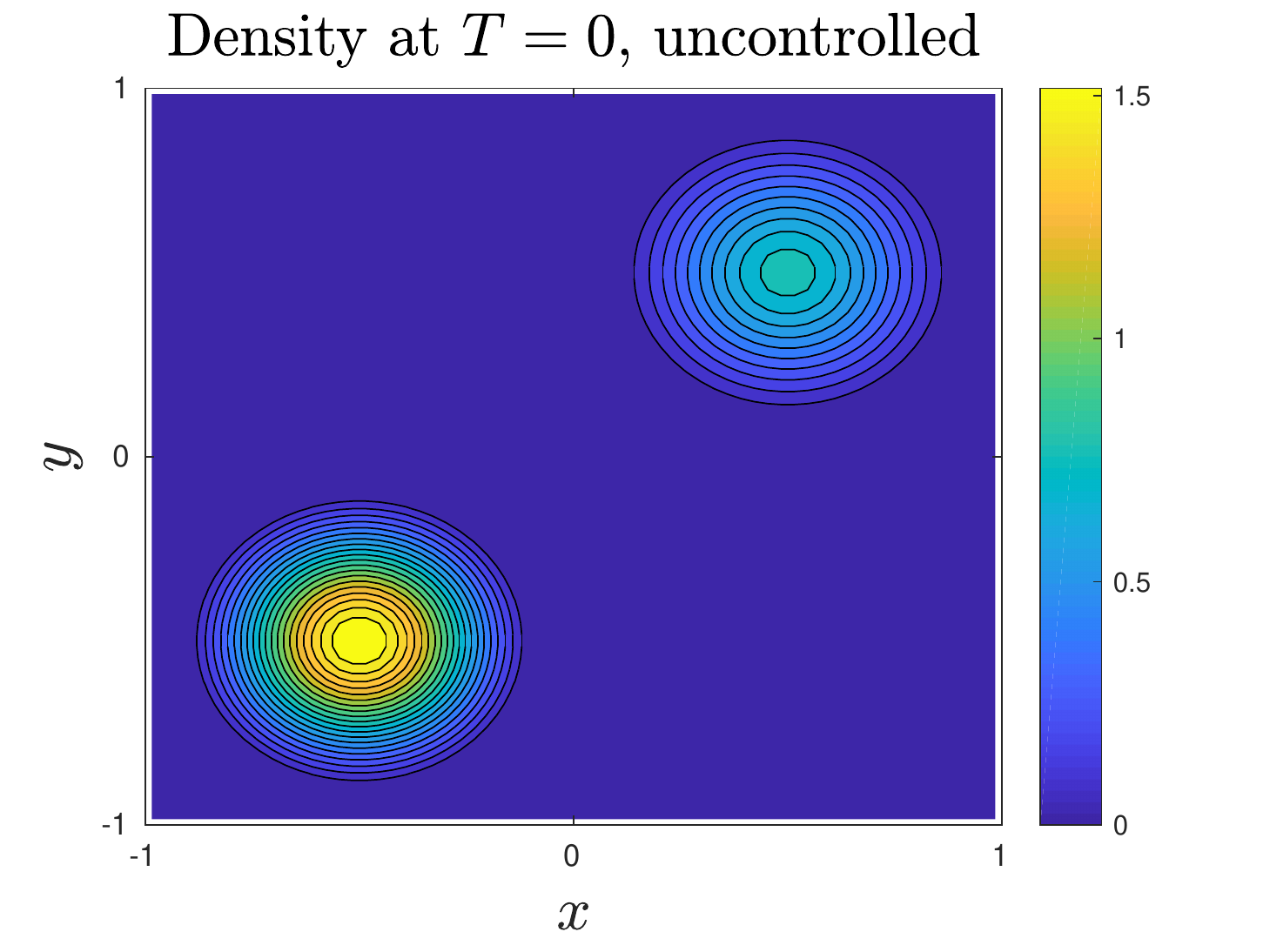}\hfill
		\includegraphics[width=0.34\textwidth]{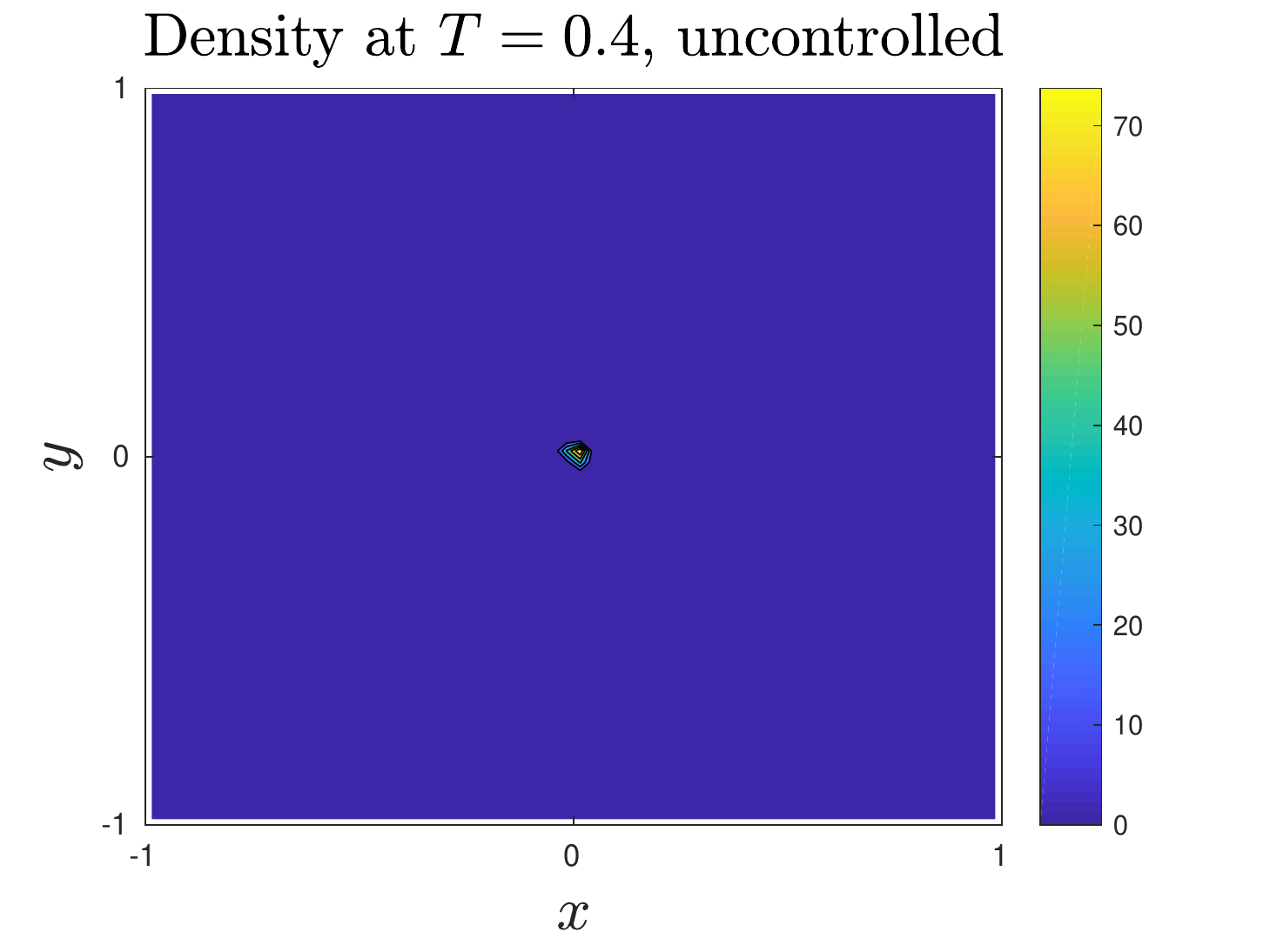}\hfill
		\includegraphics[width=0.34\textwidth]{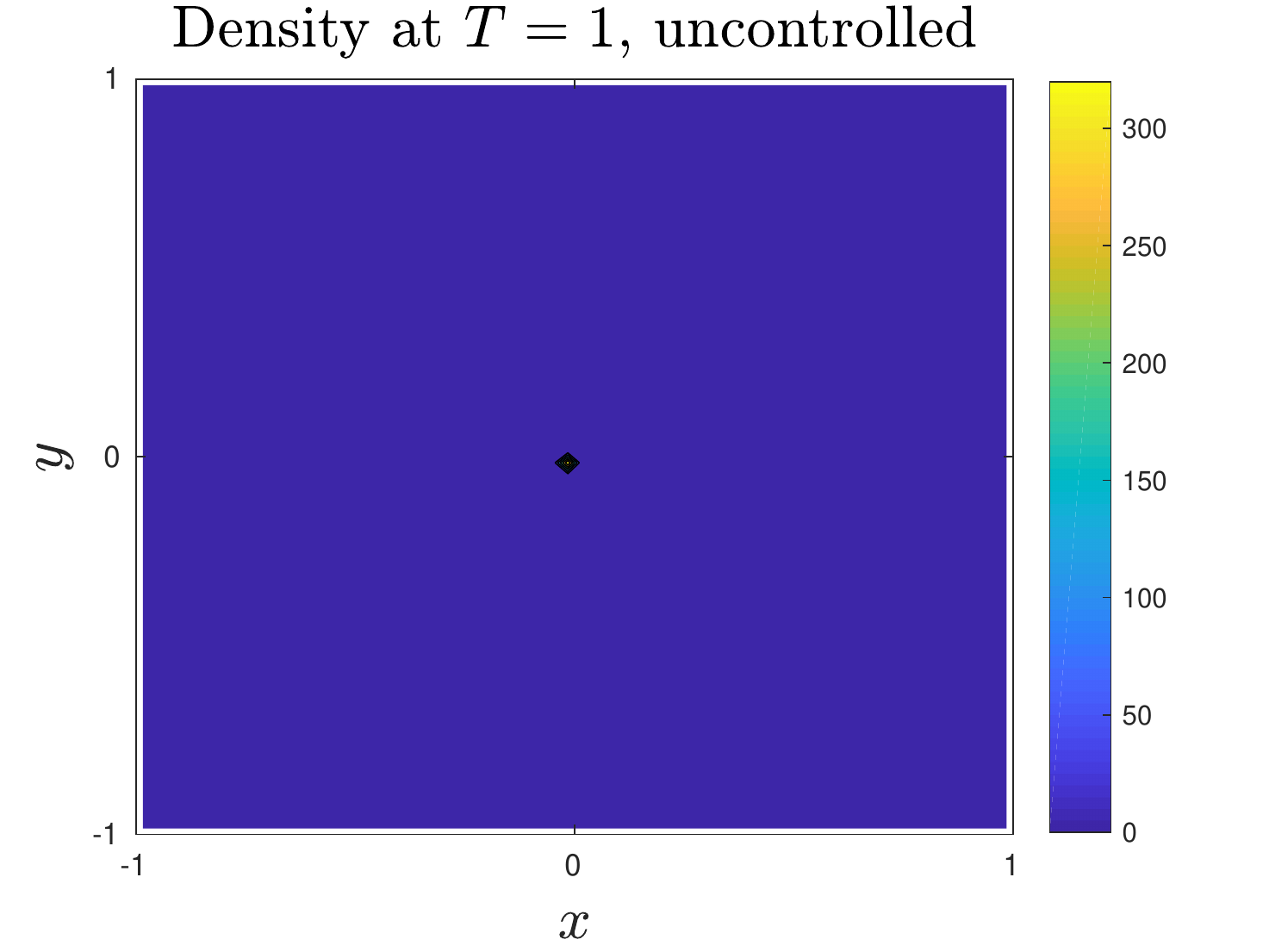}\\
		\hline\\
		\includegraphics[width=0.34\textwidth]{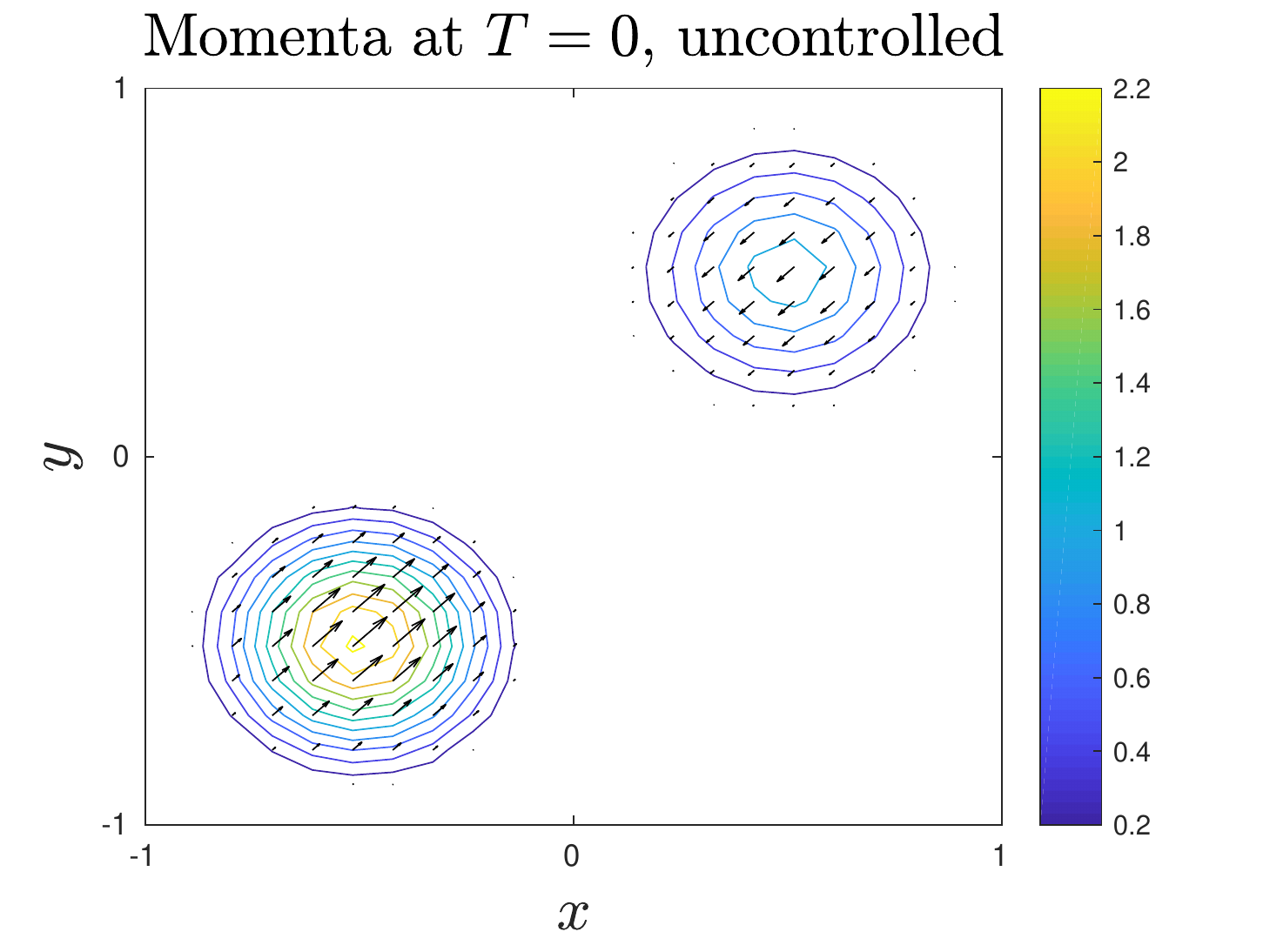}\hfill
		\includegraphics[width=0.34\textwidth]{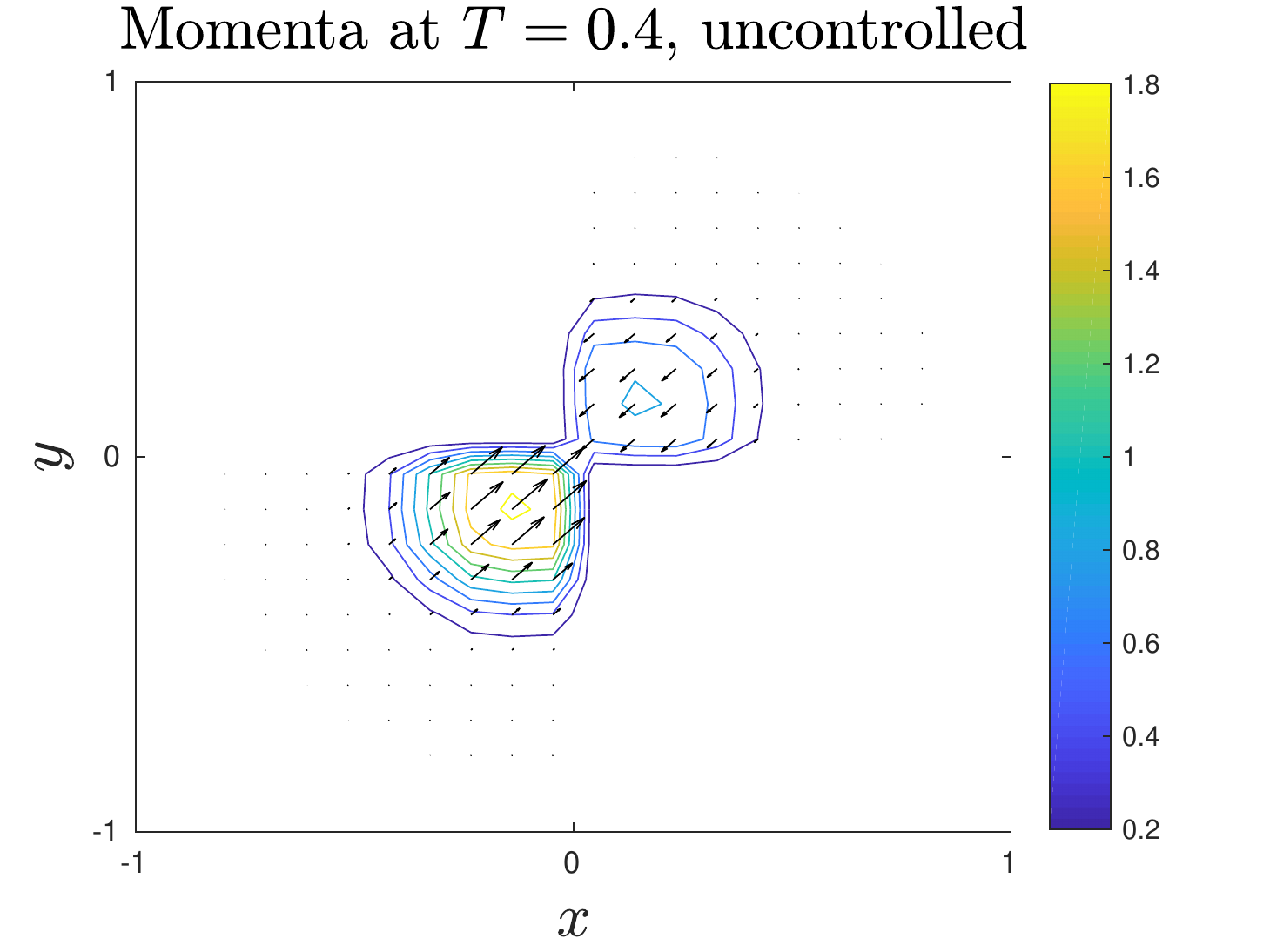}\hfill
		\includegraphics[width=0.34\textwidth]{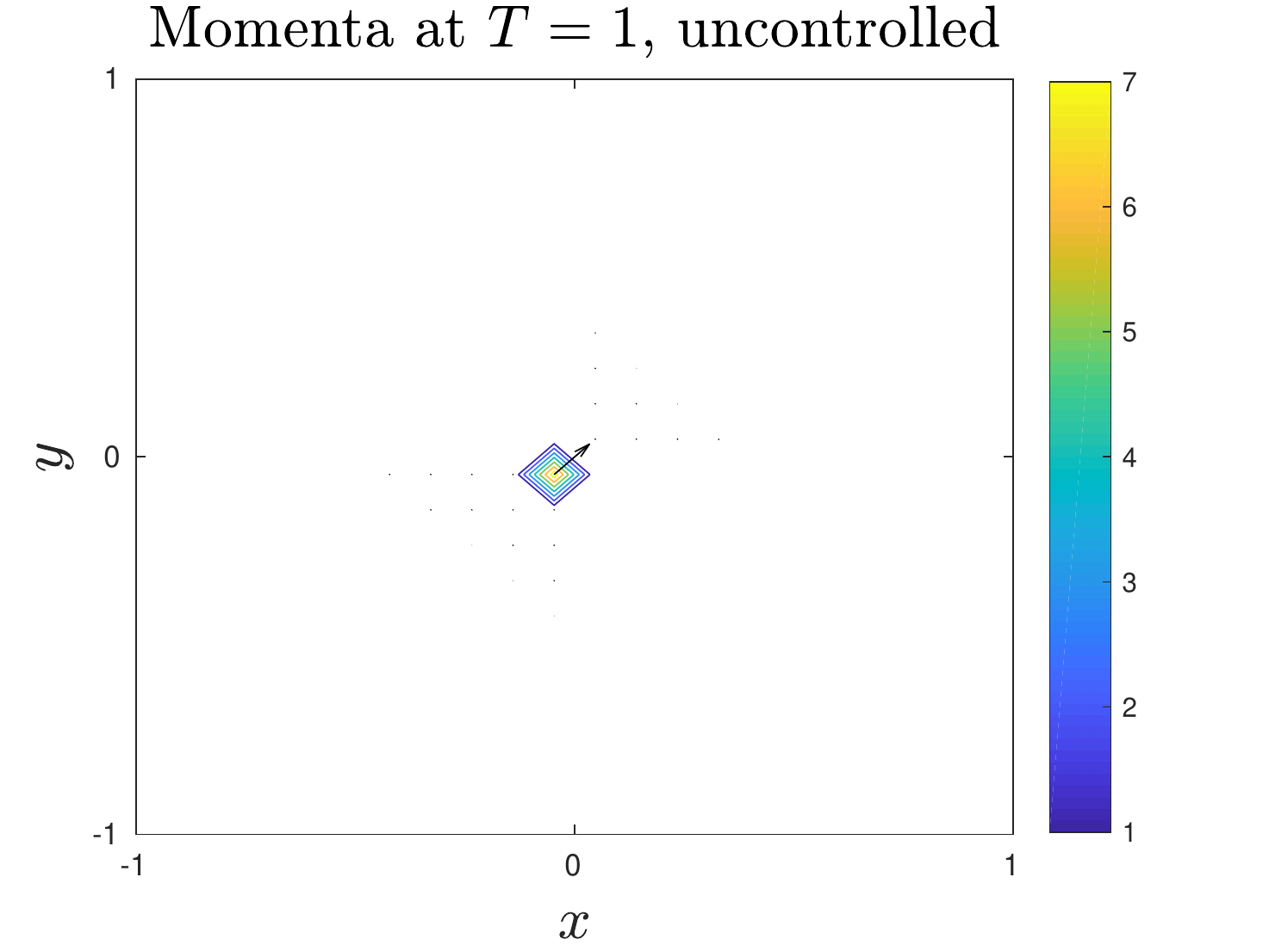}\\
		\hline
	\end{tabular}
	\caption{(Test 2D: Uncontrolled asymmetric heaps). Uncontrolled case to initial data \eqref{ID2d1}. The vector field of the momenta and velocities ($\in \R^2$) is displayed as a quiver plot. Density as contour plot.} \label{fig:4}
\end{figure}

On the other hand, we study the action of the control for different values of $\gamma$ and desired velocity field $\bar u(x,y)=(0,0)^T$.
We show the evolution in Figure \ref{fig:5} the evolution of density and momenta, where we observe in the first column that in case of a control parameter $\gamma=10$ the evolution is similar as in Figure \ref{fig:4}, but  with less high concentration at the center, since by the chosen desired velocity is $\bar{u}=(0,0)^T$ and hence, is a {\it damping term}. In the second column, for $ \gamma=1$, we have an even less concentrated profile and lower momenta that in the previous case. In case of $ \gamma=0.1$ we do almost see no evolution at all and almost no remaining momenta. 
\begin{figure}[t]
	\centering
	\begin{tabular}{@{}c@{\hspace{1mm}}c@{\hspace{1mm}}c@{\hspace{1mm}}c@{}}
		\hline
		\\
		\includegraphics[width=0.34\textwidth]{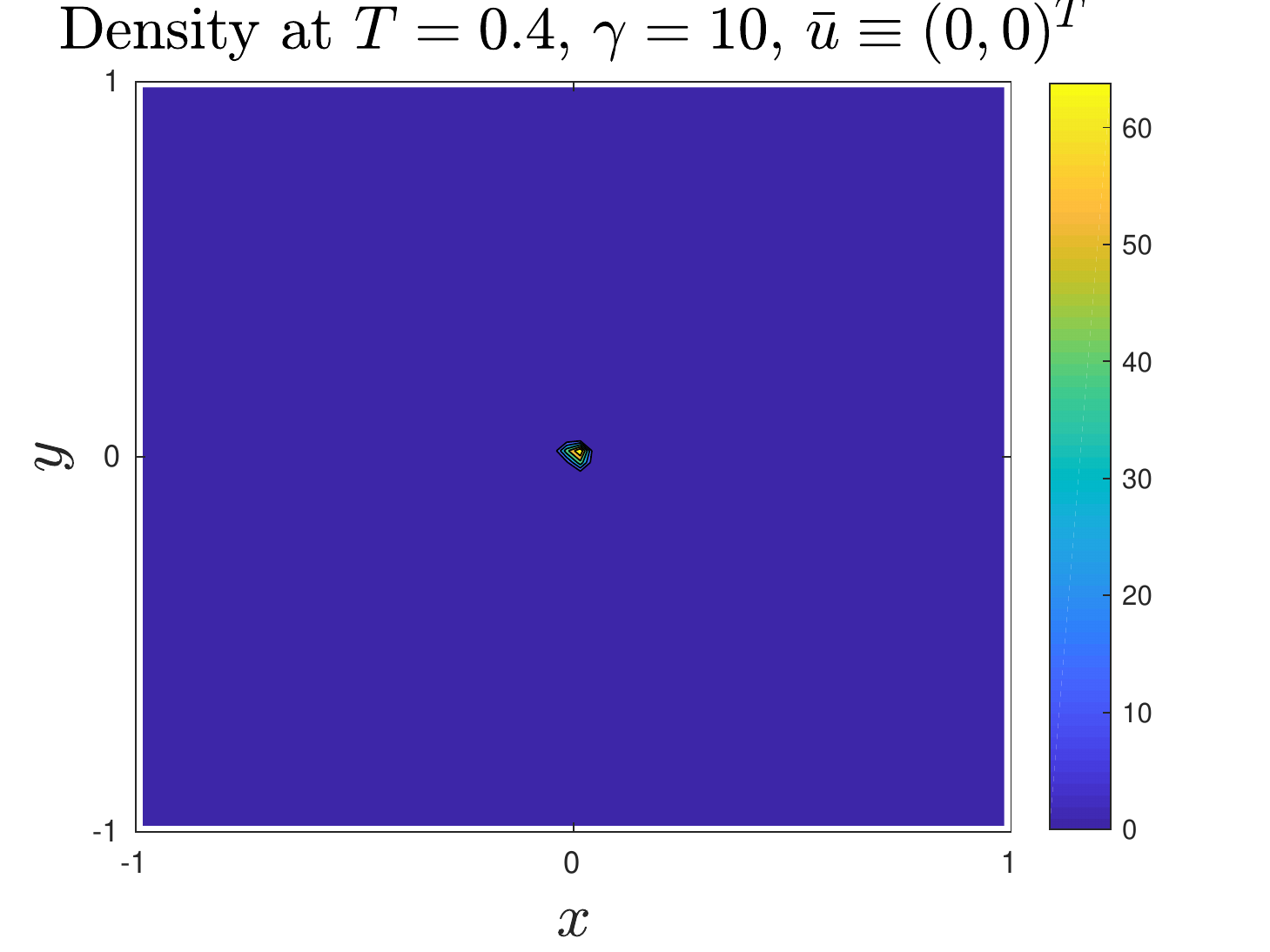}\hfill
		\includegraphics[width=0.34\textwidth]{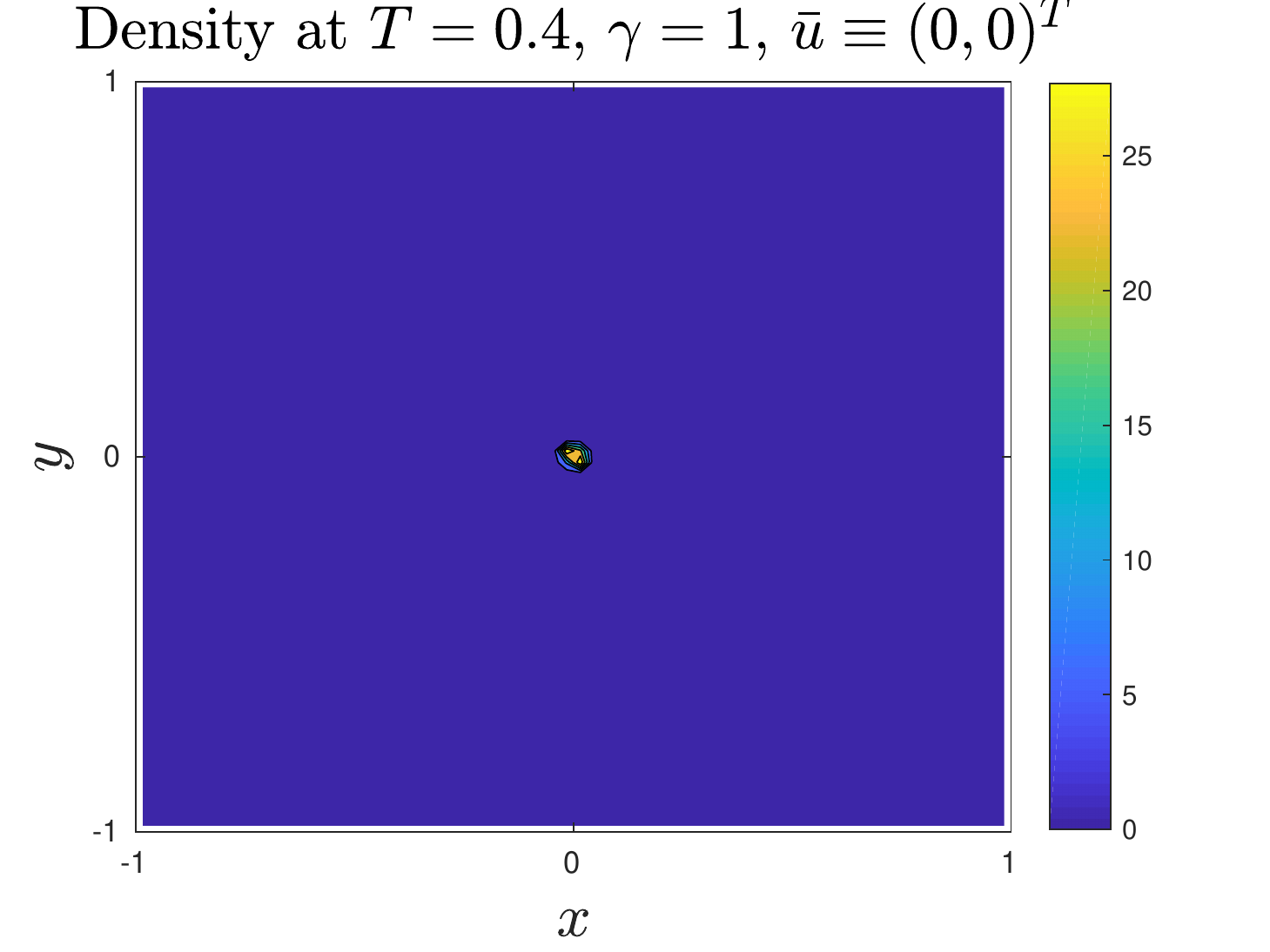}\hfill
		\includegraphics[width=0.34\textwidth]{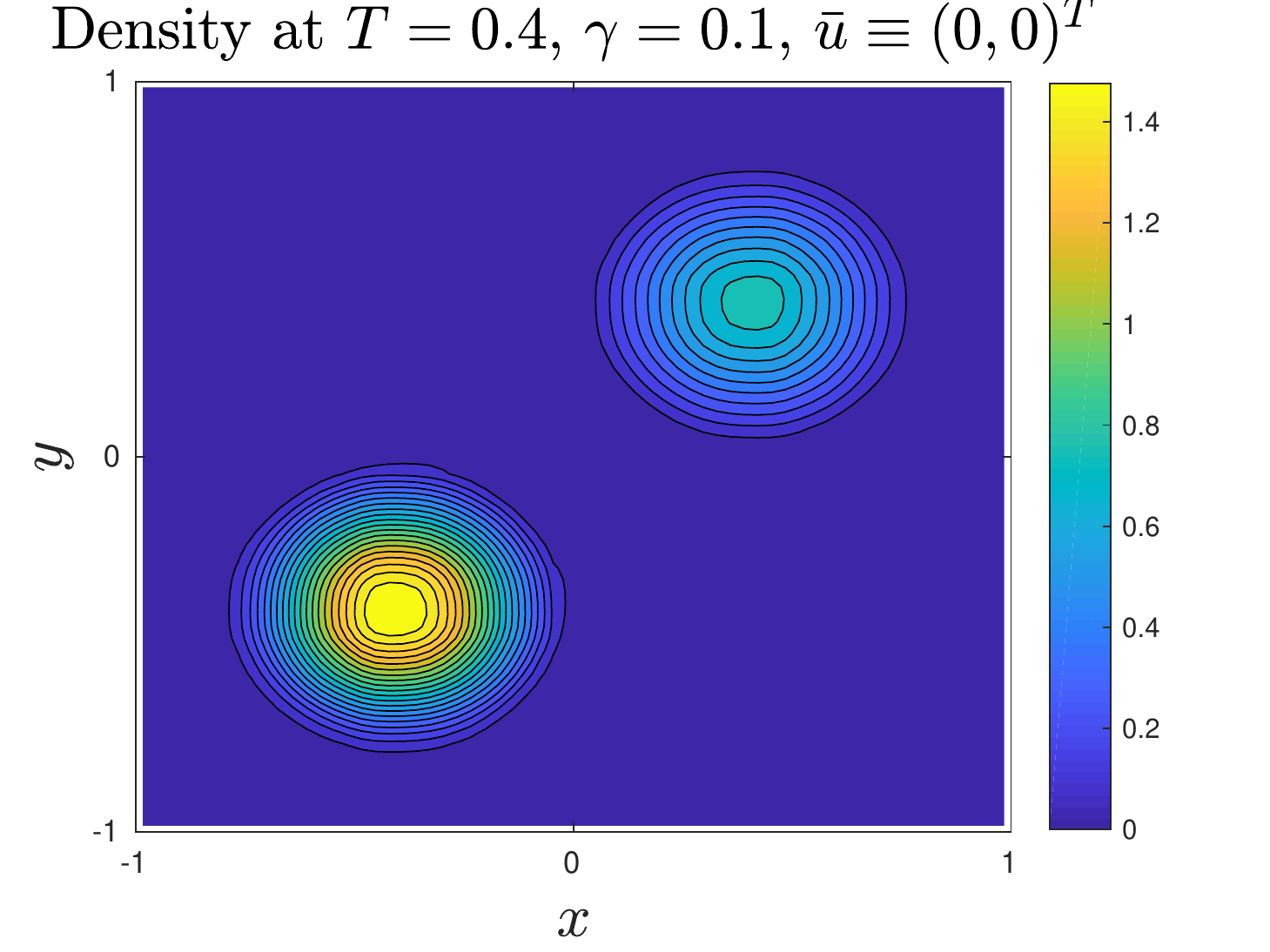}
		\\
		\hline\\   
		\hline\\
		\includegraphics[width=0.34\textwidth]{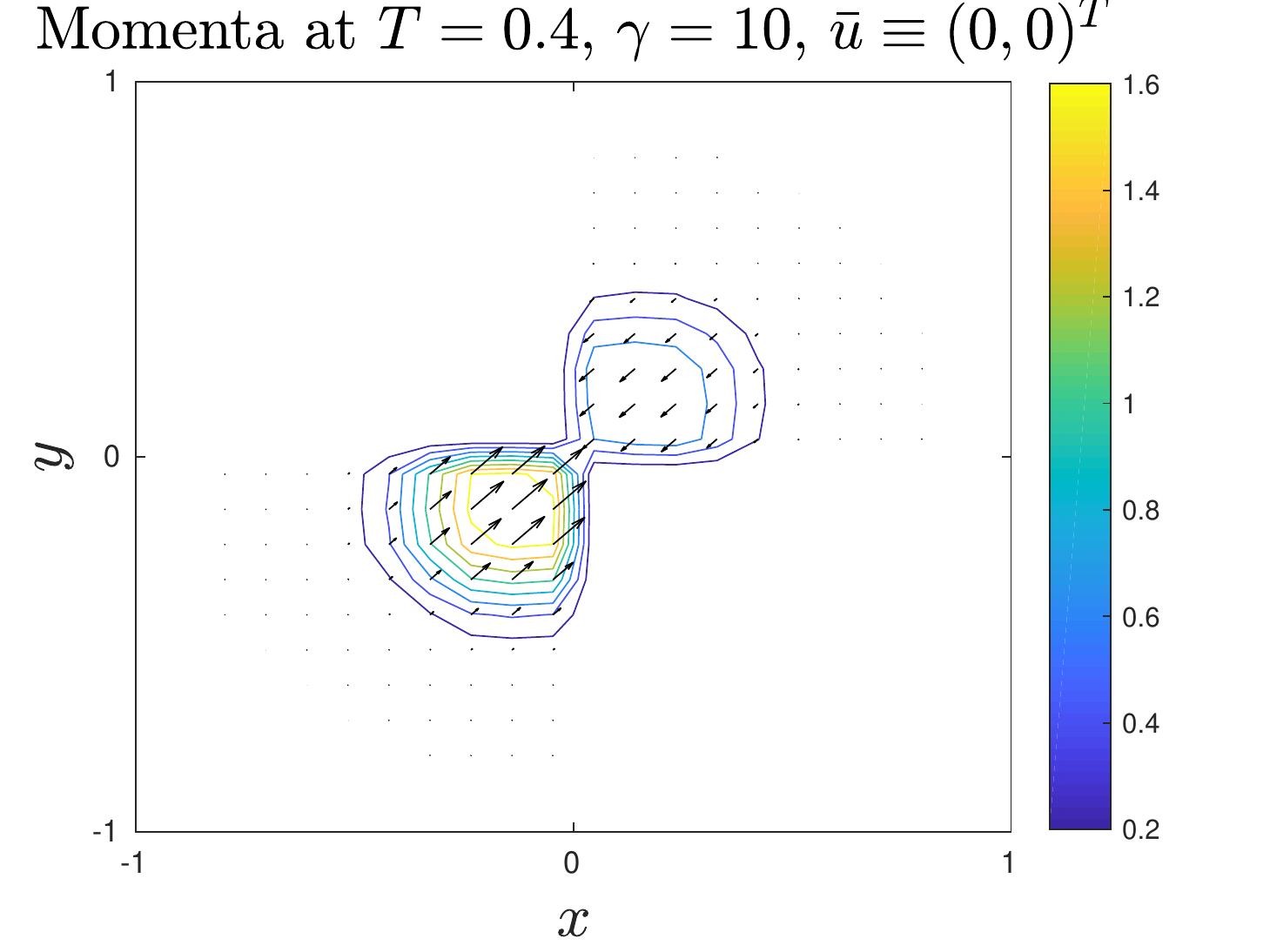}\hfill
		\includegraphics[width=0.34\textwidth]{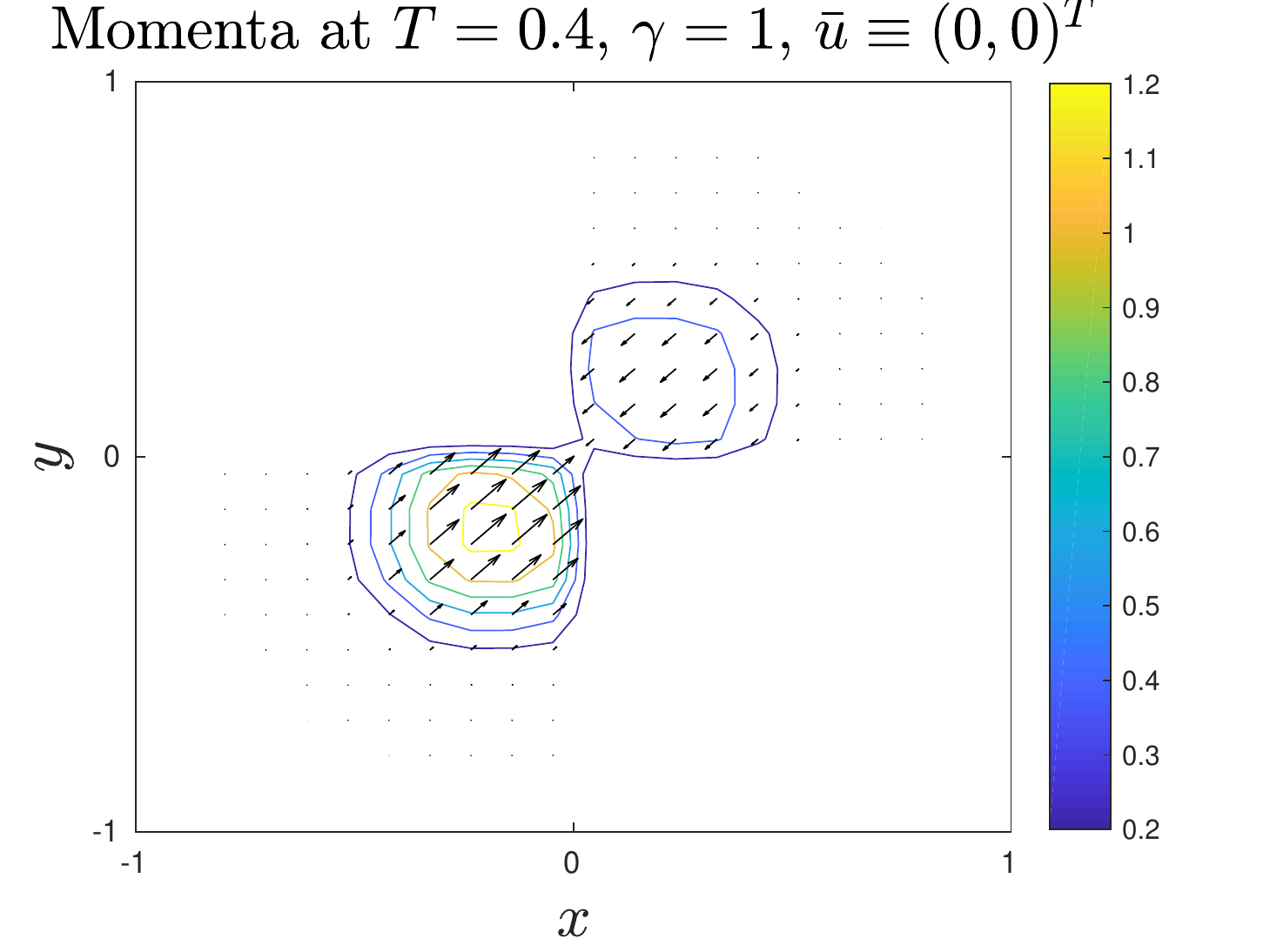}\hfill
		\includegraphics[width=0.34\textwidth]{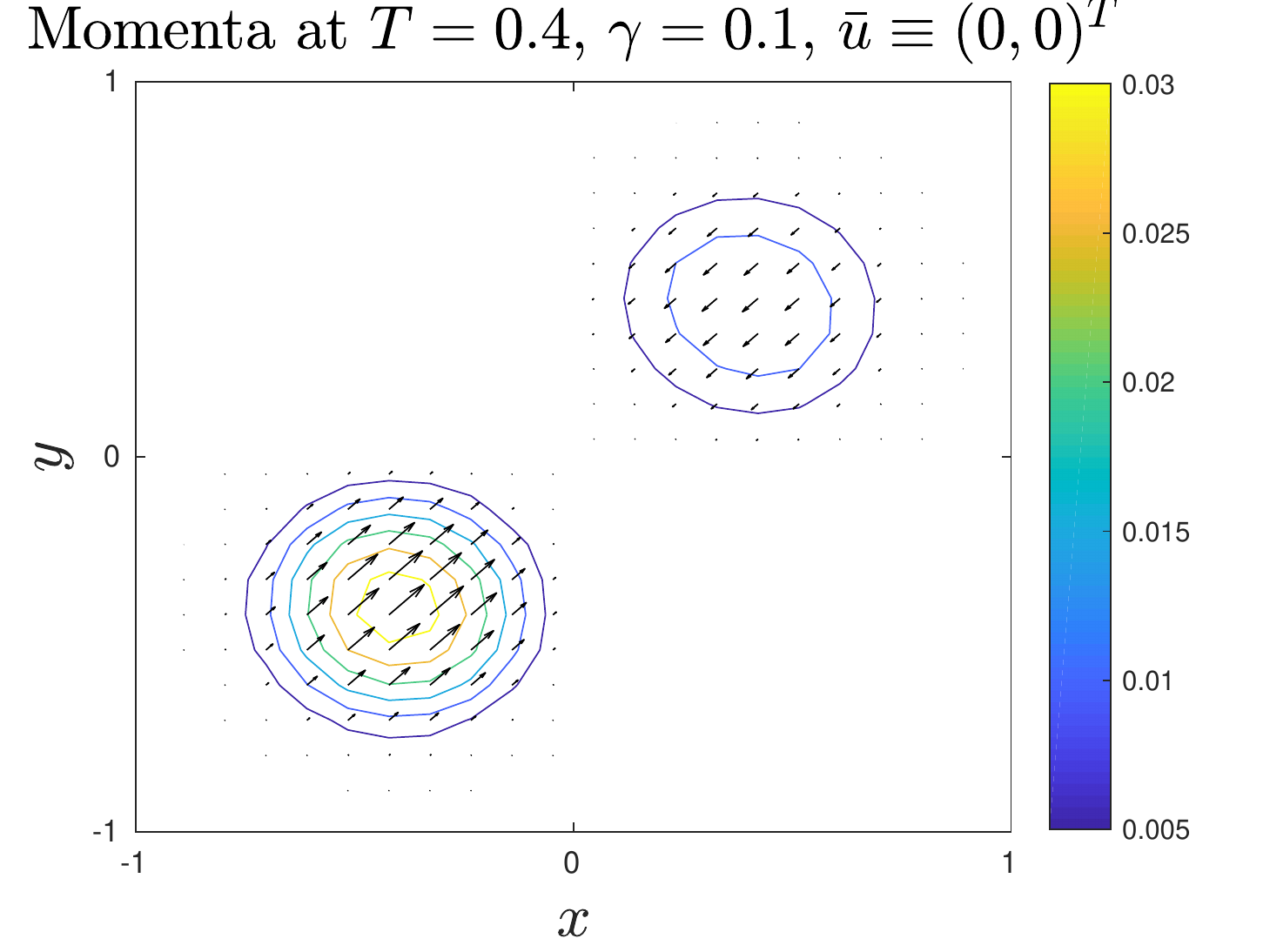}\\
		\hline\\
	\end{tabular}
	\caption{(Test 2D: Controlled asymmetric heaps). Controlled case to initial data \eqref{ID2d1} with $\bar{u}=(0,0)^T$. First row show the evolution of the density at time $T=0.4$, for  $\gamma=10,1,0.1$, respectively from left to right. Bottom line shows the evolution of the momenta $\rho u$.} \label{fig:5}
\end{figure}


\paragraph{Test 2D: Reorientation.}
We consider an aligned flock, and we want to use the control function $\phi$ in order to invert its direction. Hence we define the following initial data,
\begin{align}\label{ID2d2}
\rho_0(x,y) \equiv 1,\quad u_0(x,y) \equiv (1,1)^T ,
\end{align}
with desired velocity field
$\bar{u}(x,y) \equiv (-1,-1)^T$ .
In Figure \ref{fig:7} we report  the evolution of the system, where we display the quiver plots of the momenta, and we overlap its representation with the 2-norm of the velocity field $\|u(x,y)\|$.\\
We see, that the momenta field  changes the direction after some time (depending on the control parameter $\gamma$). However, the new transport direction does not instantly assume the full "speed" in the new direction $(-1,-1)^T$. The additional information on the magnitude of the velocity field shows that at time $T=10$, for $\gamma=10$ the 2-norm of the velocity field is $\|u\| \simeq 0.3737$, whereas for $\gamma=0.1$ the 2-norm approaches the expected value of $\sqrt{2}$. 
\begin{figure}[t]
	\centering
	\begin{tabular}{@{}c@{\hspace{1mm}}c@{\hspace{1mm}}c@{\hspace{1mm}}c@{}}
		\hline
		\\
		\includegraphics[width=0.4\textwidth]{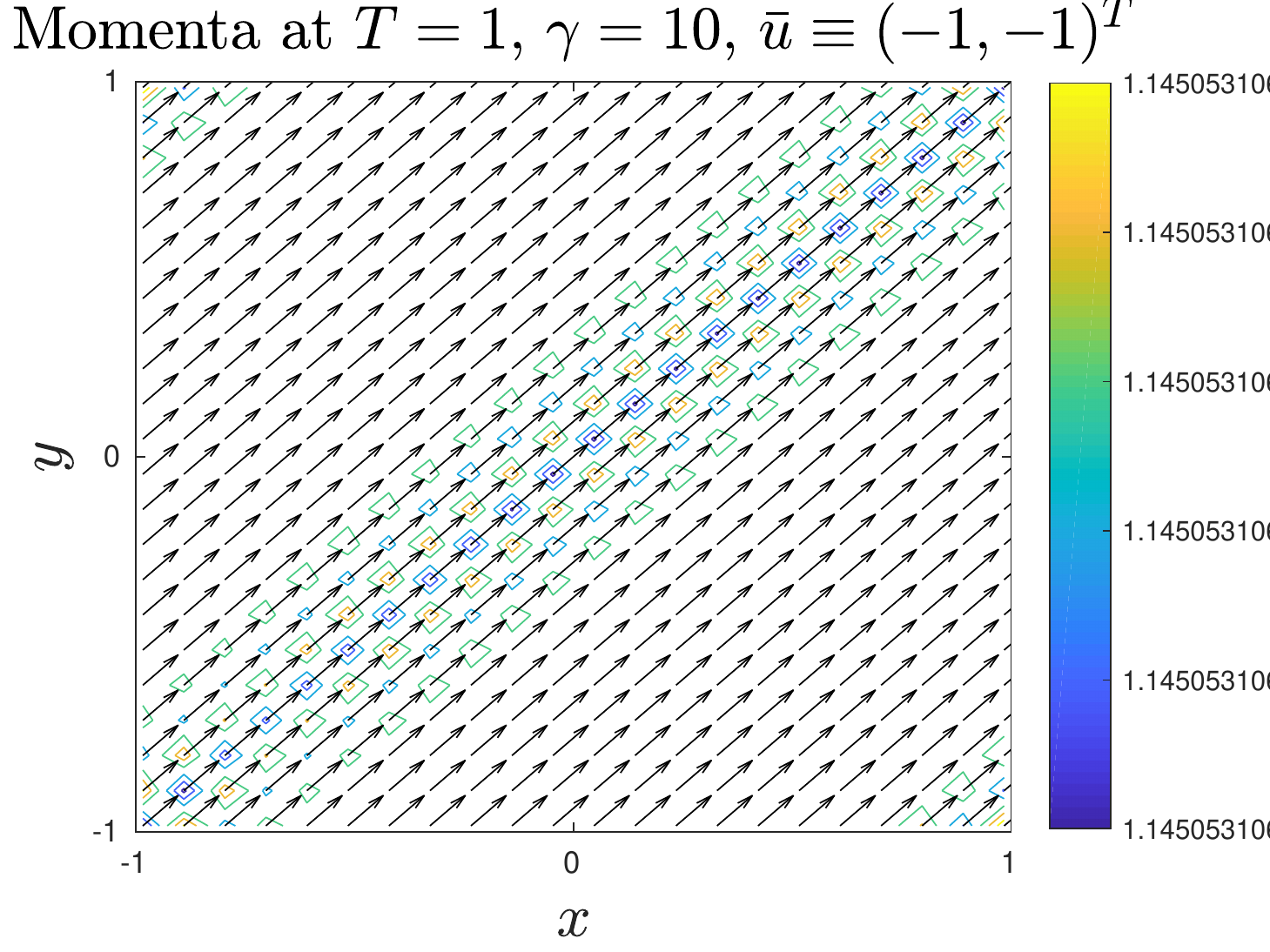}\hfill
		\includegraphics[width=0.4\textwidth]{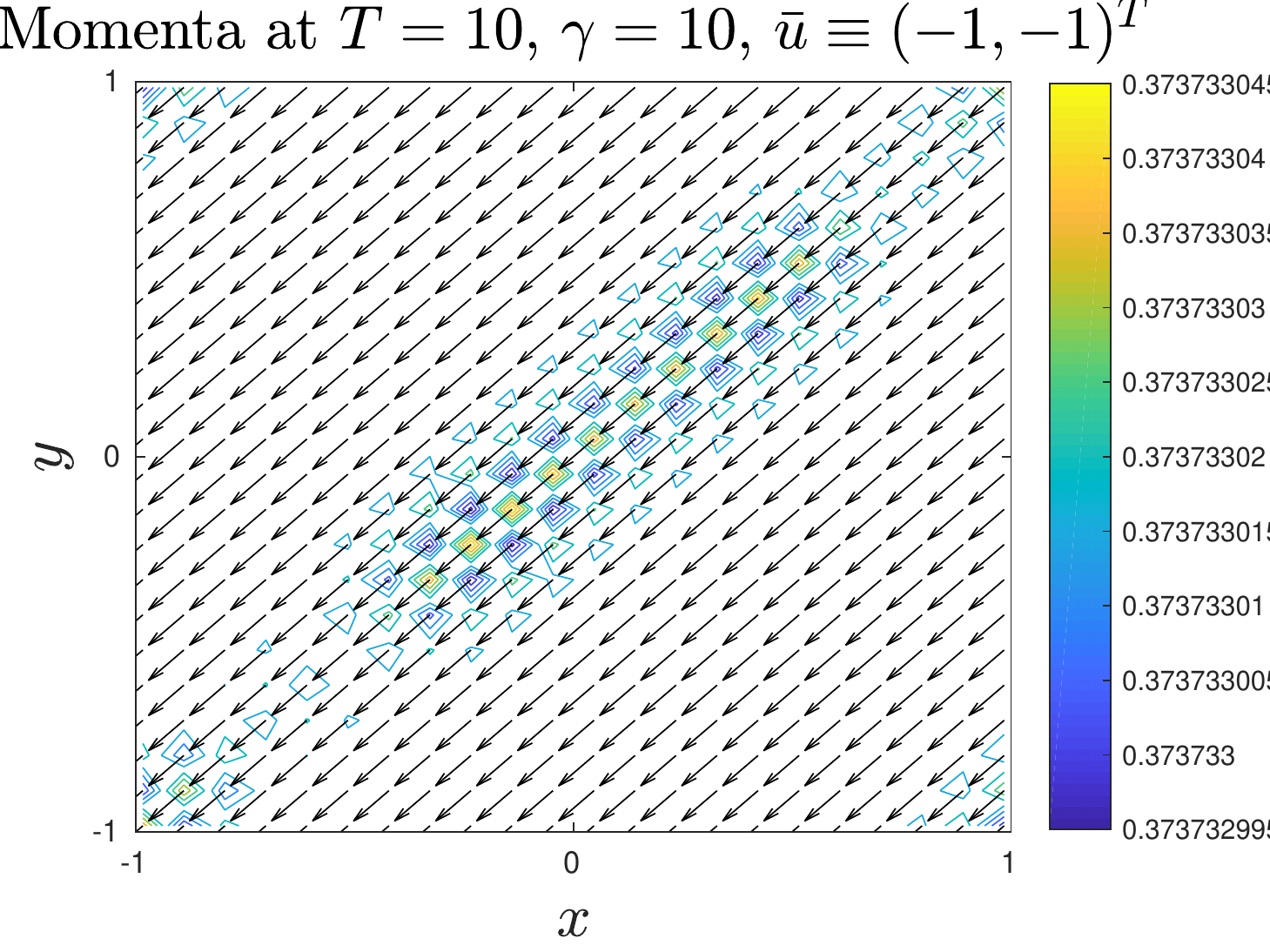}\\
		\hline\\   
		\hline\\	\includegraphics[width=0.4\textwidth]{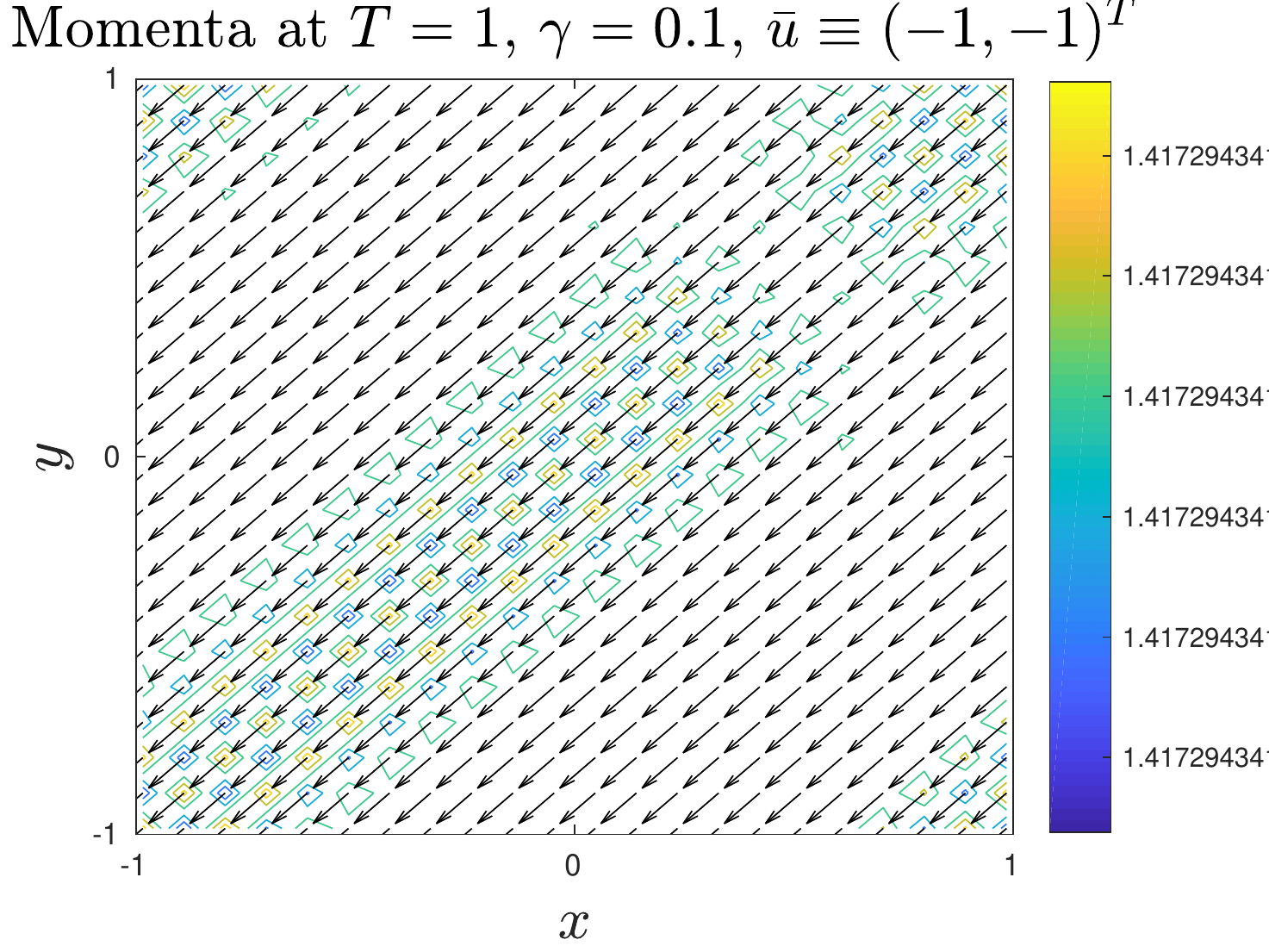}\hfill
		\includegraphics[width=0.4\textwidth]{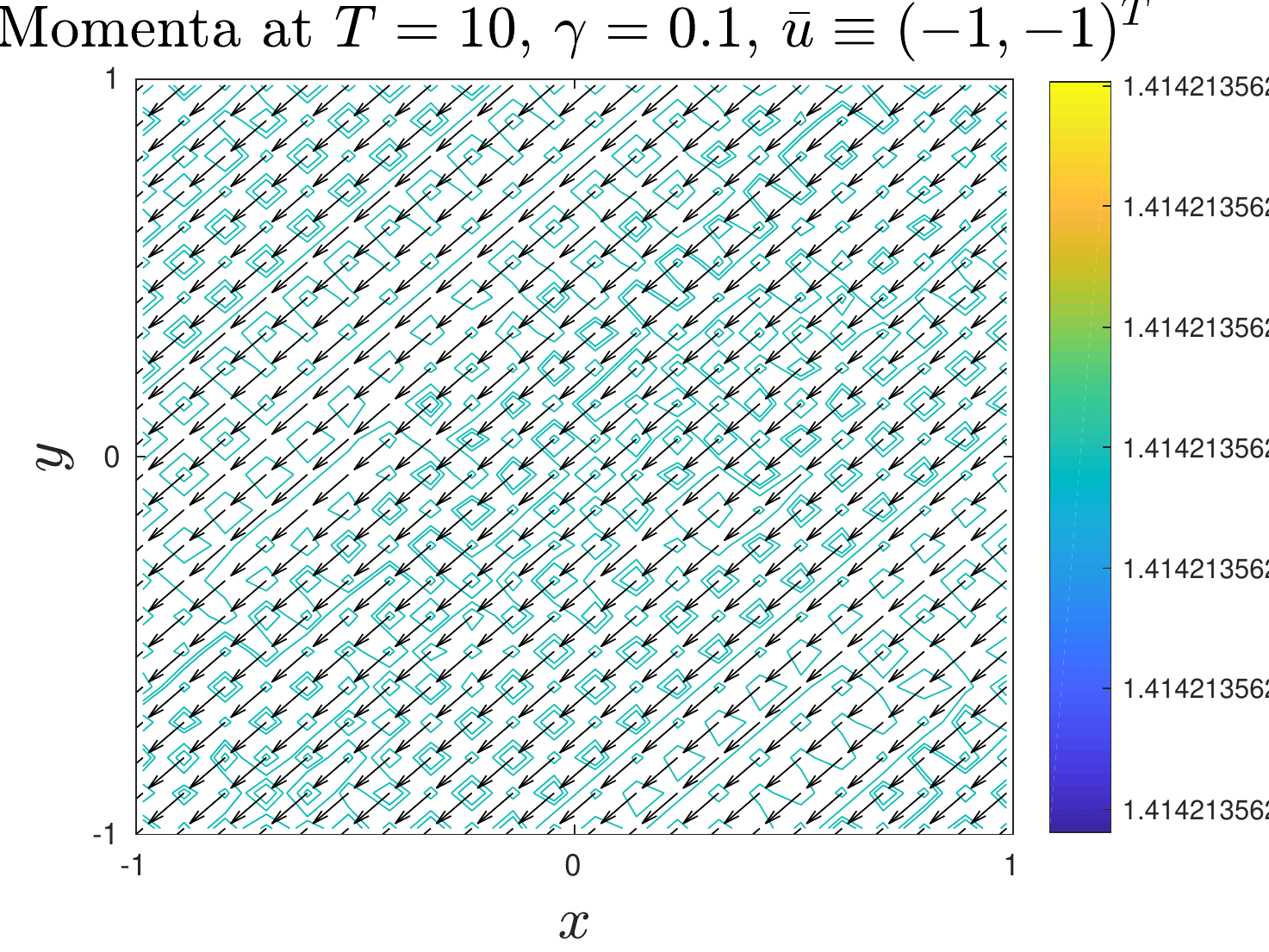}\\
		\hline
	\end{tabular}
	\caption{(Test 2D: Reorientation). Controlled case to initial data \eqref{ID2d2} with $\bar{u}=(-1,-1)^T$. Only the evolution of momenta for different $\gamma$ is displayed.} \label{fig:7}
\end{figure}

\paragraph{Test 2D: Birdcage.}
The following case models a flock, that is ``caught in a cage" by the control $\bar{u}.$ To this purpose we suggest the following initial data:
\begin{align}\label{ID2d3}
\rho_0(x,y) \equiv 1,\quad u_0(x,y) \equiv (1,1)^T ,
\end{align}
with the exerted desired velocity
\begin{align}\label{cont2d3}
\bar u(x,y) =& (2H(x)-1,2H(y)-1)^T \times  \left( 1-\chi_{\left[-\frac{L}{5},\frac{L}{5}\right]}(x) \right) \times \left(1\chi_{\left[-\frac{L}{5},\frac{L}{5}\right]}(y)\right) ,
\end{align}
that is displayed in Figure \ref{fig:8}. In Figure \ref{fig:8}, we observe that for a control parameter $\gamma=10$ we do almost not ``catch" any mass at the attractive center of the control $\bar{u}$ \eqref{ID2d3}, and the initial velocity remains dominating. By decreasing the value of $\gamma$, we observe how successively more mass gathers in the center square. \par
\begin{figure}[t]
	\centering
	\begin{tabular}{@{}c@{\hspace{1mm}}c@{\hspace{1mm}}c@{\hspace{1mm}}c@{}}
		\hline
		\\
		\includegraphics[width=0.34\textwidth]{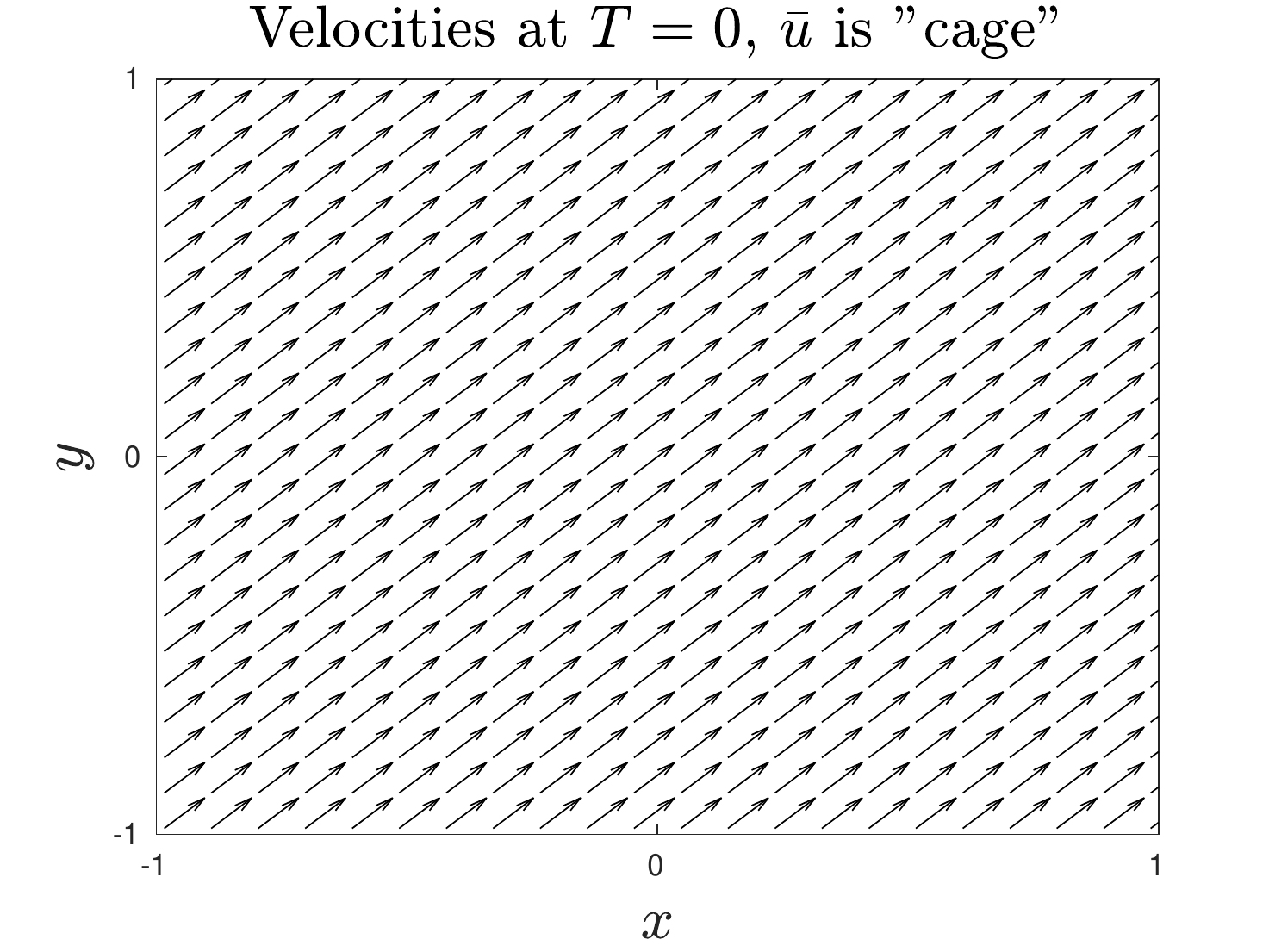}\hfill
		\hfill
		\includegraphics[width=0.34\textwidth]{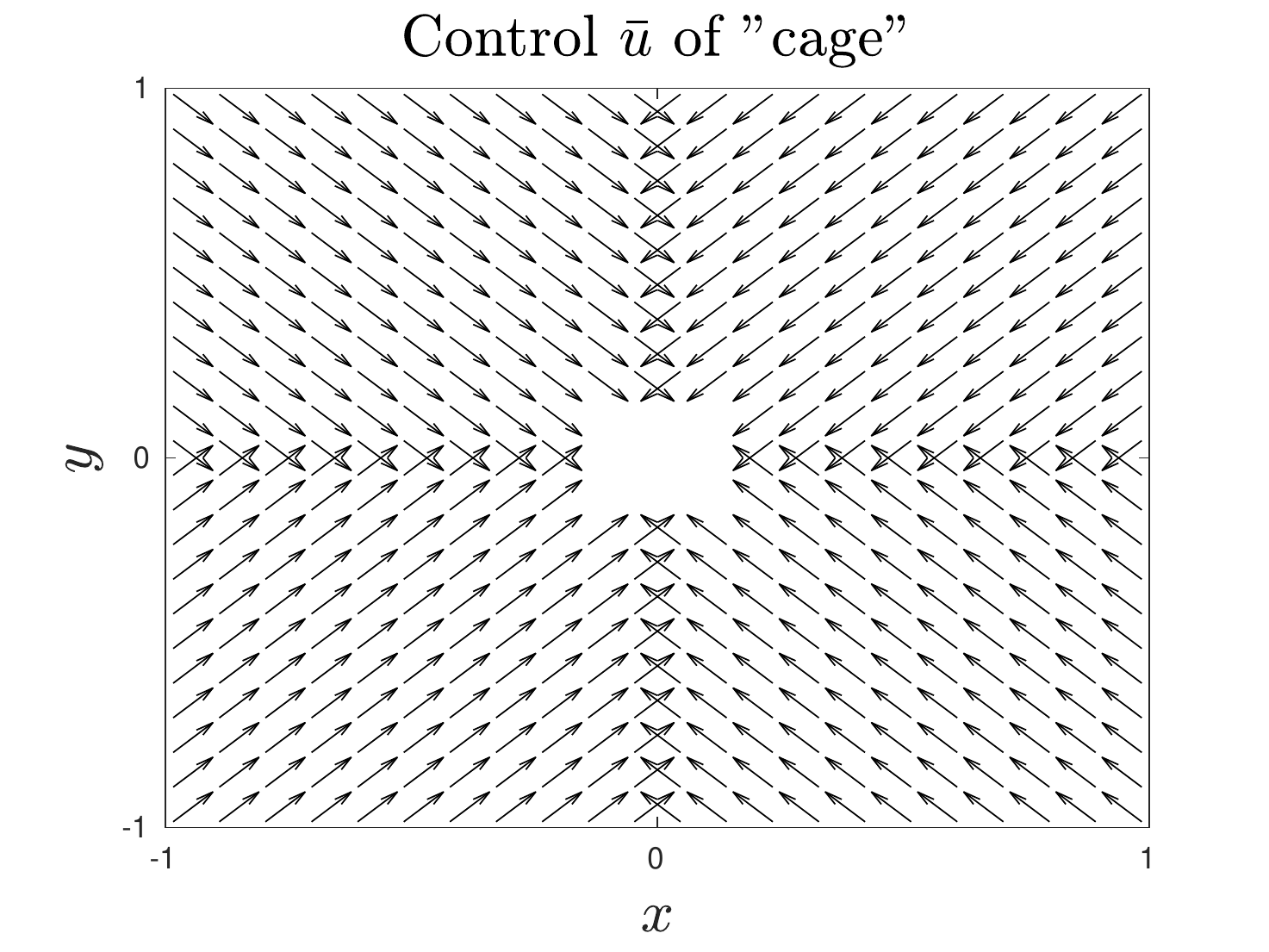}\\
		\hline\\   
		\hline\\
		\includegraphics[width=0.34\textwidth]{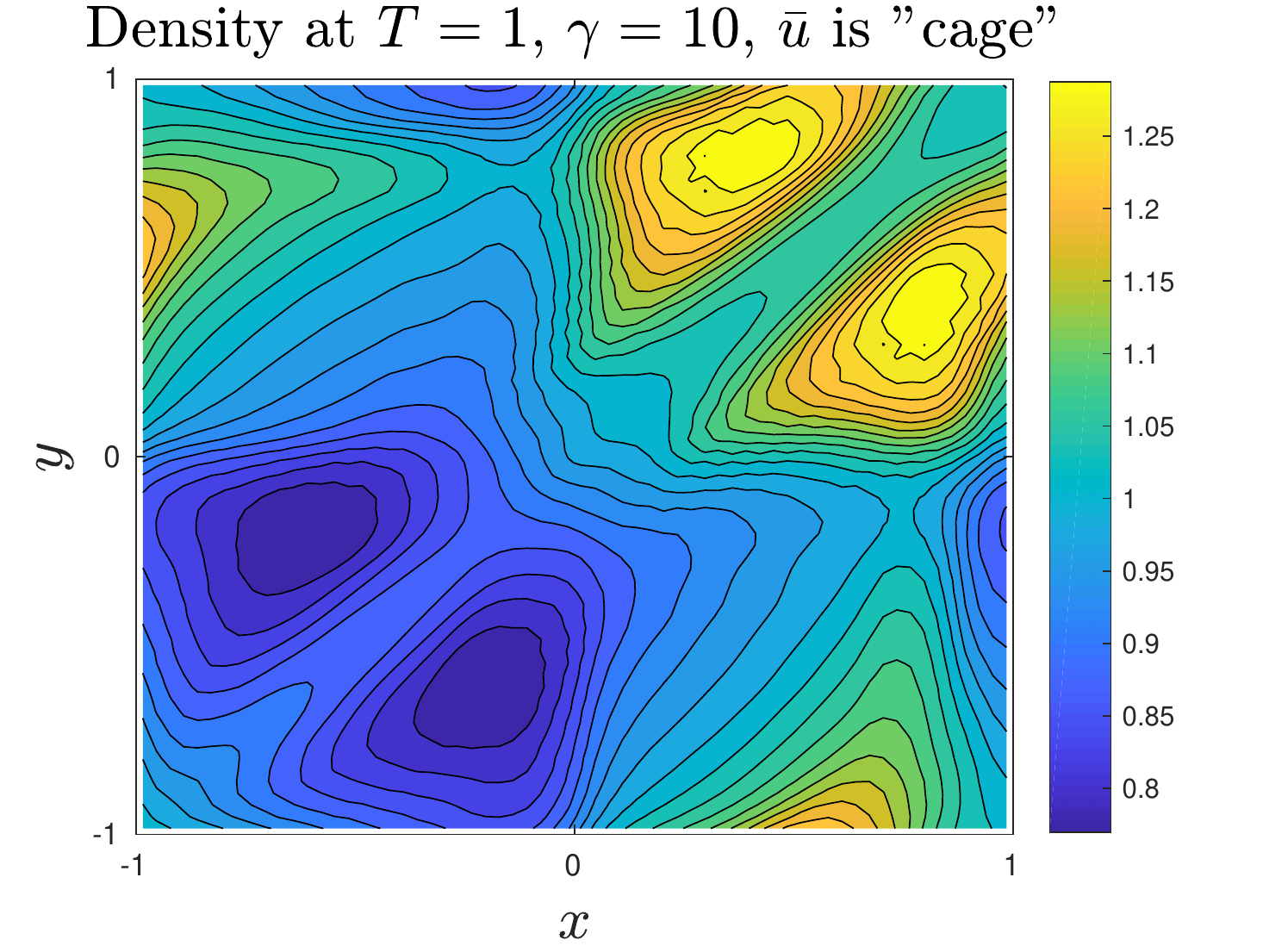}\hfill
		\includegraphics[width=0.34\textwidth]{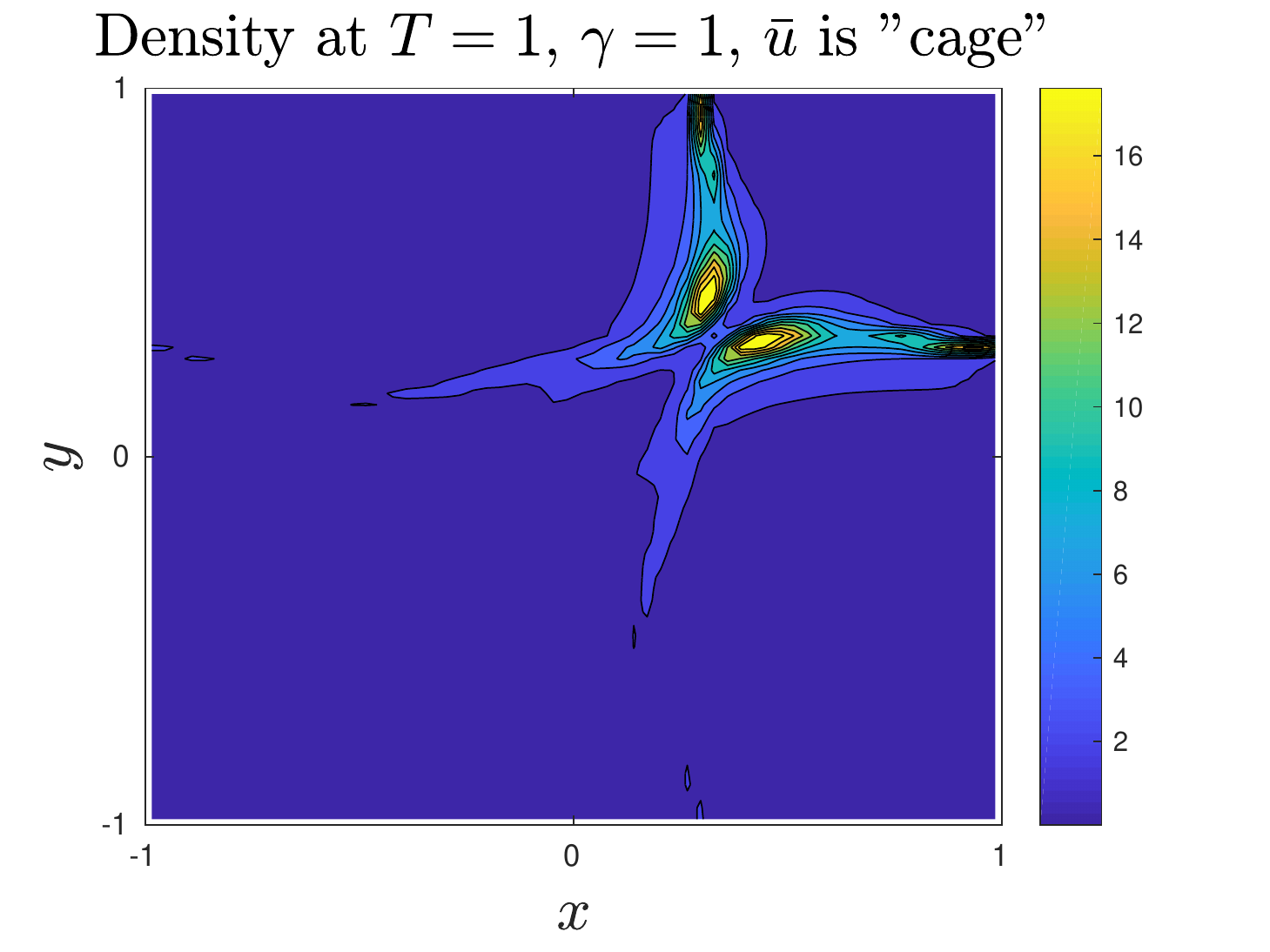}\hfill
		\includegraphics[width=0.34\textwidth]{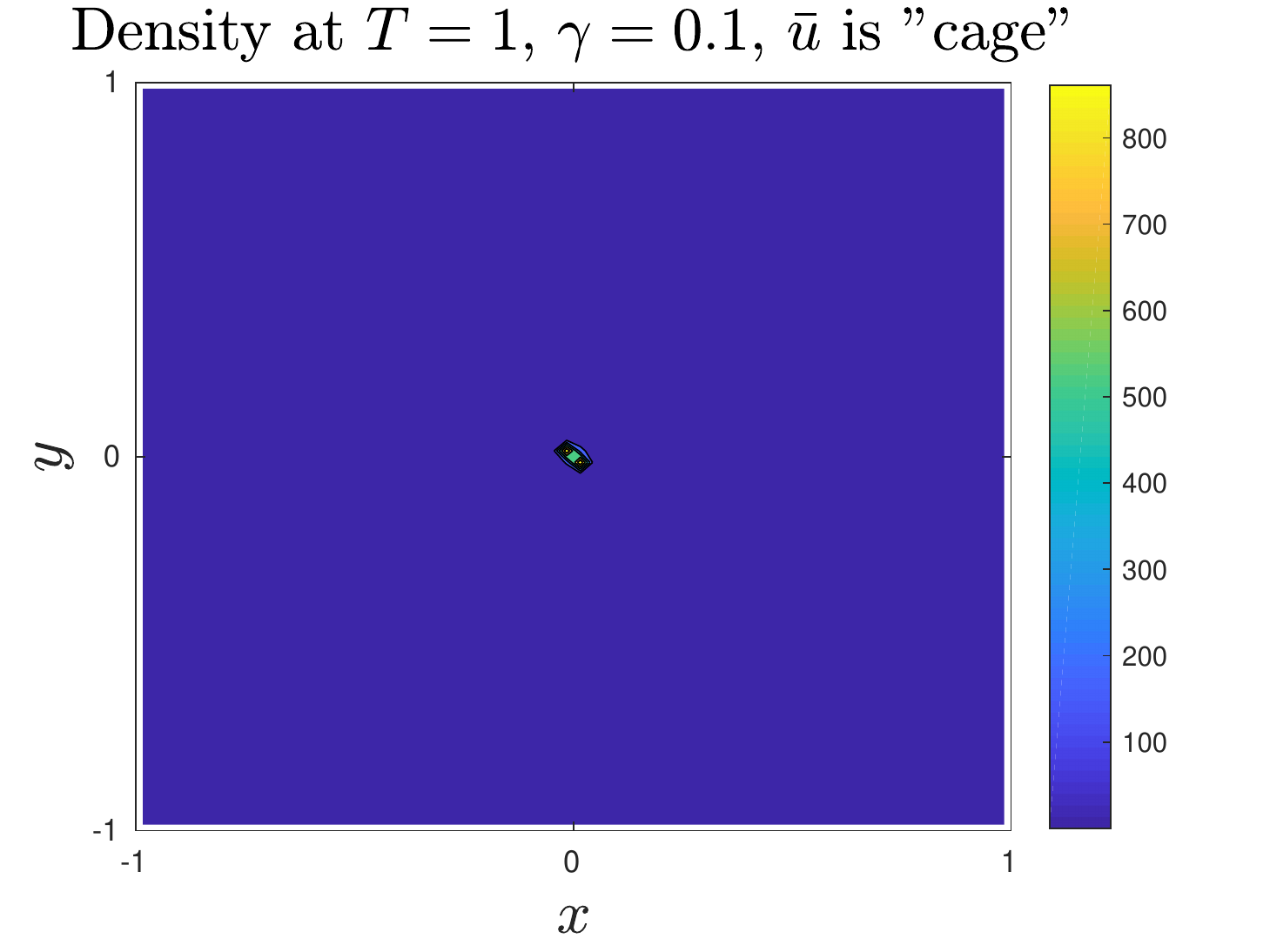}\\
		\hline\\
		\hline\\  
		\includegraphics[width=0.34\textwidth]{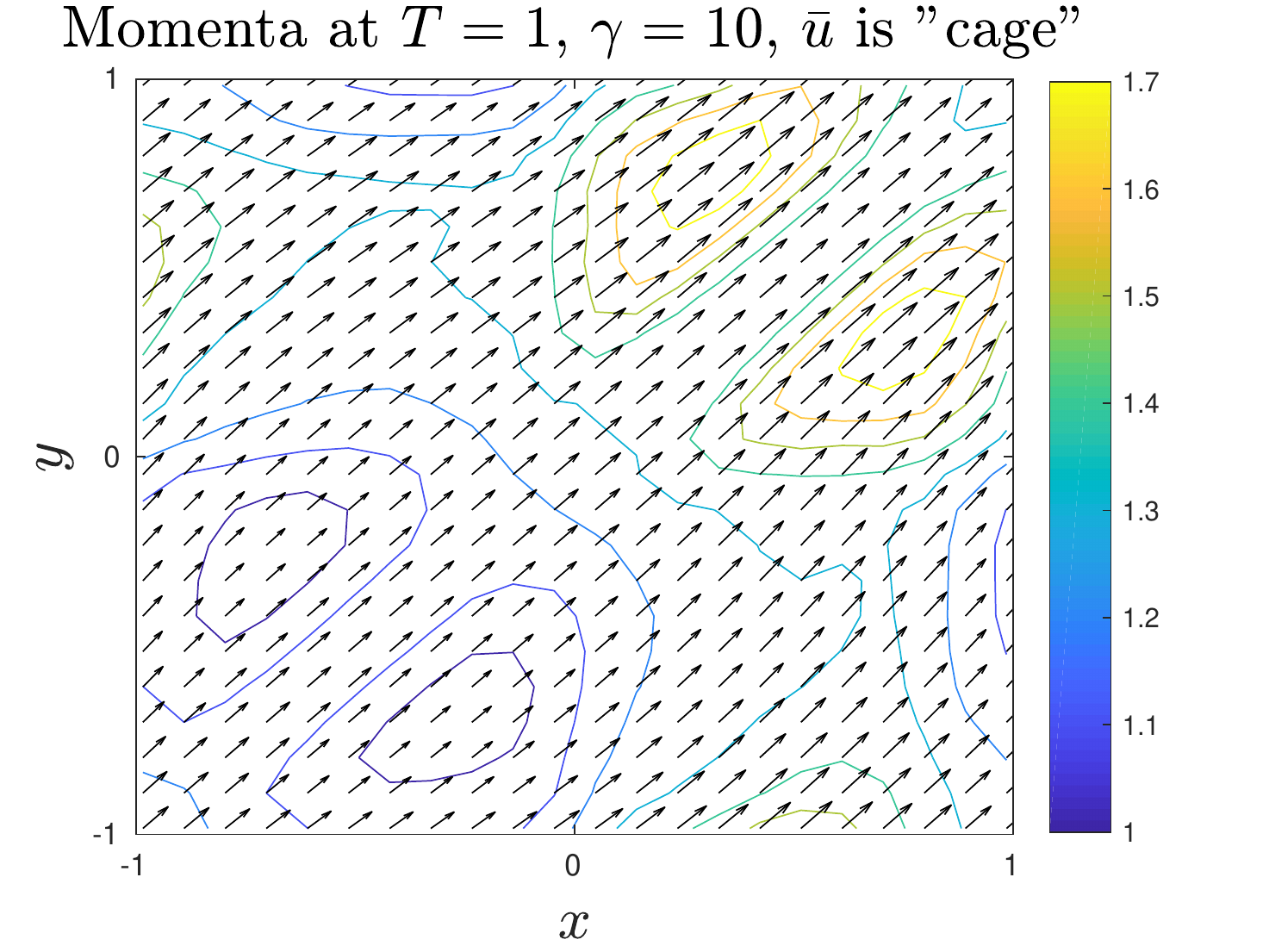}\hfill
		\includegraphics[width=0.34\textwidth]{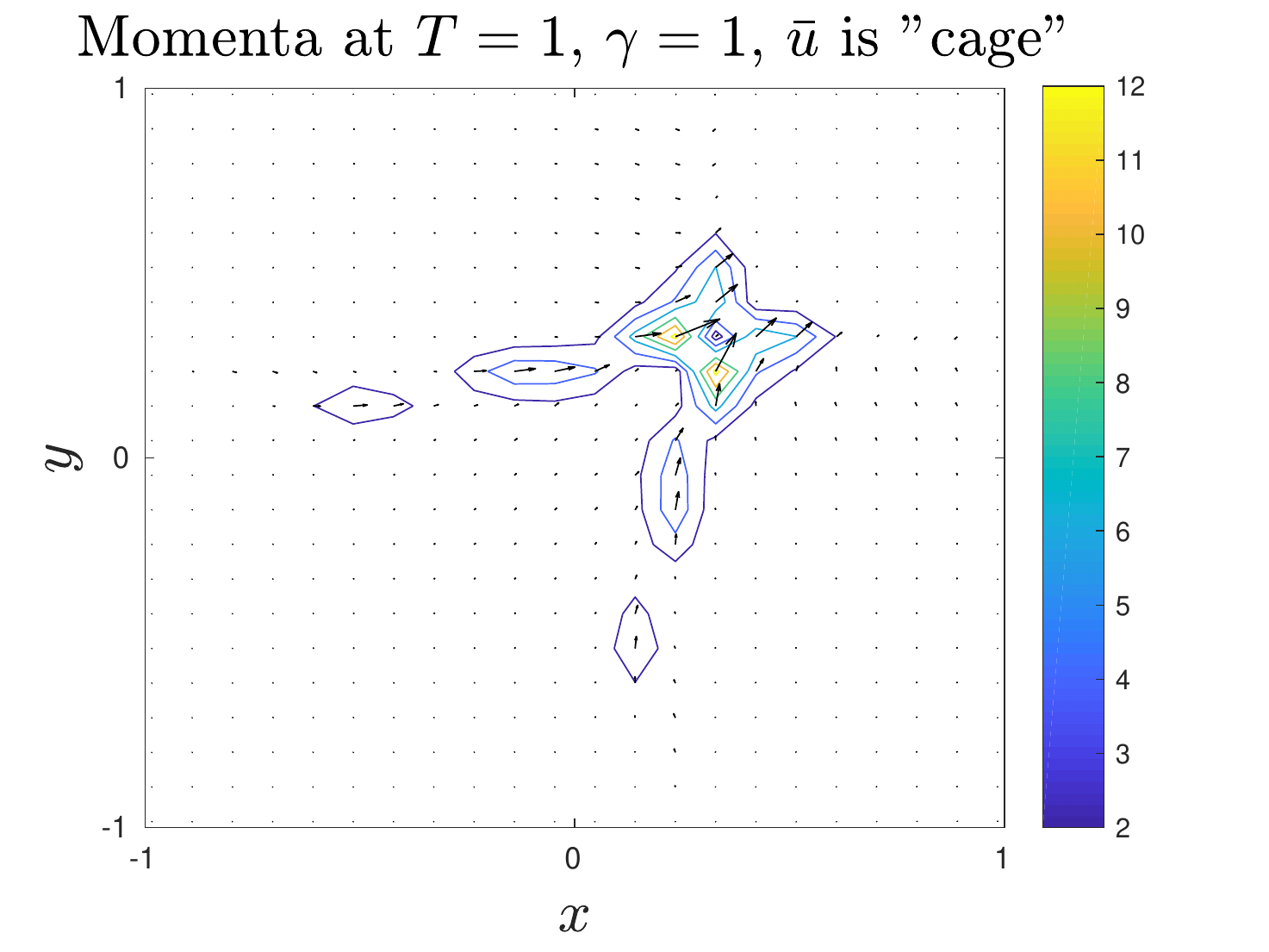}\hfill
		\includegraphics[width=0.34\textwidth]{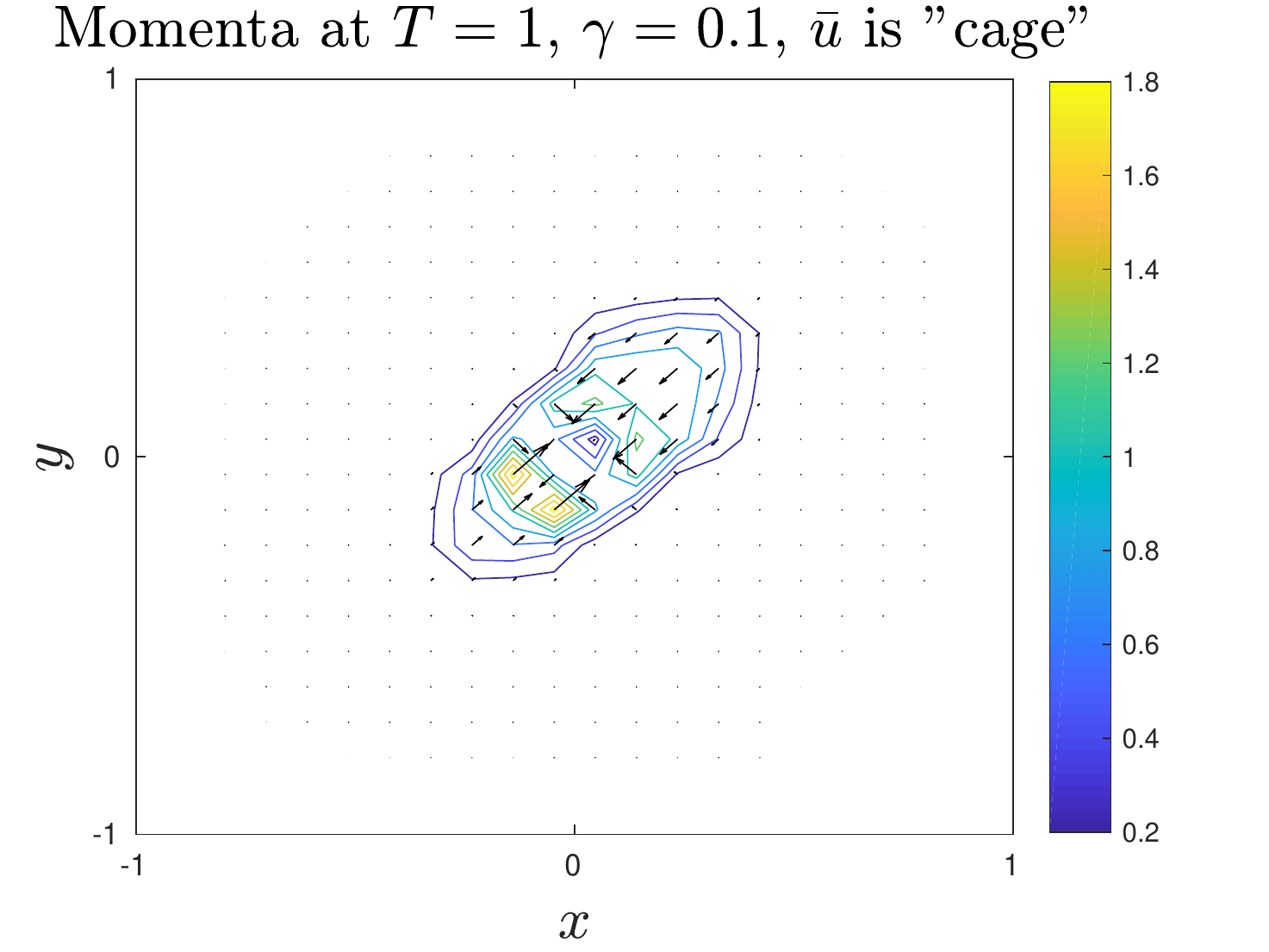}
		\\
		\hline
	\end{tabular}
	\caption{(Test 2D: Birdcage). The initial velocity with the control is given in the first row. Followed by the evolution for different $\gamma\in\{10,1,0.1\}$ of the density in the second row and the momenta in the last row.} \label{fig:8}
\end{figure}

\paragraph{Test 2D: Scarecrow.}
In this case we want to model a flock that is suddenly repelled from a center position. To model the sudden commotion in the flock, we set uniform density \textcolor{black}{$\rho_0\equiv1$} overall domain, and an initial velocity $u_0$ in milling state as follows,
$$u_0(x,y) =  (\cos(\varphi),-\sin(\varphi))^T , $$
where $\varphi=\text{arg}(x)$ is the {\it argument}, i.e. the angle of the associated polar coordinates of the vector $(x,y)$.

The ``scarecrow" action of the the control $\phi$ is modelled by using the following desired velocity
$$ \bar{u}(x,y) = \begin{cases} 2 \frac{(x,y)^T}{\|(x,y)\|_2}  & x^2+y^2\leq\frac{1}{10} ,\\ (0,0)^T & \text{elsewhere.} \end{cases}$$
\par
Figure \ref {fig:10} displays the initial data $u_0$, and the desired velocity $\bar u$ in the first row. Then, in the second and third rows are reported the evolution of the mass and momenta, respectively, for different choices of $\gamma.$ We observe that the control ensures the abstinence of mass at the center for lower $\gamma$ and for longer evolutions.\par
\begin{figure}[h!]
	\centering
	\begin{tabular}{@{}c@{\hspace{1mm}}c@{\hspace{1mm}}c@{\hspace{1mm}}c@{}}
		\hline
		\\
		\includegraphics[width=0.34\textwidth]{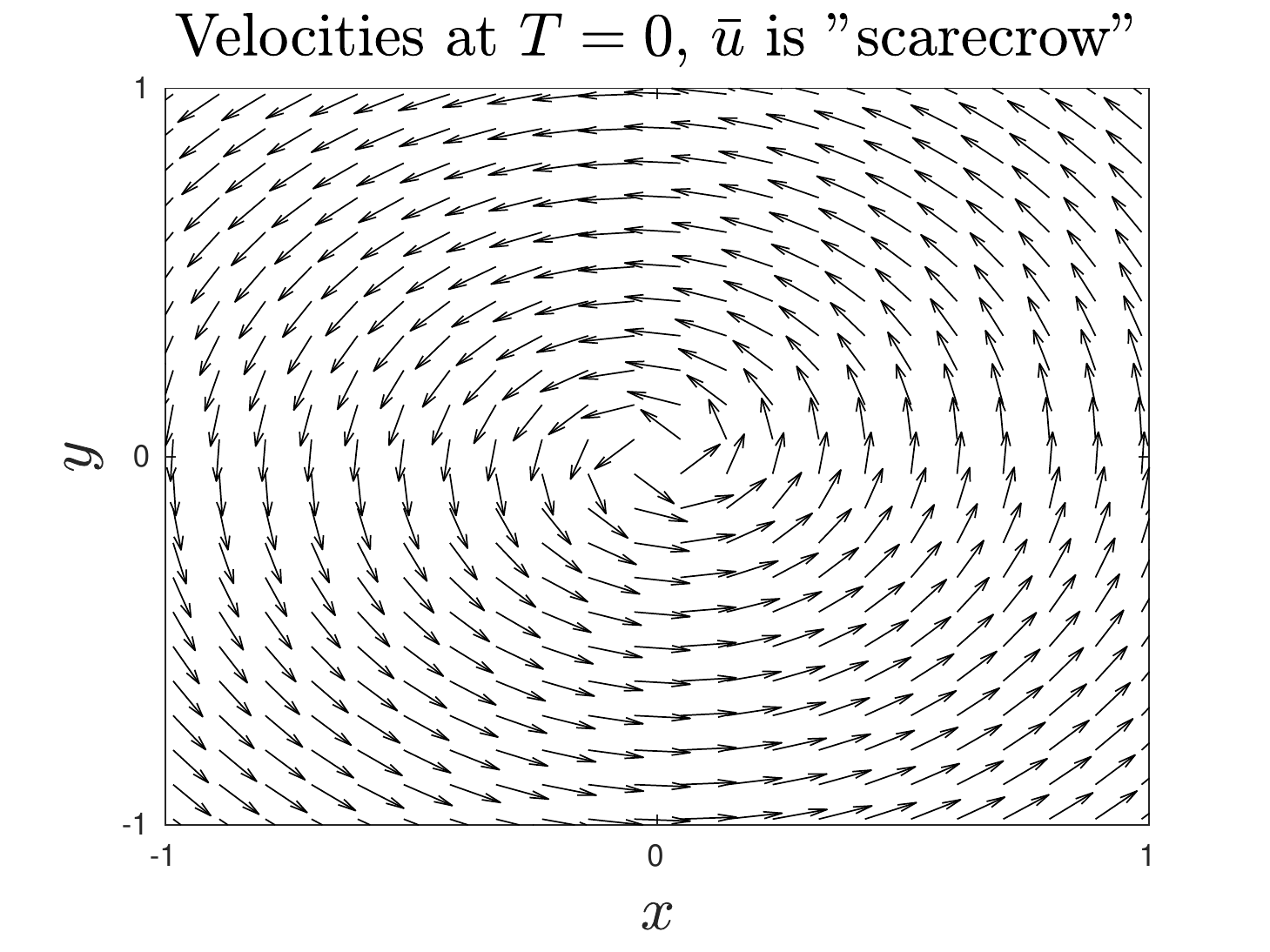}\hfill
		\includegraphics[width=0.34\textwidth]{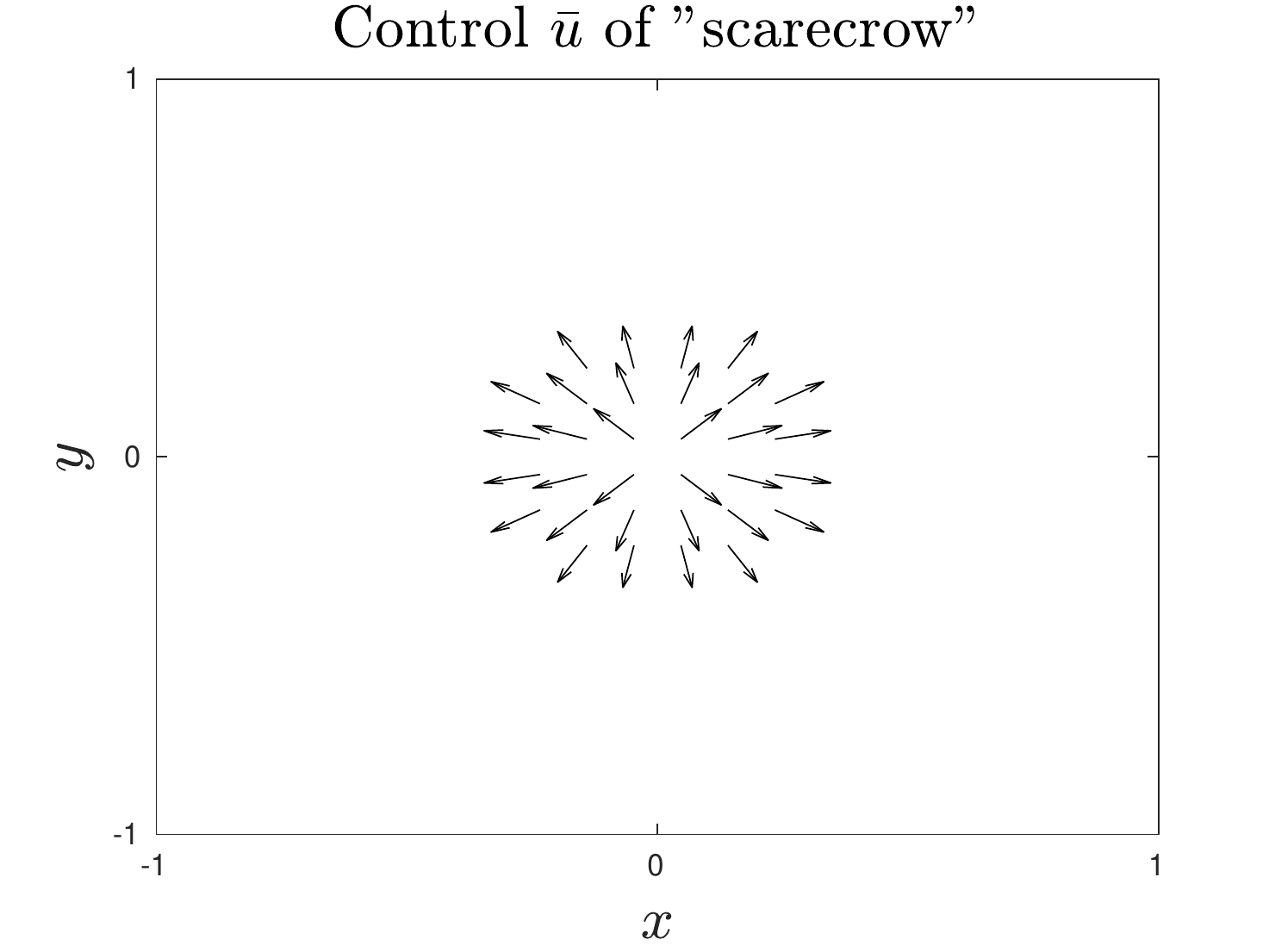}\\
		\hline\\   
		\hline\\
		\includegraphics[width=0.34\textwidth]{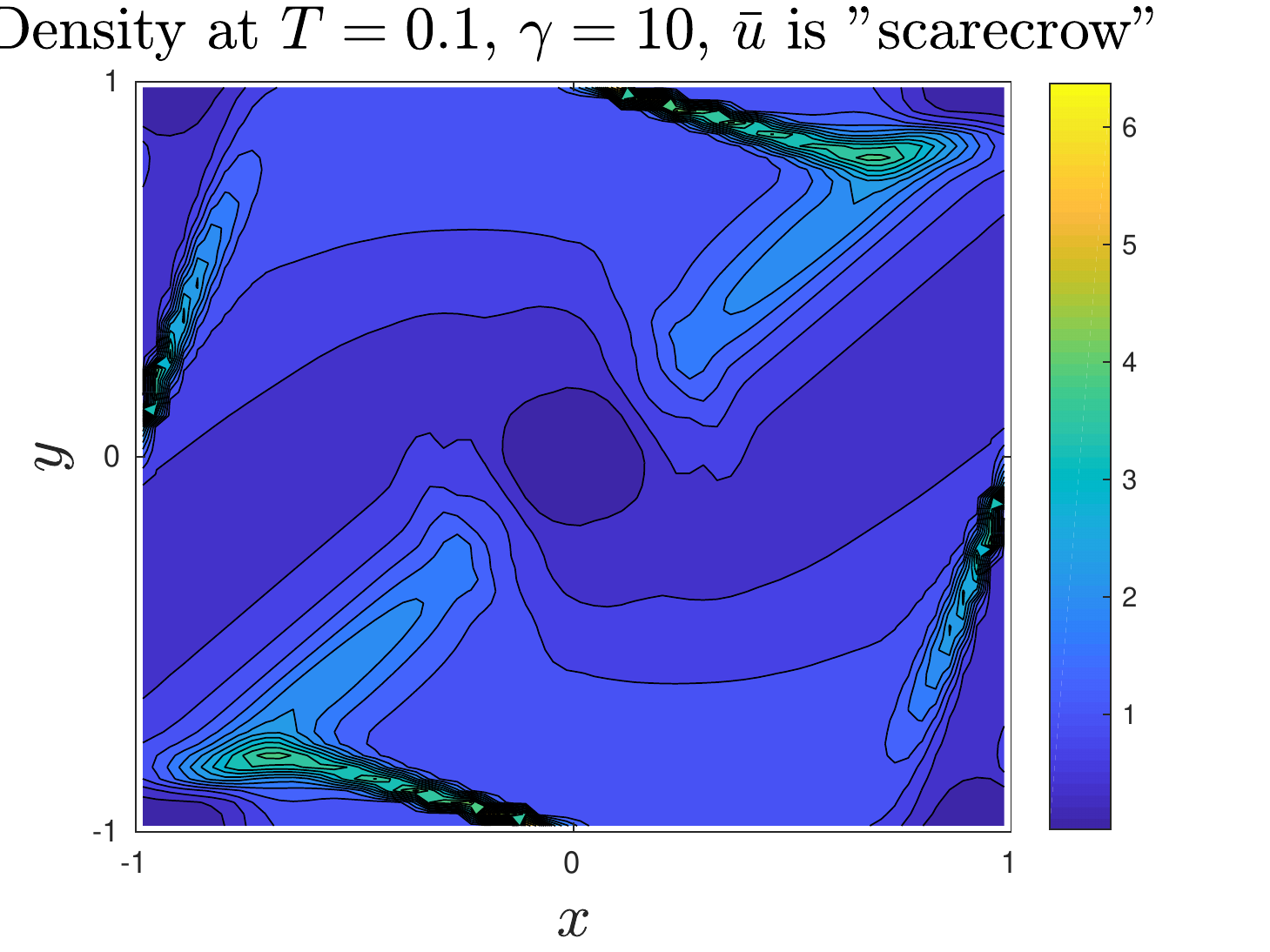}\hfill
		\includegraphics[width=0.34\textwidth]{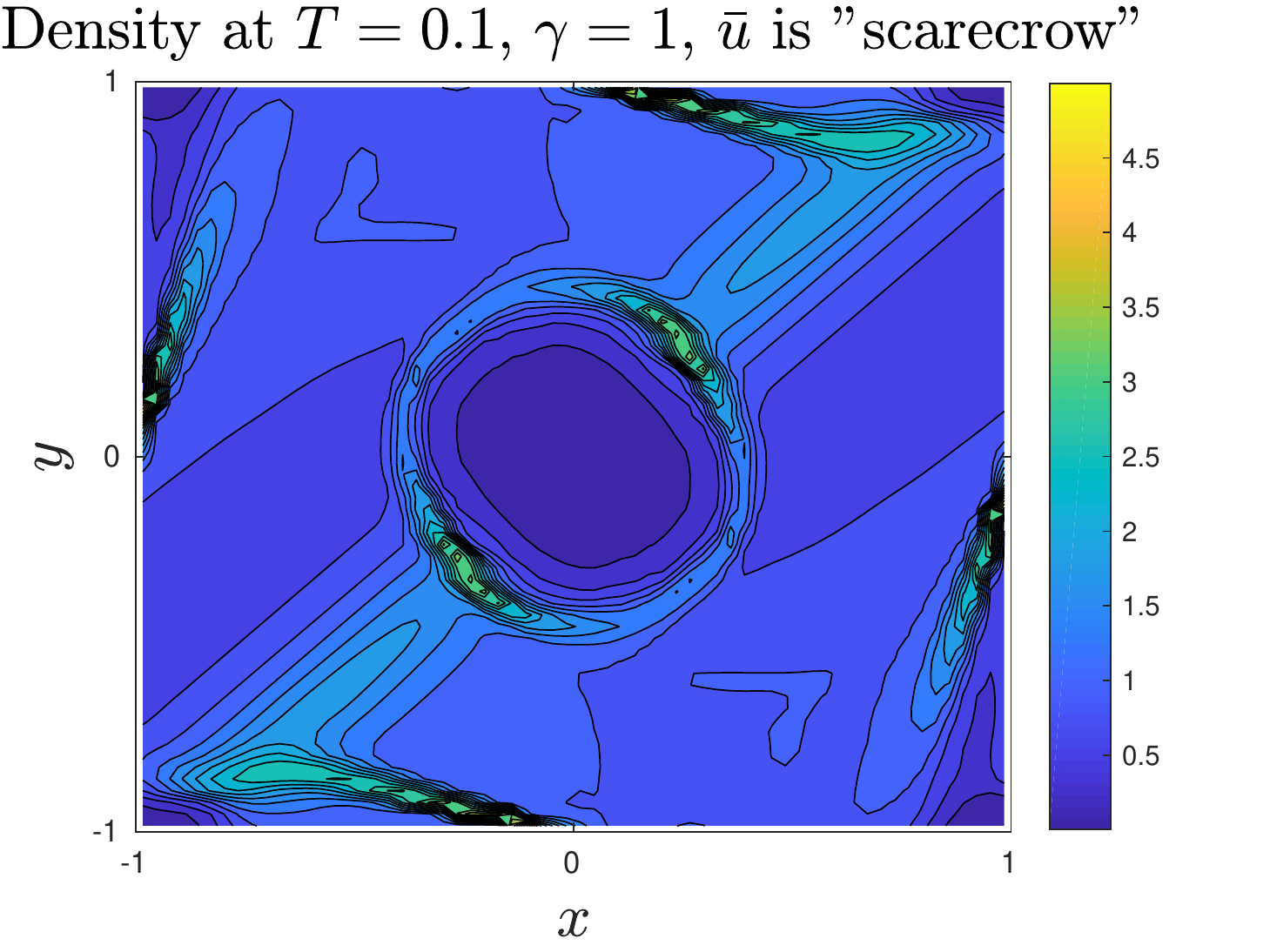}\hfill
		\includegraphics[width=0.34\textwidth]{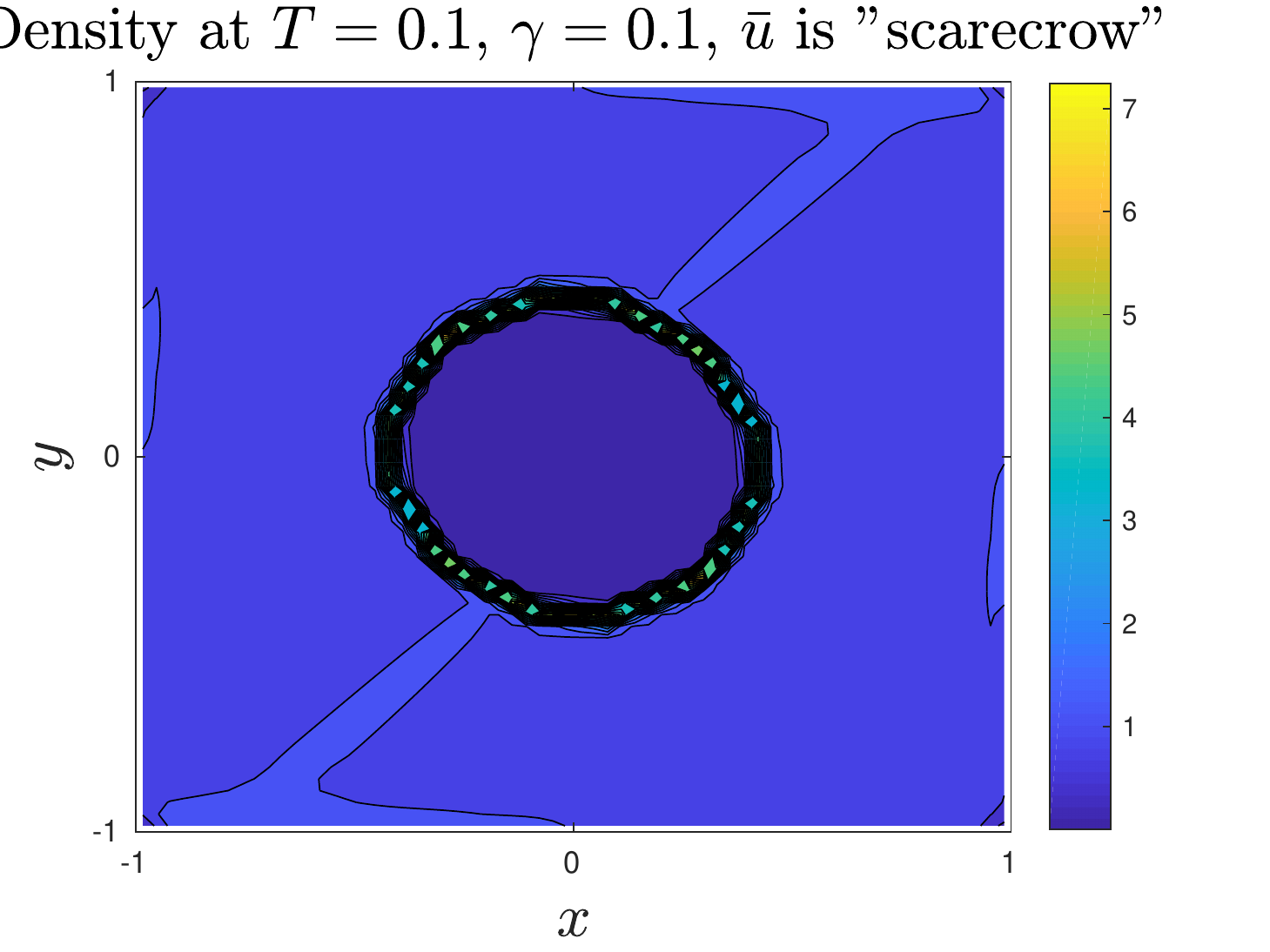}\\
		\hline\\
		\hline\\  
		\includegraphics[width=0.34\textwidth]{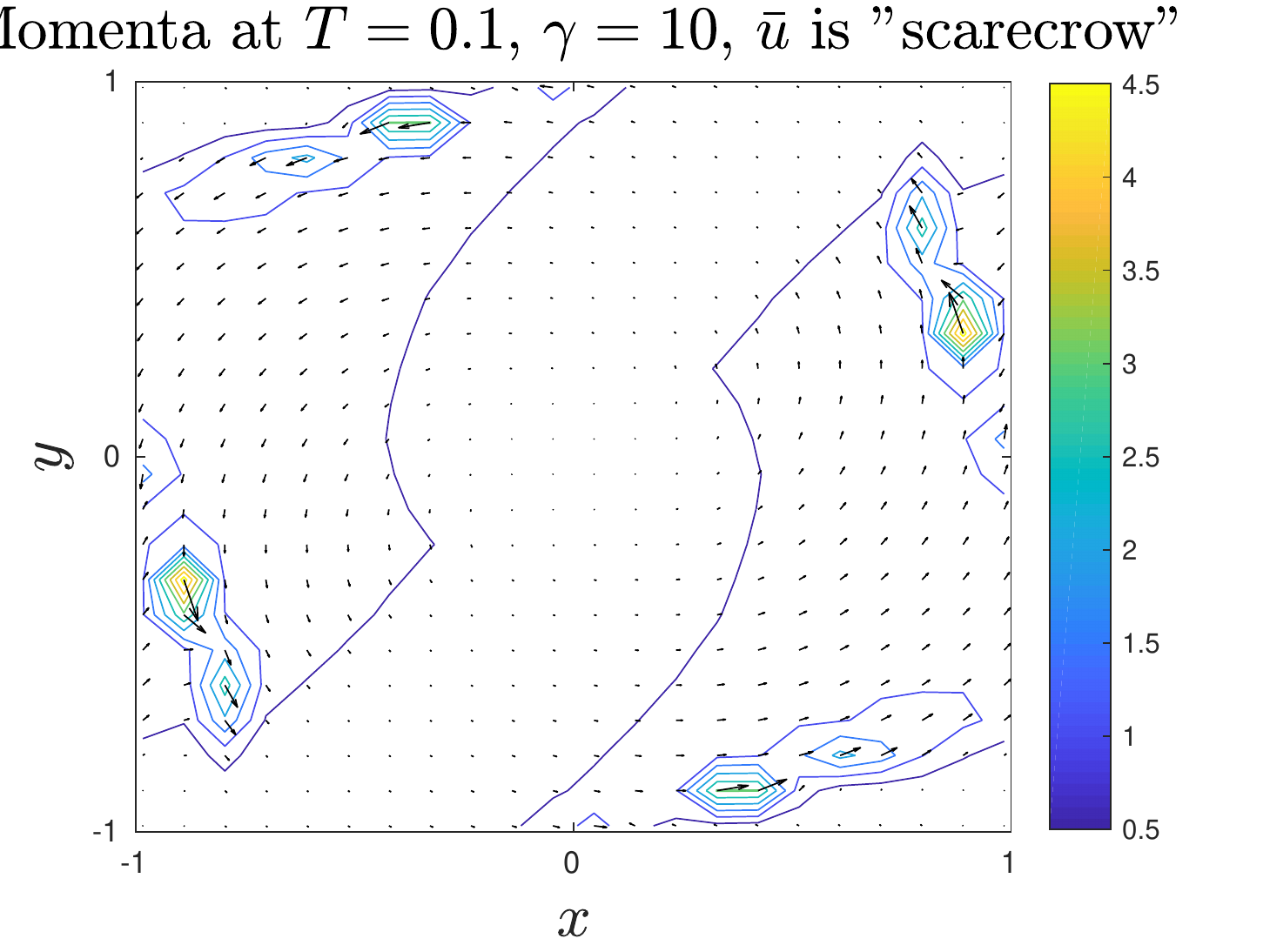}\hfill
		\includegraphics[width=0.34\textwidth]{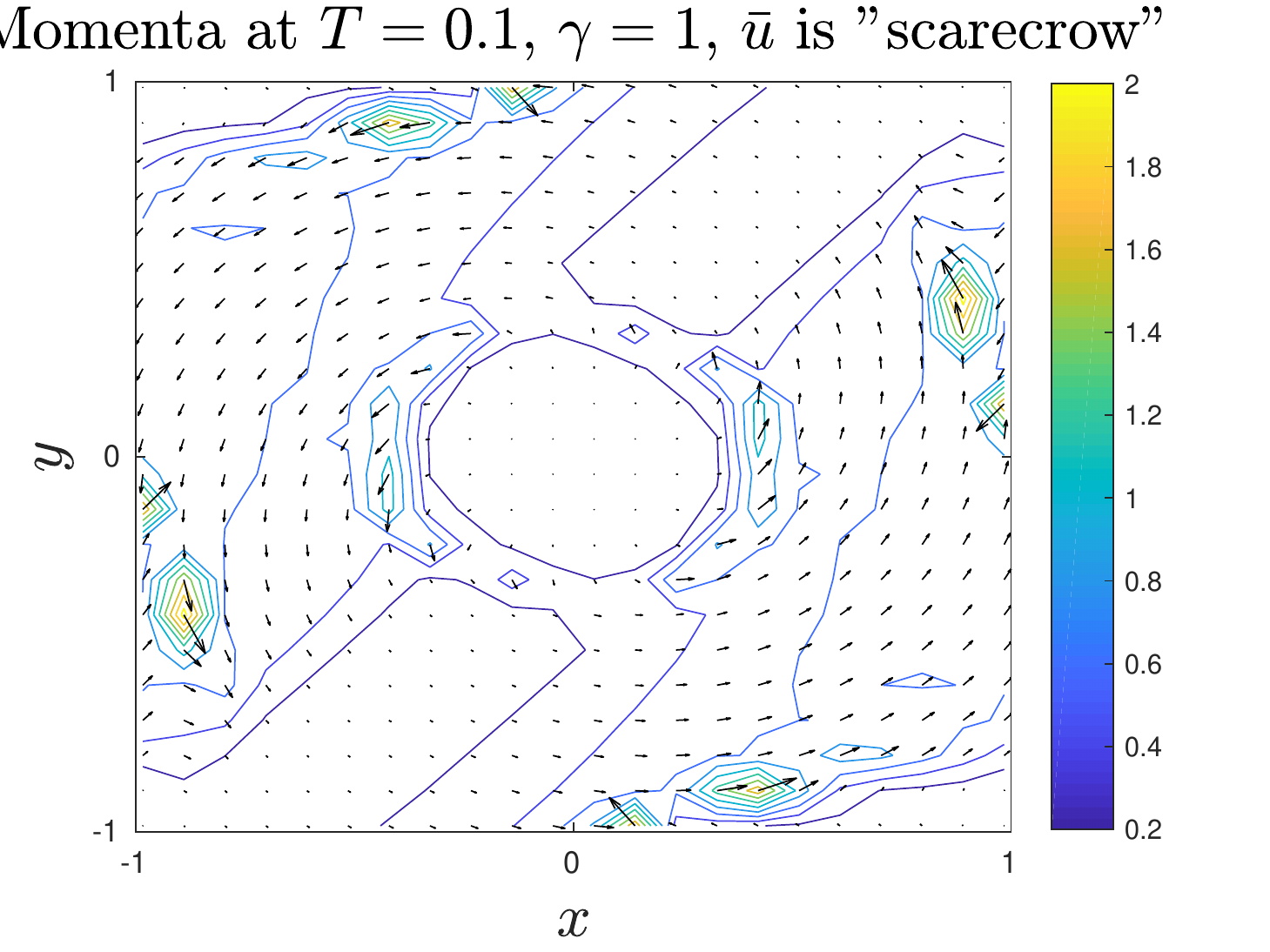}\hfill
		\includegraphics[width=0.34\textwidth]{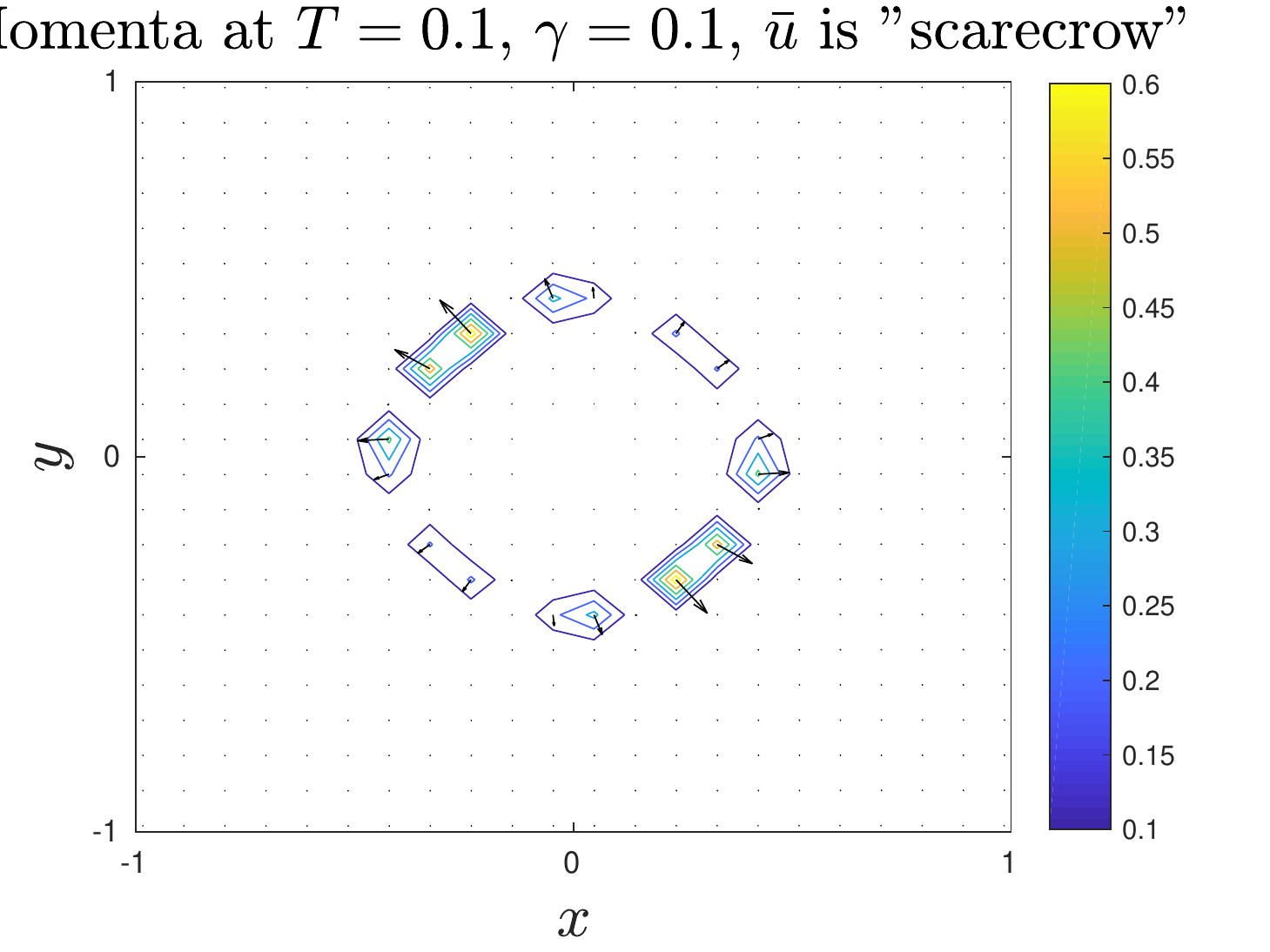}\\
		\hline
	\end{tabular}
	\caption{(Test 2D: Scarecrow). Initial velocity $u_0$, control $\bar{u}$ and states at $T=0.1$ under different values of $\gamma$ are displayed.} \label{fig:10}
\end{figure}

\paragraph{Test 2D: Blow-up phenomena.}
We want observe numerically 2D blow-up phenomena for the limit case of uncontrolled dynamics ($\gamma=\infty$), and compare it with the strong control case ($\gamma<1$). Hence, we set a new initial data as follows
\begin{subequations}\label{2dbuid}
\begin{align}
\rho_0(x,y)  =& \max\left\{\exp\left[ -2x^2 -2 y^2 \right] -\frac{1}{10},0\right\},\\
u_0(x,y) =& (2H(x)-1,2H(y)-1)^T,
\end{align}
\end{subequations}
where, in the control case, we consider homogeneous desired state $\bar u \equiv 0$.
Note that according to Theorem \ref{thm2dblowup}, the regularity of the solution can be controlled through the parameter $\gamma$. In particular we expect that, given the same initial data, blow-up does not occur when $\gamma$ is taken small enough. 
%
In Figure \ref{busol2d} we show the final density $\rho$ of \eqref{2dbuid} for the uncontrolled case at terminal time $T=0.65.$, and we compare it with respect to the controlled case with parameter $\gamma = 0.01$. We do not report the velocity field, since at time $T=0.65$ in both cases converges to (numerical) zero, therefore the densities $\rho$ have already reached their steady state.

The uncontrolled case, on the left-hand side, shows a blow up in finte time, where all mass is clustered in one cell.  Indeed, this is as far as a finite volume method goes to resolve a Dirac delta distribution. To be certain that we face such a case, we changed the resolution to $65\times 65$ cells, such that we have a single cell located on an neighborhood of $x = (0,0)^T$. It turns out that this centered cell actually does assume the entire mass of the system. Hence, the support of the density distribution $\text{supp}\{\rho\}$ at $T=0.65$ is entirely contained in this cell, and as desired in case of a blow up. (See the evolution of $V(t)$ in Section \ref{2dtheo}, and Remark \ref{FVR} for different resolutions.) 
On the right-hand side, we represents the controlled case with $\gamma = 0.01$, where we can easily observe that the mass does not concentrate that intensely at the center, as well as that the support of the density is smeared out around zero.

\begin{figure}[t]
	\centering
	\begin{tabular}{@{}c@{\hspace{1mm}}c@{\hspace{1mm}}c@{\hspace{1mm}}c@{}}
		\includegraphics[width=0.45\textwidth]{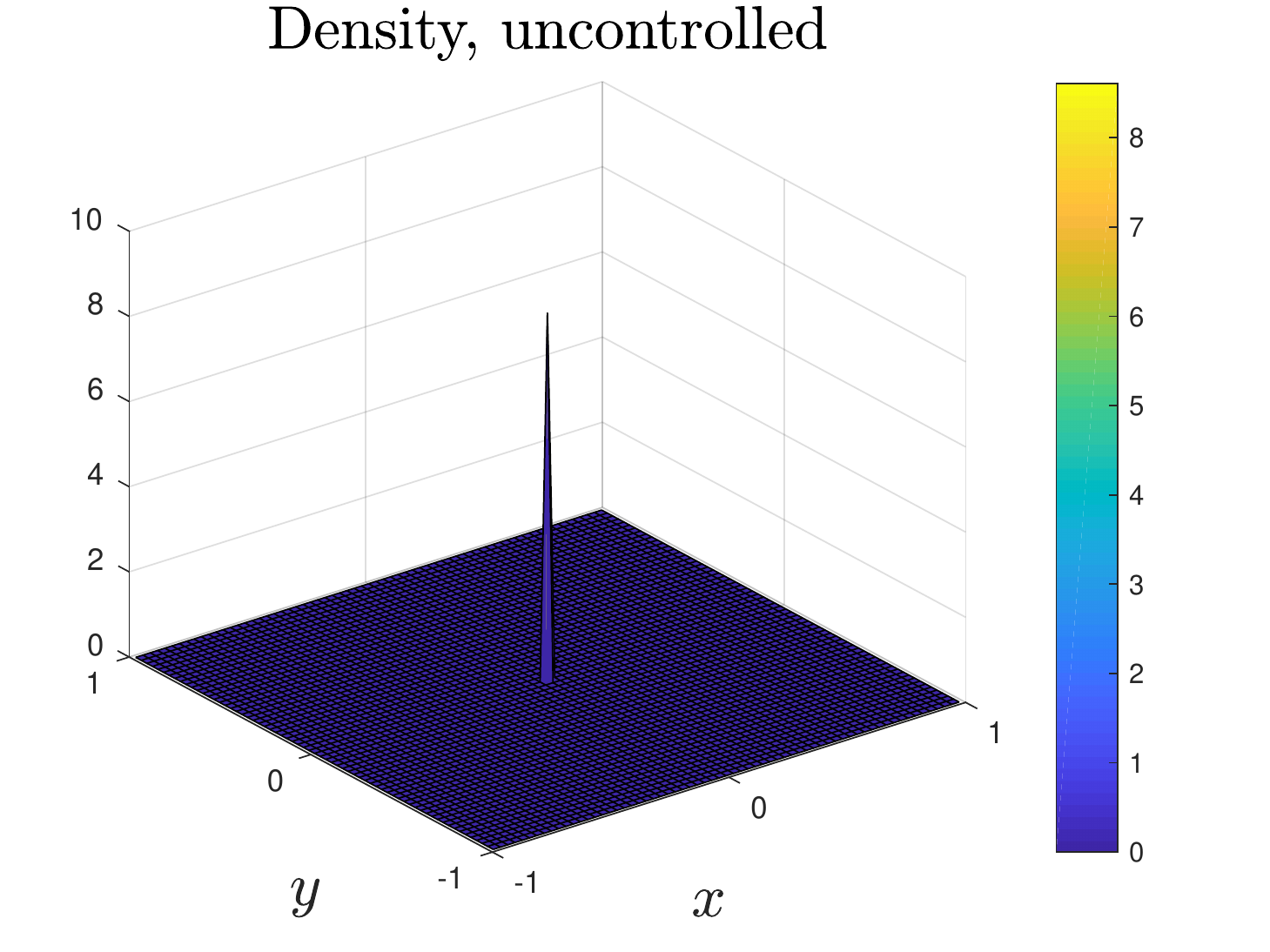}
		&
		\includegraphics[width=0.45\textwidth]{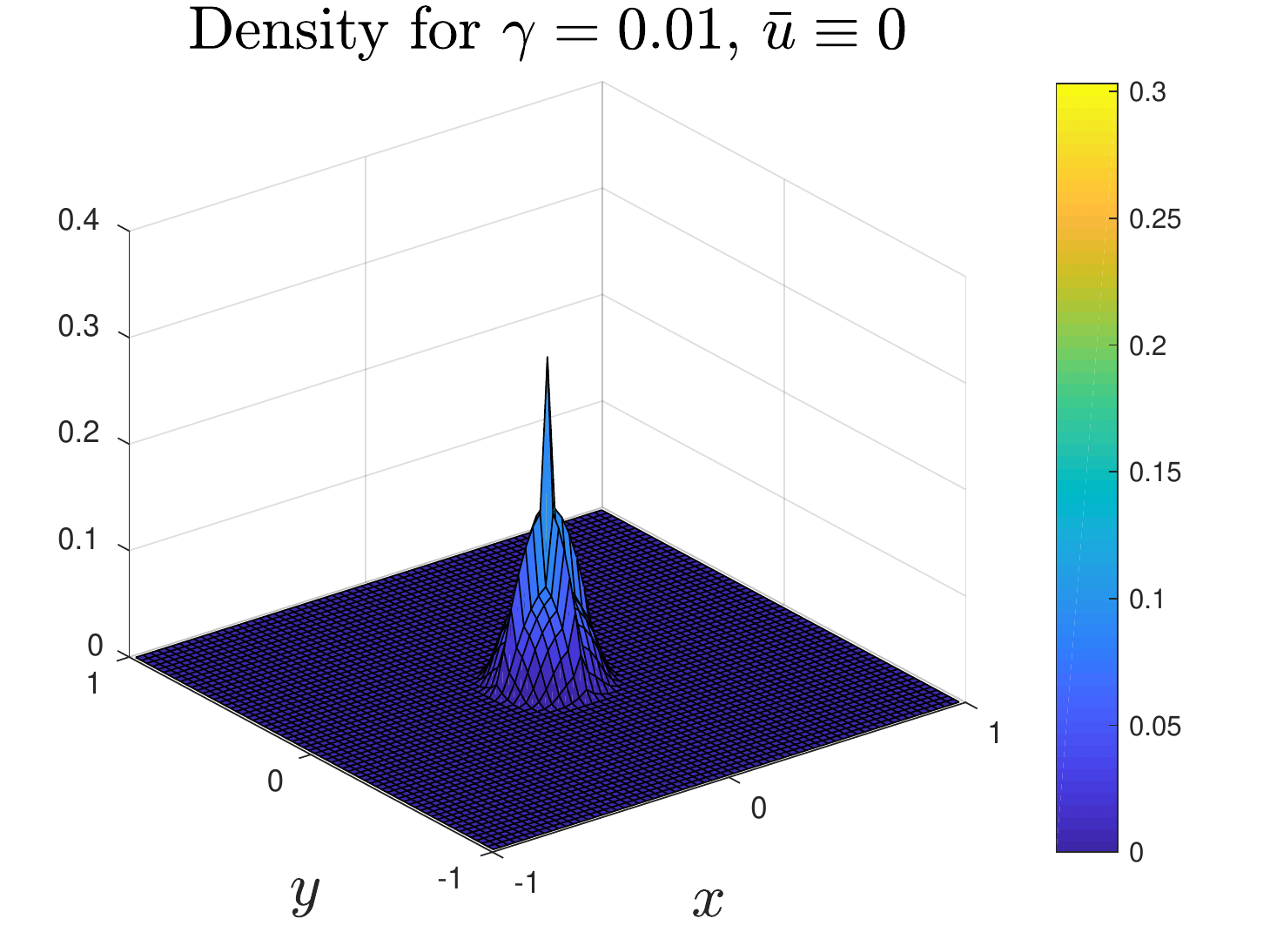}
	\end{tabular}
	\caption{(Test 2D : Blow up phenomena, uncontrolled vs controlled). On the left-hand side, the $\delta$-distribution (in the density) generated in finite time for the uncontrolled case from the initial data \eqref{2dbuid}. On the right-hand side, for comparison we report the controlled case with $\gamma=0.01$ and $\bar u=0$, in this case the support of the density is not entarely concentrated in a single cell.}\label{busol2d}
\end{figure}

\begin{remark}[Finite volume resolution]\label{FVR}
In Figure \ref{deltaRes} we report rough and finer grid representations for the uncontrolled case of Figure \ref{busol2d}. We observe, that the value of the concentration at the central cell depends on the chosen discretization. For a given Dirac delta at a point $x_0$ we have that this (here for $\rho$) mass is located in the central cell of a certain volume $\|\Omega\|$. If we resolve the domain by a higher resolution than the corresponding volume of the central cell will be smaller. The Dirac delta is still located in this central cell and still shows the same value. Hence, the average value of the central cell will be higher for finer resolutions --- as verified in Figure \ref{deltaRes} for the resolutions $45 \times 45$ and $85 \times 85$ cells. 
\begin{figure}[t]
	\centering
	\begin{tabular}{@{}c@{\hspace{1mm}}c@{\hspace{1mm}}c@{\hspace{1mm}}c@{}}
		\includegraphics[width=0.49\textwidth]{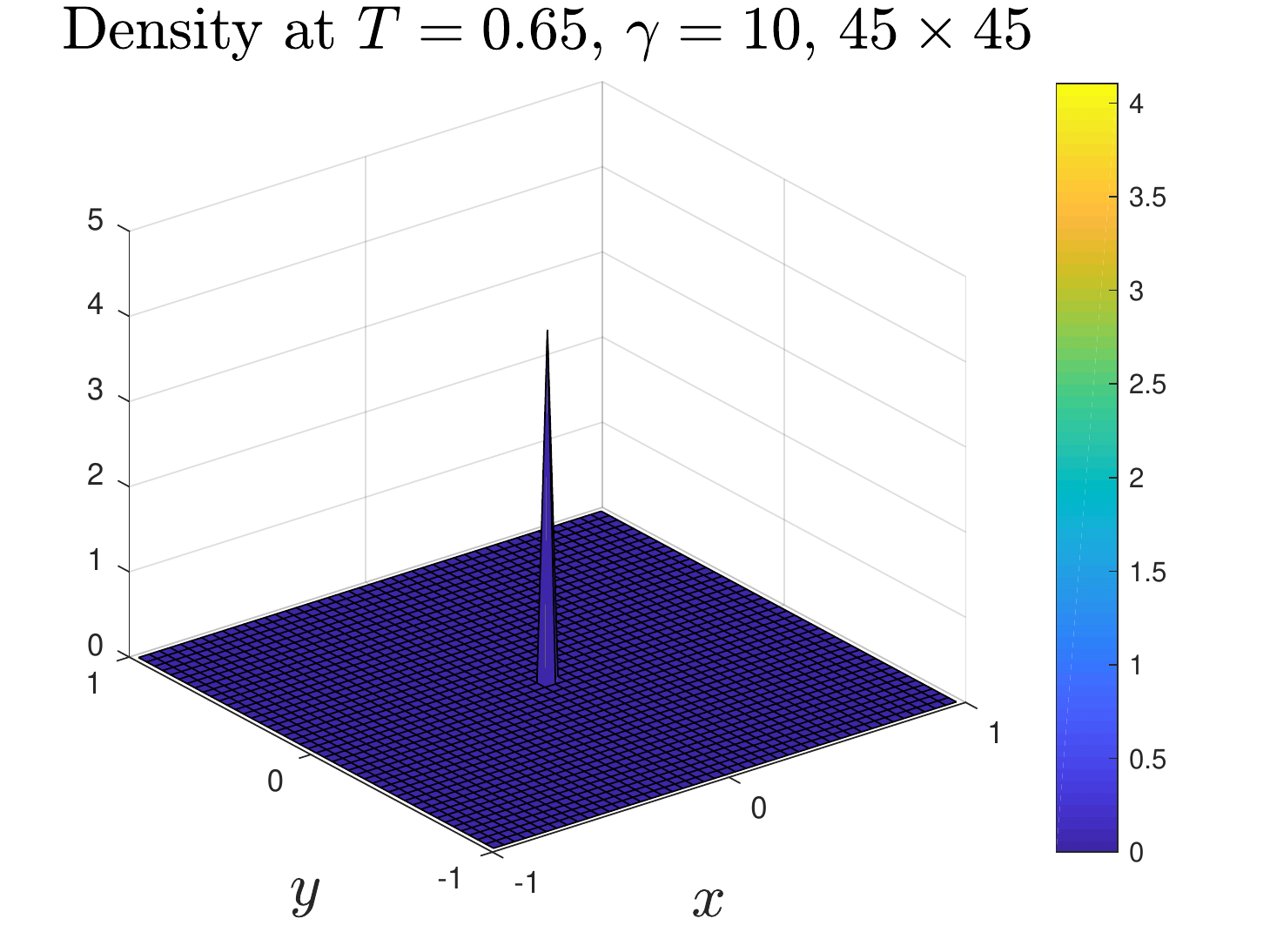}
		&
	\includegraphics[width=0.49\textwidth]{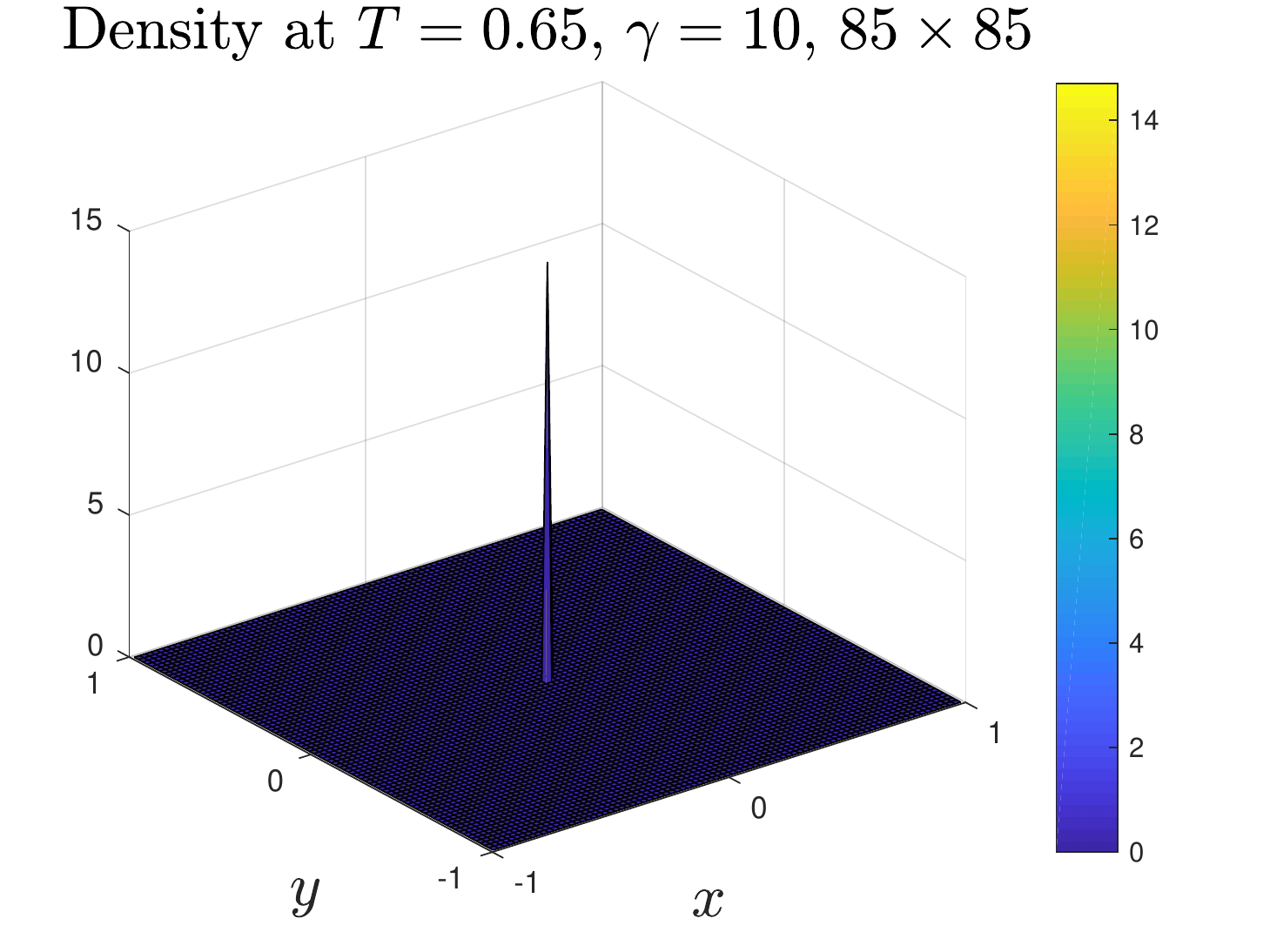}
	\end{tabular}5
	\caption{Blow-up of the uncontrolled dynamics for different resolutions of the finite volume scheme. On the left-hand side grid of $45\times45$ cells, on the right-hand side finer grid with $85\times85$ cells. We observe, that the value of the accumulated mass in the central cell increases for a finer discretization.}\label{deltaRes}
\end{figure}
\end{remark}
\section{Conclusions}
We have studied a pressureless Euler type system for the macroscopic description of particle alignment in presence of control. For that hydrodynamic model, we provide critical thresholds results for the global regularity and finite-time breakdown of solutions, both in one- and two-dimensional settings. Finally, we validate our modelling and analytical results with several numerical experiments, employing a finite volume scheme for the pressureless Euler alignment model. Further perspectives are the investigation of different non-local interaction operators, such as attraction or repulsion forces, whereas also  singularity are included. At the numerical level, challenging problems arising in this study are the solution of the full optimal control problem \eqref{main-eq}--\eqref{eq:MaOC}, and the development of high-order schemes.

\section*{Acknowledgments} 
G.A. acknowledges the support by GNCS--INDAM fundings. 
Y.-P. C. is supported by National Research Foundation of Korea(NRF) Grant funded by the Korea government (MSIP) (Nos. 2017R1C1B2012918 and 2017R1A4A1014735) and
POSCO Science Fellowship of POSCO TJ Park Foundation. 
A.-S. H. acknowledges the support by DAAD-MIUR fundings.

\end{document}